\documentclass[a4paper,10pt]{article}
\usepackage{hyperref}
\usepackage{amsfonts}
\usepackage{amssymb}
\usepackage{amsthm}
\usepackage{amsmath}
\usepackage[shortcuts]{extdash}
\usepackage{color}
\usepackage[matrix,arrow,curve]{xy}
\oddsidemargin=-\evensidemargin
\numberwithin{equation}{section}
\newtheorem{theorem}[equation]{Theorem}
\newtheorem{lemma}[equation]{Lemma}
\newtheorem{remark}[equation]{Remark}
\newtheorem{corollary}[equation]{Corollary}
\newtheorem{proposition}[equation]{Proposition}

\theoremstyle{definition}
\newtheorem{definition}[equation]{Definition}
\newtheorem{example}[equation]{Example}

\theoremstyle{remark}
\newtheorem{sideremark}[equation]{Side-remark}

\newcommand{\QQ}{\mathbb Q}
\newcommand{\CC}{\mathbb C}
\newcommand{\ZZ}{\mathbb Z}
\newcommand{\NN}{\mathbb N}

\newcommand{\rk}{\mathop{\mathrm{rk}}}

\newcommand{\supp}{\mathop{\mathrm{supp}}}

\newcommand{\codim}{\mathop{\mathrm{codim}}\nolimits}

\newcommand{\pt}{\mathrm{pt}}

\newcommand{\CH}{\mathop{\mathrm{CH}}}

\newcommand{\kpar}{}

\newcommand{\kcomment}[1]{#1}
\newcommand{\kdel}[1]{}

\newcommand{\weylmax}{w_0}

\begin{document}
\title{Multiplicity-free products of Schubert divisors}
\author{Rostislav Devyatov\thanks{Laboratory of Algebraic Geometry and its Applications,
Department of Mathematics,
National Research University Higher School of Economics,
6 Usacheva str.,
Moscow 119048,
Russian Federation.\newline\textit{Email address:} \texttt{deviatov@mccme.ru} .}}
\maketitle

\begin{abstract}
Let $G/B$ be a flag variety over an arbitrary field, 
where $G$ is a semisimple split algebraic group with a simply laced Dynkin diagram, and $B$ is a Borel subgroup.
We say that the product of several classes of
Schubert divisors in the Chow ring
is \emph{multiplicity-free}
if it is possible to multiply it by a Schubert class (not necessarily of a divisor)
and get the class of a point. 
In the present paper we find 
all possible degrees (in the Chow ring) of 
multiplicity-free products of classes of Schubert divisors.

Also, given a product of several classes of Schubert divisors, we can decompose it into 
a linear combination of classes of Schubert varieties with (as was known before) nonnegative coefficients.
We study the coefficients in this linear combination and provide a criterion detecting 
if such a coefficient equals 1, is greater than 1, or equals zero (i.e. a Schubert variety is not actually present in the linear combination).
\end{abstract}

\section{Introduction}
Let 
$G$ be a semisimple split 
algebraic group over an arbitrary field.
Consider the \emph{generalized flag variety} $G/B$, 
where $B\subseteq G$ is a Borel subgroup.
We are going to study the Chow ring of 
$G/B$ (or, equivalently, the cohomology group of $G/B$ if the base field is $\CC$) and its multiplicative structure.
The multiplicative structure of $\CH^*(G/B)$ was studied, for example, in \cite{chevalley} and \cite{demazure}.
As a $\ZZ$-algebra, $\CH^*(G/B)$ is ``almost generated'' by the classes of \emph{Schubert divisors} (defined precisely below) in $G/B$.
Denote the classes of Schubert divisors by $D_1,\ldots, D_r$, where $r=\rk G$.
To be precise, ``almost generated'' above means that every element of $\CH^*(G/B)$ becomes a polynomial in $D_i$s after multiplication
by an appropriate integer\footnote{By \cite[\S 4, Theorem 2]{grothendiecktorsion}, this integer equals the torsion index of the simply connected cover of $G$. 
The last previously unknown torsion indices of simply connected split simple algebraic groups were found
in \cite{TotaroE} and \cite[Theorem 0.1]{TotaroSpin}.
A survey on the prior results on torsion indices is also present in these two papers.}, 
or, in other words, $\CH^*(G/B) \otimes_{\ZZ} \QQ$ is generated by $D_i$s as a $\QQ$-algebra.
The ideal of relations between $D_i$s was computed in 
\cite{borel} (see also \cite[Propostion 1.3]{bgg}).

We will be particularly interested in monomials in $D_i$s.
Let us say that a monomial $D_1^{r_1}\ldots D_r^{n_r}$ is \emph{multiplicity free}
if there exists 
$X \in \CH^*(G/B)$
such that 
$D_1^{r_1}\ldots D_r^{n_r}X=[\pt]$.
Our goal is, for all groups $G$ \emph{with simply-laced Dynkin diagram} (i.e. of type $A$, $D$, or $E$ or a product of such groups), 
to 
find
\textit{all possible degrees (in the Chow ring) of multiplicity-free monomials in $D_1,\ldots,D_r$}.
In other words, 
we find
\emph{all possible values of the sum $n_1+\ldots+n_r$ for the $n$-tuples $n_1,\ldots,n_r$
of nonnegative integers such that $D_1^{n_1}\ldots D_r^{n_r}$ is a multiplicity-free monomial}.
This question is particularly interesting in the case when $G$ is of type $E$, 
because the answer may be used to compute upper bounds on the 
\emph{canonical dimension} (for the definition, see \cite{merkurjevupd3}) of 
the simply connected covering of $G$ (or of $G$ itself if it was already simply connected)
over non-algebraically closed fields, 
similarly to the arguments of 
\cite{karpenko}.

These possible degrees
for types $E_6$, $E_7$, and $E_8$ 
turn out to be all integers between 0 and
(respectively) 19, 26, 34, see Theorem \ref{answertypee}. 
Also, for a group of type $A_r$, i.e. for the ``classical'' variety of complete flags, 
we will get all integers between 0 and $r(r+1)/2 = \dim (G/B)$, see Lemma \ref{answertypea}.
This agrees with the fact that the torsion index of $SL_{r+1}$ is 1.
And for a group of type $D_r$ ($r\ge 4$), we will get all integers between 0 and $r(r+1)/2-1$, see Proposition \ref{answertyped}, 

To give some more motivation to this question and to 
explain how we are going to solve it, let us introduce some more notation and terminology.
First, let us fix a (split) maximal torus in $B$. We will not need the torus explicitly, but it further canonically defines 
the Weyl group, which we denote by $W$.
Then, in $G/B$, one can associate a so-called \emph{Schubert subvariety} to any 
$w\in W$.
There are many different ways to establish such a correspondence, we choose the following one:
for each $w\in W$, denote $Z^w=[\overline{B\weylmax w^{-1}B/B}]$, where $\weylmax$ is the longest element of the Weyl group.
Then $\codim Z^w=\ell(w)$, 
where $\ell(w)$ is the length of an element $w\in W$. 
In other words, $Z^w \in \CH^{\ell(w)}(G/B)$.
It is known that the classes of $Z^w$ in Chow ring for all $w\in W$ 
freely generate $\CH^*(G/B)$ as an abelian group.
The highest possible, 
the $\dim(G/B)$th, degree 
of the Chow group is $\ZZ$-generated by $Z^{\weylmax}=[\pt]$.
Also, we can say precisely now that 
$D_i = Z^{\sigma_i}$, where $\sigma_i \in W$ is the $i$th simple reflection
(the simple roots are enumerated as in \cite{bou}).

So, we have two ways to generate or ``almost generate'' $\CH^*(G/B)$: with Schubert divisors or with all Schubert varieties. 
The advantage of all Schubert varieties is that they really generate $\CH^*(G/B)$ without any division (and without any relations).
The disadvantage is that so far we only see the additive structure this way, and the description of 
multiplication in terms of $Z^w$s is quite complicated, not fully known, and probably hard to explain, because there are quite many elements in $W$.
Anyway, it would be nice to understand connections between these two descriptions.
Since $Z^w$s freely generate $\CH^*(G/B)$ as an abelian group,
every monomial in $D_i$s equals a linear combination of 
$Z^w$s:
$$
D_1^{n_1}D_2^{n_2}\dots D_r^{n_r}=\sum c_{w,n_1,\ldots,n_r}Z^w.
$$
It would be nice to understand the coefficients in this linear combination.

It is known (see, for example, \cite[Proposition 1.3.6]{brionarxiv}) that all 
coefficients $c_{w,n_1,\ldots,n_r}$ are nonnegative. 
So, it makes sense to ask: given $w, n_1, \ldots, n_r$, do we have
$c_{w,n_1,\ldots,n_r}=0$, $c_{w,n_1,\ldots,n_r}=1$, or $c_{w,n_1,\ldots,n_r} > 1$?
As we will see in the next paragraph, this question is directly related to the 
multiplicity-free monomials in Schubert divisors introduced above (and provides further motivation to study them).

To explain this connection, we need a theorem about multiplication of Schubert varieties of complimentary 
dimensions from \cite{chevalley}.
%
Using the notation $Z^w$ we introduced above, this result can be formulated as follows.
\begin{theorem}[discussion after {\cite[Proposition 9]{chevalley}}]
Let $w,w'\in W$ be such that $\ell(w)+ \ell(w')=\ell(\weylmax)$. 
Then $Z^wZ^{w'}=[\pt]$ if and only if $w=w'\weylmax$, and $Z^wZ^{w'}=0$ otherwise\footnote{In a bit more details, 
in \cite{chevalley}, Schubert varieties are indexed with the Weyl group differently
(see \cite[Introduction]{chevalley}): $X(u)=[\overline{BuB/B}]$ for $u \in W$. 
This notation is related to related our notation $Z^w$ as follows: 
$Z^w=X(\weylmax w^{-1})$, $X(u)=Z^{u^{-1}\weylmax}$.}.
\end{theorem}
\noindent Now,
it follows directly from this theorem that 
a monomial $D_1^{n_1} \ldots D_r^{n_r}$ is multiplicity-free (resp.\ $D_1^{n_1} \ldots D_r^{n_r} Z^w = [\pt]$)
if and only if there exists $w'\in W$ such that $c_{w',n_1,\ldots,n_r}=1$
(resp.\ $c_{w \weylmax,n_1,\ldots,n_r}=1$).

Now we are ready to say what we are actually going to do 
to find 
all possible degrees of 
multiplicity-free products of Schubert divisors.
For a group $G$ with simply-laced Dynkin diagram, given $w \in W$ and $n_1, \ldots, n_r \in \ZZ_{\ge 0}$, 
we are going to understand whether $c_{w, n_1, \ldots, n_r}=0$, $c_{w, n_1, \ldots, n_r}=1$, 
or $c_{w, n_1, \ldots, n_r}>1$ (see Proposition \ref{sortabilitycriterium} an Theorem \ref{MainTheorem}). 
In fact, the detection of whether $c_{w, n_1, \ldots, n_r}=0$ or $C_{w, n_1, \ldots, n_r}>0$ works even if the Dynkin diagram is not simply-laced
(Proposition \ref{sortabilitycriterium} does not depend on the type of $G$)
and may be of independent interest, although it follows quite easily from previously known results.
The criterion for $c_{w, n_1, \ldots, n_r}=1$ will then be easy enough to find 
the values of $\ell(w)$ for all tuples 
$w, n_1, \ldots, n_r$ such that $c_{w, n_1, \ldots, n_r}=1$.

These results were preliminarily announced in a short note \cite{multiplicityfreeshortnote} by the author.

It also seems natural to ask when, for given numbers $n_1,\ldots,n_r$, all
coefficients $c_{w,n_1,\ldots,n_r}$ for all $w\in W$ equal either 0 or 1.
But this happens quite rarely, and we are not trying to answer this question here. 
We will return to this question in a subsequent paper.

After (the first version of) this text was posted on arXiv, the examples of multiplicity-free monomials in Section \ref{sectionnumerical}
were generalized to groups of type $B$, $C$, $F$, and $G$ in \cite[Example 2.10 and Example 2.12]{zaynullinbcfg}.
Therefore, lower bounds on the maximal possible degree of a multiplicity-free monomial for groups of these types were obtained.
The question of the exact value of 
the maximal possible degree of a multiplicity-free monomial for groups of type $B$, $C$, $F$, and $G$ remains open.

\subsection*{Acknowledgments} I thank Kirill Zaynoulline for bringing my attention to this problem and 
Valentina Kiritchenko for numerous useful discussions.

\subsection*{Funding} The author thanks University of Ottawa, Max Planck Institute for Mathematics in Bonn, 
and the Department of Mathematical Sciences of KAIST (South Korea) for financial support and hospitality.
The study has been funded within the framework of the HSE University Basic Research Program (HSE-BR-2025-061).

\section{Notation and Preliminaries}
\subsection{Root systems}
We keep all notation and assumptions we have set in the Introduction.
Additionally, 
starting from Section \ref{sectionuniqsortnes}, we will always assume that the group $G$ has simply laced Dynkin diagram (i.e. group of type $A$, $D$, or $E$).
We denote
the root system
(resp.\ the subset of positive, negative, simple roots) by $\Phi$
(resp.\ $\Phi^+$, $\Phi^-$, $\Pi=\{\alpha_1, \ldots, \alpha_r\}$).
We enumerate the simple roots as in \cite{bou}.
We will also often talk about the coordinates with respect to the basis of simple roots, 
so we denote the function computing the $i$th coordinate by $\alpha_i^*$.
We denote the neutral element of $W$ by $1_W$.

Denote the reflection corresponding to a (not necessarily simple) root $\alpha \in \Phi$ by $\sigma_{\alpha}$.
(This means that a simple reflection is denoted by $\sigma_{\alpha_i}$, but we will use the notation $\sigma_i$ from the Introduction as well.)
The (invariant) scalar product of 
two roots $\alpha$ and $\beta$ is denoted by $(\alpha,\beta)$. 
Whenever
the Dynkin diagram is simply laced, all roots in $\Phi$ have equal lengths.
In this case, let us choose 
the bilinear form
on $\Phi$ so that the scalar square of each simple root is 2. 
Note that with this choice of 
the bilinear form,
we can use a simple formula for reflection.
Usually, we write
$$
\sigma_\alpha \beta = \beta-\frac{2(\beta, \alpha)}{(\alpha,\alpha)}\alpha.
$$
But with our choice of 
the bilinear form,
we can write
$$
\sigma_\alpha \beta = \beta-(\beta,\alpha)\alpha.
$$

We denote the Cartan bracket by $\langle \cdot, \cdot \rangle$: 
for two arbitrary vectors $v, w$ from the ambient space of $\Phi$, $w\ne 0$, we have
$$
\langle v,w \rangle
=
\frac{2(v,w)}{(w,w)}
$$
Also, denote the fundamental weight corresponding to a simple root $\alpha_i$ by $\varpi_i$.
If 
all roots in $\Phi$ have equal lengths, we also have $\alpha_i^*(\alpha) = \langle \varpi_i, \alpha \rangle$
for all $\alpha \in \Phi$, $\alpha_i \in \Pi$.

For each $\alpha \in \Phi$ (or even for any vector $\alpha$ from the root lattice), we say that its \emph{support} is 
$\supp \alpha = \{\alpha_i \in \Pi \mid \alpha_i^*(\alpha) \ne 0\} = \{ \alpha_i \in \Pi \mid \langle \varpi_i, \alpha \rangle \ne 0 \}$
(the last equality of two sets actually holds for all types of root systems, not necessarily $A$, $D$, or $E$).
At some point, we will also need to consider the subgroups of $W$ generated by reflections corresponding to some (but not all) roots.
For each subset $I \subseteq \Pi$, denote $\Phi_I = \{ \alpha \in \Phi \mid \supp \alpha \subseteq \Pi \}$.
Then, denote by $W_I$ the subgroup of $W$ generated by the reflections $\sigma_\alpha$ with $\alpha \in \Phi_I$
(equivalently, $W_I$ is generated by the \emph{simple} reflections corresponding to the simple roots from $I$, 
but we will not need this description).


We will also need to consider 
different partial orders on 
the root lattice.
First, there is the standard order $\prec$ on 
the root lattice:
we say that 
$\alpha\prec \beta$ if $\beta-\alpha$ is a sum of positive roots. Additionally, for each $w\in W$ we will need an order we will denote by $\prec_w$:
we say that $\alpha \prec_w \beta$ if $w^{-1}\alpha \prec w^{-1}\beta$.

To work with the Weyl group and its elements $w \in W$, we will extensively use the set $\Phi^+\cap w\Phi^-$.
Let us mention the following 
properties
of this set right away.
\begin{lemma}[{follows from \cite[Lemma 2.2]{bgg}}]\label{phipmlength}
If $w\in W$, then $\ell(w)=|\Phi^+\cap w\Phi^-|$. 
\end{lemma}
\begin{lemma}\label{phiwphidetermines}
The set $\Phi^+ \cap w \Phi^-$ determines $w$ uniquely. 
In other words, if $v,w \in W$ are such that 
$\Phi^+ \cap v \Phi^- = \Phi^+ \cap w \Phi^-$, then $v=w$.
\end{lemma}
\begin{proof}
Let us check that $\Phi^+\cap v^{-1} w\Phi^- = \varnothing$. 
Assume the contrary: let $\alpha \in \Phi^+\cap v^{-1} w\Phi^-$. Then $v \alpha \in w \Phi^-$ and $w^{-1}v \alpha \in \Phi^-$.

If $v \alpha \in \Phi^+$, then $v \alpha \in \Phi^+ \cap w \Phi^- = \Phi^+ \cap v \Phi^-$, and $v^{-1} v \alpha \in \Phi^-$.
But $v^{-1}v \alpha = \alpha \in \Phi^+$, a contradiction.

If $v \alpha \in \Phi^-$, then $-v\alpha \in \Phi^+ \cap v \Phi^- = \Phi^+ \cap w \Phi^-$, so $-w^{-1}v\alpha \in \Phi^-$ 
and $w^{-1}v\alpha \in \Phi^+$. 
But we know that $w^{-1}v \alpha \in \Phi^-$, a contradiction.

So, $\Phi^+\cap v^{-1} w\Phi^- = \varnothing$. By Lemma \ref{phipmlength}, $\ell(v^{-1} w)=0$, and $v^{-1}w=1_W$.
Hence, $w=v$.
\end{proof}

Also, for root systems of types $A$, $D$, or $E$, we will often need the following well-known 
correspondence between scalar products 
and the possibility to add or subtract roots. 
We will sometimes use these facts implicitly, without quoting this lemma.
\begin{lemma}\label{sumexists}
Let $\alpha,\beta\in \Phi$. Then all possible values of $(\alpha,\beta)$ are $-2$, $-1$, 0, 1, and 2.
Moreover:
\begin{enumerate}
\item $\alpha=-\beta$ if and only if $(\alpha,\beta)=-2$.
\item $\alpha+\beta\in \Phi$ if and only if $(\alpha,\beta)=-1$.
\item $\alpha-\beta\in \Phi$ if and only if $(\alpha,\beta)=1$.
In this case, $\alpha$ and $\beta$ are comparable for $\prec$ and for 
the orders $\prec_w$ for all $w\in W$.
\item $\alpha=\beta$ if and only if $(\alpha,\beta)=2$.
\end{enumerate}
\end{lemma}

We will have some examples for $G=SL_{r+1}$. 
In these examples, we will use the standard construction for the root system $\Phi$ of type $A_r$: 
we start with a $\QQ$-vector space $\tilde{E}$ with orthonormal basis $\varepsilon_1, \ldots \varepsilon_{r+1}$
and set $\Phi=\{\varepsilon_i-\varepsilon_j \mid i \ne j\}$, $\alpha_i = \varepsilon_i-\varepsilon_{i+1}$.
The weight lattice is then more conveniently viewed as embedded into the quotient $E=\tilde{E}/\langle \varepsilon_1+\ldots +\varepsilon_{r+1}\rangle$
(rather than embedded into the subspace of $\tilde{E}$ where the sum of coordinates is zero): 
then we can simply say $\varpi_i = \varepsilon_1+\ldots+\varepsilon_i$.

The Weyl group in this case is the permutation group $S_{r+1}$.
For brevity, we will not use the standard two-row notation for elements of $S_{r+1}$.
Instead, we will denote each permutation $w \in S_{r+1}$ by the result of $w$ applied to the vector $(1,2,\ldots,r+1) \in E$.
(In fact, this vector can be viewed as the weight $-\rho=-\varpi_1-\ldots-\varpi_r=-\sum_{\alpha \in \Phi^+} \alpha/2$.)
Formally, if $\{s_1,\ldots,s_{r+1}\}=\{1, \ldots, r+1\}$, then 
we denote the permutation $w \in S_{r+1}$
such that $w(1,2,\ldots,r+1)=(s_1,\ldots,s_{r+1})$
by $[s_1,\ldots,s_{r+1}]$.
The transposition interchanging the $i$th and the $j$th positions 
(formally: the $i$th and the $j$th coordinates of vectors in $E$) will be denoted by $(i\leftrightarrow j)$.
It also equals $\sigma_{\varepsilon_i-\varepsilon_j}$.


\begin{example}
The length of an element $[s_1,\ldots,s_{r+1}]\in W$ is the number of inversions, i.e. the number of pairs $(i,j)$ with $i<j$ and $s_i>s_j$.
And the set $\Phi^+ \cap w \Phi^-$ consists of all roots $\varepsilon_j - \varepsilon_i = \alpha_i + \ldots + \alpha_{j-1}$ for all 
$i$ and $j$ satisfying the same conditions: $i < j$ and $s_i > s_j$.
\end{example}

\subsection{Monomials in Schubert divisors}
To simplify the notation for monomials in Schubert divisors, we will use a lattice $\ZZ^r$ (more precisely, its orthant $\ZZ^r_{\ge 0}$) for exponents:
if $n_\bullet = (n_1, \ldots, n_r) \in \ZZ^r_{\ge 0}$, we write $D_{\bullet}^{n_{\bullet}} = D_1^{n_1} \ldots D_r^{n_r}$. 
We will also use this notation for the coefficients $c_{w,n_1, \ldots, n_r}$: we write them as $c_{w,n_\bullet}$, 
i.e. $D_\bullet^{n_\bullet} = \sum_{w \in W} c_{w,n_\bullet} Z_w$.
Also, we denote $|n_\bullet| = n_1 + \ldots + n_r$ (then $D_\bullet^{n_\bullet} \in \CH^{|n_\bullet|}(G/B)$)
and we define the \emph{support} of $n_\bullet$ as $\supp n_\bullet = \{\alpha_i \in \Pi \mid n_i > 0\}$.

We will often need to remove some factors from a monomial in Schubert divisors, in other words, to replace some of the exponents $n_i$ with zeros.
To simplify the notation for that, we denote the standard basis of $\ZZ^r$ by $e_1, \ldots, e_r$ (so, $D_\bullet^{e_i} = D_i$)
and say that this basis is orthonormal. Then, for any $I \subseteq \Pi$, denote by $p_I$ the orthogonal projection 
$\ZZ^r \to \langle e_i \mid \alpha_i \in I \rangle$. In other words, $D_\bullet^{p_I(n_\bullet)} = \prod_{i \colon \alpha_i \in I} D_i^{n_i}$.

The following Chevalley-Pieri formula proved in \cite{demazure}\footnote{Same remark as in the Introduction: in \cite{demazure}, 
a different correspondence between $W$ and Schubert varieties is used: $X_u = [\overline{BuB/B}]$ (like in \cite{chevalley})
instead of 
$Z^w=X_{\weylmax w^{-1}}$.
} 
will be our basic tool to work with multiplication in $\CH^*(G/B)$.
\begin{proposition}[{\cite[\S 4.4, Corollary 2]{demazure}}]\label{pieriformula}
Let $G$ be a semisimple split simply connected group.
Let $\alpha_i\in \Pi$,
and let $w\in W$. Then
$$
D_iZ^w=\sum_{\substack{\alpha\in \Phi^+\\\ell(\sigma_{\alpha}w)=\ell(w)+1}}\langle \varpi_i,\alpha\rangle Z^{\sigma_{\alpha}w}.
$$
\end{proposition}

In view of this formula, let us recall the definition of \emph{Bruhat graph}. 
Its vertices are (in one-to-one correspondence with) $W$, and we have an edge $w \to w'$ with label $\alpha \in \Phi^+$
if and only if $w' = \sigma_\alpha w$ and $\ell(w')=\ell(w)+1$.
Following \cite{bgg}, we will often simply write ``$w \xrightarrow{\alpha} w'$'' instead of 
``$w \xrightarrow{\alpha} w'$ is an edge in the Bruhat graph''.
Then the formula in Proposition \ref{pieriformula} can be rewritten as follows.
\begin{equation}\label{pieriformulabruhat}
D_iZ^w=\sum_{\alpha\in \Phi^+ : w \xrightarrow{\alpha} \sigma_{\alpha}w}\langle \varpi_i,\alpha\rangle Z^{\sigma_{\alpha}w}.
\end{equation}

\begin{sideremark}
Bruhat graph can also be used to detect adjunction between Schubert varieties (see \cite{chevalley}, \cite[Theorem 2.11]{bgg}), 
although we will not need this in what follows.
Let $w, w' \in W$, $X_w = \overline{BwB/B}$, $X_{w'} = \overline{Bw'B/B}$.
Then there is an edge $w \to w'$ in the Bruhat graph if and only if $X_w \subset X_{w'}$ and $\dim X_{w'} = \dim X_w + 1$.
And $X_w \subseteq X_{w'}$ (without restrictions on dimensions) if and only if there is an oriented path (of arbitrary length) from $w$ to $w'$.
\end{sideremark}


\begin{example}
The Bruhat graph for $G=SL_3$ looks as follows.
$$
\xymatrix{
& [3,2,1] & \\
[3,1,2]\ar[ur]^{\alpha_2} && [2,3,1]\ar[ul]_{\alpha_1} \\
[1,3,2]\ar[urr]|(0.33){\alpha_1+\alpha_2} \ar[u]^{\alpha_1} && [2,1,3]\ar[ull]|(0.33){\alpha_1+\alpha_2} \ar[u]_{\alpha_2} \\
& [1,2,3] \ar[ur]_{\alpha_1} \ar[ul]^{\alpha_2}
}
$$
\end{example}


%
%
%
%

\section{Criterion for $c_{w,n_\bullet}=0$}\label{sectioncriterionzero}
In this section, we are going to establish a criterion for $c_{w,n_\bullet}=0$ in terms of the set $\Phi^+ \cap w \Phi^-$.
In this section, the type of group $G$ is still arbitrary, not necessarily $A$, $D$, or $E$.
We start with 
a few
corollaries of Chevalley-Pieri formula and a remark about it.
\begin{corollary}\label{pieriformulacwn}
Let $w \in W$ and $n_\bullet \in \ZZ_{\ge 0}^r$ be such that $\ell(w)=|n_\bullet|$. 
For $w = 1_W$, we have $c_{1_W, 0}=1$.
If $\ell(w) > 0$ and 
$\alpha_i \in \supp n_\bullet$, then
\begin{equation}\label{pieriformulacwnformula}
c_{w, n_\bullet}=\sum_{\beta \in \Phi^+ : \sigma_{\beta}w \xrightarrow{\beta} w} c_{\sigma_{\beta} w, n_\bullet - e_i} \langle \varpi_i,\beta \rangle.
\end{equation}
\end{corollary}
\begin{proof}
We write 
$D_{\bullet}^{n_\bullet - e_i} = \sum_{v \in W} c_{v,n_\bullet - e_i} Z^v$,
multiply both sides by $D_i$, and apply (\ref{pieriformulabruhat}) to the right-hand side.
\end{proof}
\begin{remark}\label{shortenpath}
It follows directly from the non-negativity of the coefficients $\langle \varpi_i,\alpha \rangle$
in (\ref{pieriformulacwnformula})
that if $D_{\bullet}^{n_\bullet}$ 
is \emph{not} a multiplicity-free monomial for some $n_\bullet \in \ZZ_{\ge 0}$, then 
$D_i D_{\bullet}^{n_\bullet}$ is not a multiplicity-free monomial either (for $1 \le i \le r$).
If $d$ is the maximal possible degree of a multiplicity-free monomial (for a given group $G$), 
then all possible degrees of multiplicity-free monomials are exactly all integers between 0 and $d$, inclusively.
\end{remark}
\begin{corollary}\label{pieriformulacwnmultiple}
Let $w \in W$ and $m_\bullet, n_\bullet \in \ZZ_{\ge 0}^r$ be such that $\ell(w)=|m_\bullet + n_\bullet|$.
Denote $l=|n_\bullet|$, and let $\alpha_{i_1},\ldots,\alpha_{i_l}$ be a sequence of simple roots 
where each $\alpha_i$ occurs exactly $n_i$ times in total.
Then
\begin{equation}\label{pieriformulacwnmultiplerel}
c_{w,m_\bullet+n_\bullet} = 
\sum_{v \xrightarrow{\beta_1} \ldots \xrightarrow{\beta_l} w} 
\langle \varpi_{i_l}, \beta_l \rangle \ldots \langle \varpi_{i_1}, \beta_1 \rangle c_{v,m_\bullet},
\end{equation}
where the sum is taken over all paths in the Bruhat graph ending at $w$ (without any further restrictions imposed on the beginning $v$ of the path).

In particular, if we take $m_\bullet=0$, we get 
\begin{equation*}
c_{w,n_\bullet} = 
\sum_{1_W \xrightarrow{\beta_1} \ldots \xrightarrow{\beta_l} w} 
\langle \varpi_{i_l}, \beta_l \rangle \ldots \langle \varpi_{i_1}, \beta_1 \rangle
\end{equation*}
(here, the sum is again taken over all paths ending at $w$, 
but now such path automatically begins at $1_W$, because this is the only option for a path of length $l=|n_\bullet|=\ell(w)$ ending at $w$).
\end{corollary}
\begin{proof}
%
Follows from Corollary \ref{pieriformulacwn} directly by induction on $l$
(see also \cite[\S 4]{poststan}, where this corollary is deduced for powers of a single divisor and for $m_\bullet=0$).
\end{proof}
%
This corollary motivates the following terminology. Let us say that 
a \emph{pre-enhanced path} in the Bruhat graph is a pair consisting of an oriented path 
$u_0 \xrightarrow{\beta_1} u_1 \xrightarrow{\beta_2} \ldots \xrightarrow{\beta_l} u_l$ in the Bruhat graph 
and a sequence of the same length $\alpha_{i_1}, \ldots, \alpha_{i_l}$ of simple roots.
The sequence will be called the \emph{enhancement} of the pre-enhanced path. 
The \emph{Z-multiplicity} of such a pre-enhanced path is defined as
$\langle \varpi_{i_1},\beta_1 \rangle \langle \varpi_{i_2},\beta_2 \rangle \ldots \langle \varpi_{i_l},\beta_l \rangle$.
(Note that if $G$ is of type $A$, $D$, or $E$, 
then we can replace $\langle \varpi_{i_j},\beta_j\rangle$ in this formula with $\alpha_{i_j}^*(\beta_j)$.)
And an \emph{enhanced path} is a pre-enhanced path with positive Z-multiplicity or, equivalently, 
a pre-enhanced path such that $\alpha_{i_j} \in \supp \beta_j$ for all $j \in \{1, \ldots, l\}$ (this equivalence holds for arbitrary type of $G$).
Finally, the \emph{vector of D-multiplicities} of the sequence $\alpha_{i_1}, \ldots, \alpha_{i_l}$
is the vector $n_\bullet \in \ZZ_{\ge 0}^r$, 
where each $n_j$ equals the number of times $\alpha_j$ occurs among $\alpha_{i_1}, \ldots, \alpha_{i_l}$.

Then, Corollary \ref{pieriformulacwnmultiple} for $m_\bullet=0$ can be reformulated as follows.
\begin{corollary}\label{pieriformulamultiple2}
Let $n_\bullet \in \ZZ_{\ge 0}^r$, $l = |n_\bullet|$.
Let $w \in W$ be an element with $\ell(w)=l$.
Choose an arbitrary sequence of simple roots $\alpha_{i_1}, \ldots, \alpha_{i_l}$ with vector of D-multiplicities $n_\bullet$.
Then $c_{w, n_\bullet}$ equals the number of pre-enhanced paths in the Bruhat graph leading from $1_W$ to $w$, 
with enhancement $\alpha_{i_1}, \ldots, \alpha_{i_l}$, counting their Z-multiplicities.
Alternatively, instead of all pre-enhanced paths, it is sufficient to consider enhanced paths only.
\end{corollary}
Note that the formulas in Corollaries \ref{pieriformulacwnmultiple} and \ref{pieriformulamultiple2}
do not immediately imply that the multiplication of Schubert divisors is commutative. 
But in reality, we know that the multiplication in the Chow ring is commutative, which, in terms of Corollary \ref{pieriformulamultiple2},
implies that the answer ($c_{w, n_\bullet}$, or the sum of Z-multiplicities) does not depend on the choice of the enhancement
as long as the vector of D-multiplicities stays the same.


Next (and many further times throughout the paper),
we 
need
a relation between the sets $\Phi^+ \cap w' \Phi^-$ and 
$\Phi^+ \cap w \Phi^-$ if $w' \to w$ is an edge in the Bruhat graph. One important case of such a relation follows from the results of \cite{bgg}.
\begin{lemma}\label{antisimpleexistsadmissible}
\begin{enumerate}
\item\label{antisimpleexistsbgg} Let $w \in W$, $w \ne 1_W$. Then there exists $\alpha \in \Phi^+ \cap w (-\Pi)$ (i.e. $\Phi^+ \cap w (-\Pi) \ne \varnothing$).
\item\label{antisimpleadmissible} Let $w \in W$, $\alpha \in \Phi^+ \cap w (-\Pi)$. 
Then $\sigma_\alpha w \xrightarrow{\alpha} w$, 
and $\Phi^+ \cap \sigma_\alpha w \Phi^- = (\Phi^+ \cap w \Phi^-) \setminus \{ \alpha \}$.
\end{enumerate}
\end{lemma}
\begin{proof}
For part \ref{antisimpleexistsbgg}, let $w = \sigma_{i_1} \ldots \sigma_{i_l}$ be a reduced expression. 
Let $\alpha = w(-\alpha_{i_l}) \in w(-\Pi)$. We have $\alpha = \sigma_{i_1} \ldots \sigma_{i_l} (-\alpha_{i_l}) = 
\sigma_{i_1} \ldots \sigma_{i_{l-1}} (\alpha_{i_l})$. Hence, by \cite[Lemma 2.2]{bgg}, $\alpha \in \Phi^+$.

Let us prove part \ref{antisimpleadmissible}.
First, $\alpha \in w \Phi^-$, so by \cite[Lemma 2.3 (ii)]{bgg} and its proof, $\ell(\sigma_\alpha w) < \ell (w)$. 
On the other hand, $w = \sigma_\alpha w \sigma_{-w^{-1}(\alpha)}$, and $-w^{-1} (\alpha) \in \Pi$, so $\ell(\sigma_\alpha w) \ge \ell(w)-1$, 
and $\sigma_\alpha w \xrightarrow{\alpha} w$.

Now, let $l = \ell(w)$, and let $\sigma_\alpha w = \sigma_{i_1} \ldots \sigma_{i_{l-1}}$ be a reduced expression. 
Denote $\alpha_{i_l} = -w^{-1} (\alpha)$, then $w = \sigma_{i_1} \ldots \sigma_{i_{l-1}} \sigma_{i_l}$ is also a reduced expression.
By \cite[Lemma 2.2]{bgg}, $\Phi^+ \cap \sigma_{\alpha} w \Phi^- = \{ \sigma_{i_1} \ldots \sigma_{i_{k-1}} (\alpha_{i_k}) \}_{1 \le k \le l-1}$, 
and $\Phi^+ \cap w \Phi^- = \{ \sigma_{i_1} \ldots \sigma_{i_{k-1}} (\alpha_{i_k}) \}_{1 \le k \le l}$.
Finally, $\sigma_{i_1} \ldots \sigma_{i_{l-1}} (\alpha_{i_l}) = \sigma_\alpha w (-w^{-1}(\alpha)) = \sigma_\alpha (-\alpha) = \alpha$.
\end{proof}

This lemma motivates the following definition. We say that an edge $\sigma_\alpha w \xrightarrow{\alpha} w$ 
of the Bruhat graph is \emph{antisimple} if $\alpha \in \Phi^+ \cap w(-\Pi)$ 
(equivalently, an edge $w \xrightarrow{\alpha} \sigma_\alpha w$ is antisimple if $\alpha \in \Phi^+ \cap w \Pi$).
A path in the Bruhat graph is called \emph{antisimple} if it consists of antisimple edges.
\begin{example}\label{antisimplesln}
If $G=SL_{r+1}$, then $W=S_{r+1}$. 
Let $w \in W$, $w=[s_1,\ldots,s_{r+1}]$.
Then the antisimple edges leading to $w$ are precisely edges of the form
$(i\leftrightarrow j) w \to w$, where $i<j$ and $s_i=s_j+1$ 
(recall that $i\leftrightarrow j$ denotes the transposition interchanging the $i$th and the $j$th coordinate in the $(r+1)$-dimensional space $E$ where $W$ acts, 
in other words, $(i\leftrightarrow j) w=[s_1,\ldots, s_{i-1},s_j,s_{i+1},\ldots,s_{j-1},s_i,s_{j+1},\ldots,s_{r+1}]$).
The label at such an edge is $\varepsilon_j-\varepsilon_i=\alpha_i+\ldots+\alpha_{j-1}$, and 
we have $\varepsilon_j-\varepsilon_i = -w \alpha_{s_j}$.
\end{example}

So, the set $\Phi^+ \cap w \Phi^-$ behaves quite nicely when we apply antisimple reflections. 
But we will also need to know how this set changes when we go along an arbitrary edge of the Bruhat graph, not necessarily an antisimple one.
The answer is given by the following definition and lemma.
\begin{definition}\label{phipm}
Let $w \in W$. 
Let $\alpha \in \Phi^+ \cap w \Phi^-$.
Then we define a map $\psi_{w,\alpha} \colon (\Phi^+ \cap w \Phi^-)\setminus \{\alpha\} \to \Phi$ 
as follows.
\begin{equation*}
\psi_{w,\alpha}(\beta) = 
\begin{cases}
\sigma_\alpha \beta & \text{if } (\beta, \alpha) > 0 \text{ and } \sigma_\alpha \beta \in \Phi^+ \cap w \Phi^+, \\
\beta & \text{otherwise}.
\end{cases}
\end{equation*}
Sometimes we will write simply 
$\psi_{\alpha}$
instead of 
$\psi_{w,\alpha}$ 
if it is clear which $w\in W$ we are talking about.

If $G$ is of type $A$, $D$, or $E$, then by Lemma \ref{sumexists}, we can simply say 
``$\beta - \alpha$ if $\beta-\alpha \in \Phi^+ \cap w \Phi^+$''
instead of 
``$\sigma_\alpha \beta$ if $(\beta, \alpha) > 0$ and $\sigma_\alpha \beta \in \Phi^+ \cap w \Phi^+$''.
\end{definition}

\begin{lemma}\label{admissibilitycrit}
Let $w \in W$, $\alpha \in \Phi^+$. 
Then $\sigma_\alpha w \xrightarrow{\alpha} w$ 
if and only if
$\alpha \in \Phi^+ \cap w \Phi^-$ 
and there are no pairs of roots $\beta, \gamma \in (\Phi^+ \cap w \Phi^-) \setminus \{ \alpha \}$ 
of equal lengths
such that 
$\beta + \gamma = k \alpha$ for some $k \in \NN$
(if $G$ is of type $A$, $D$, or $E$, this is can be reformulated as $\beta + \gamma = \alpha$).

In this case, 
$\psi_{w,\alpha}$
establishes a 
bijection between $(\Phi^+ \cap w \Phi^-) \setminus \{\alpha\}$ and $\Phi^+ \cap \sigma_\alpha w \Phi^-$.

Also, we have 
$\psi_{w,\alpha}(\beta) \preceq \beta$ 
for all $\beta \in \Phi^+ \cap w \Phi^-$. 
And if $\alpha \in w(-\Pi)$, 
then 
$\psi_{w,\alpha}$ 
is the identity map.
\end{lemma}
\begin{proof}
First, by the arguments in the proof of \cite[Corollary 2.3(ii)]{bgg}, 
$\alpha \in \Phi^+ \cap w \Phi^-$ if and only if 
$\ell(\sigma_\alpha w) < \ell(w)$.
So, $\alpha \in \Phi^+ \cap w \Phi^-$ 
is a necessary condition for 
$\sigma_\alpha w \xrightarrow{\alpha} w$.

Next, 
the reflection $\sigma_\alpha$ has the following three types of orbits on $\Phi$:
\begin{enumerate}
\item\label{orbit:alpha} $\{\alpha,-\alpha\}$;
\item\label{orbit:orthogonal} $\{\beta\}$ (a fixed point) for each $\beta\in \Phi$, $(\alpha,\beta)=0$;
\item\label{orbit:sum} for each pair of roots $\beta, \gamma \in \Phi \setminus \{ \pm\alpha \}$ of equal lengths 
such that $\alpha=k(\beta+\gamma)$, $k \in \NN$, 
we have two orbits: $\{\beta,-\gamma=\beta-k\alpha\}$ and $\{\gamma,-\beta=\gamma-k\alpha\}$.
\end{enumerate}
For each such orbit $O$, let us check that $|O \cap (\Phi^+ \cap \sigma_\alpha w \Phi^-)| \le |O \cap (\Phi^+ \cap w \Phi^-)|$.
Let us also check that when this inequality becomes equality, we also have 
$O \cap (\Phi^+ \cap \sigma_\alpha w \Phi^-) = \psi_{w,\alpha} (O \cap (\Phi^+ \cap w \Phi^-))$.

For the orbit of type \ref{orbit:alpha}, $O=\{\alpha, -\alpha\}$, everything is clear: on one hand, $O \cap (\Phi^+ \cap w \Phi^-) = \{\alpha\}$.
On the other hand, $-\alpha = \sigma_\alpha \alpha \in \sigma_\alpha w \Phi^-$, so $\alpha \in \sigma_\alpha w \Phi^+$, also $-\alpha \in \Phi^-$, 
so $O \cap (\Phi^+ \cap \sigma_\alpha w \Phi^-) = \varnothing$.
In this case, we actually have $|O \cap (\Phi^+ \cap \sigma_\alpha w \Phi^-)| < |O \cap (\Phi^+ \cap w \Phi^-)|$
(in fact, $|O \cap (\Phi^+ \cap \sigma_\alpha w \Phi^-)| = |O \cap (\Phi^+ \cap w \Phi^-)| - 1$).

For an orbit of type \ref{orbit:orthogonal}, $O=\{\beta\}$, everything is also clear: 
$\beta \in w \Phi^- \Leftrightarrow \beta = \sigma_\alpha \beta \in \sigma_\alpha w \Phi^-$, so 
we actually have 
$O \cap (\Phi^+ \cap \sigma_\alpha w \Phi^-) = O \cap (\Phi^+ \cap w \Phi^-) = \psi_{w,\alpha} (O \cap (\Phi^+ \cap w \Phi^-))$.

Finally, consider an orbit of type \ref{orbit:sum}.
Without loss of generality, consider $O = \{\beta, -\gamma\}$ (for the other orbit, the same argument works if we switch the notation for $\beta$ and $\gamma$). 
First, note that if $\beta = (-\gamma)+k\alpha \in \Phi^-$, then we must also have $-\gamma \in \Phi^-$ since $\alpha \in \Phi^+$.
So, if $\beta \in \Phi^-$, then 
$O \cap (\Phi^+ \cap \sigma_\alpha w \Phi^-) = O \cap (\Phi^+ \cap w \Phi^-) = \varnothing$. In what follows, suppose that $\beta \in \Phi^+$.

Similarly, if $\beta \in w \Phi^+$, then $-\gamma \in w \Phi^+$ since $\alpha \in w \Phi^-$, 
and also $\beta = \sigma_\alpha (-\gamma) \in \sigma_\alpha w \Phi^+$ and 
$-\gamma = \sigma_\alpha \beta \in \sigma_\alpha w \Phi^+$.
So, if $\beta \in w \Phi^+$, then again
$O \cap (\Phi^+ \cap \sigma_\alpha w \Phi^-) = O \cap (\Phi^+ \cap w \Phi^-) = \varnothing$. In what follows, suppose that $\beta \in w \Phi^-$.

Let us consider 4 cases for $-\gamma$: whether it is positive or negative, and whether it belongs to $w\Phi^+$ or $w\Phi^-$.
We always keep in mind that $\beta \in w \Phi^- \Leftrightarrow -\gamma = \sigma_\alpha \beta \in \sigma_\alpha w \Phi^-$
and $-\gamma \in w \Phi^- \Leftrightarrow \beta = \sigma_\alpha (-\gamma) \in \sigma_\alpha w \Phi^-$.
Also note that $(\alpha,\beta) > 0$ (this is important for the definition of $\psi_{w,\alpha}$).
\begin{enumerate}
\item $-\gamma \in \Phi^+ \cap w \Phi^+$. 
Then 
$O \cap (\Phi^+ \cap w \Phi^-) = \{ \beta \}$, 
$O \cap (\Phi^+ \cap \sigma_\alpha w \Phi^-) = \{ -\gamma \} 
= \{ \sigma_\alpha \beta \}
= \{ \psi_{w,\alpha} (\beta) \}$.
\item $-\gamma \in \Phi^+ \cap w \Phi^-$. 
Then 
$O \cap (\Phi^+ \cap w \Phi^-) = 
O \cap (\Phi^+ \cap \sigma_\alpha w \Phi^-) = O = \psi_{w,\alpha} (O)$.
\item $-\gamma \in \Phi^- \cap w \Phi^+$. 
Then 
$O \cap (\Phi^+ \cap \sigma_\alpha w \Phi^-) = \varnothing$, while 
$O \cap (\Phi^+ \cap w \Phi^-) = \{ \beta \}$.
We have $|O \cap (\Phi^+ \cap \sigma_\alpha w \Phi^-)| < |O \cap (\Phi^+ \cap w \Phi^-)|$.
\item $-\gamma \in \Phi^- \cap w \Phi^-$. 
Then 
$O \cap (\Phi^+ \cap \sigma_\alpha w \Phi^-) = \{ \beta \}$ and
$O \cap (\Phi^+ \cap w \Phi^-) = \{ \beta \}$. Also, $\psi_{w,\alpha} (\beta) = \beta$.
\end{enumerate}

Now, note that we have exactly 1 orbit $O$ of type \ref{orbit:alpha}, which gives 
$|O \cap (\Phi^+ \cap \sigma_\alpha w \Phi^-)| = |O \cap (\Phi^+ \cap w \Phi^-)| - 1$.
Therefore, we have 
$|\Phi^+ \cap w \Phi^-| - |\Phi^+ \cap \sigma_\alpha w \Phi^-| = 1$ if and only if we have no orbits 
$\{ \beta, -\gamma \}$ of type \ref{orbit:sum} with $-\gamma \in \Phi^- \cap w \Phi^+$, otherwise 
we have $|\Phi^+ \cap w \Phi^-| - |\Phi^+ \cap \sigma_\alpha w \Phi^-| > 1$.
In other words, $|\Phi^+ \cap w \Phi^-| - |\Phi^+ \cap \sigma_\alpha w \Phi^-| = 1$
if and only if we have no pairs of roots $\beta, \gamma \in \Phi \setminus \{ \pm \alpha\}$ of equal lengths such that 
$\beta + \gamma = k \alpha$ for some $k \in \NN$, 
$\beta \in \Phi^+ \cap w \Phi^-$, and $\gamma \in \Phi^+ \cap w \Phi^-$ (note that this condition is now 
symmetric in $\beta$ and $\gamma$, so the other orbit listed in type \ref{orbit:sum} would give us one more strict inequality in this case).

Finally, if we have no such pairs, then we have also checked that for each orbit $O$ except for the one orbit of type \ref{orbit:alpha},
$\psi_{w,\alpha}$ establishes a bijection between 
$O \cap (\Phi^+ \cap w \Phi^-)$ and $O \cap (\Phi^+ \cap \sigma_\alpha w \Phi^-)$.
Overall, we have a bijection between $(\Phi^+ \cap w \Phi^-) \setminus \{ \alpha \}$
and $O \cap (\Phi^+ \cap \sigma_\alpha w \Phi^-)$.

The inequality $\psi_{w,\alpha}(\beta) \preceq \beta$ follows directly from the definition of $\psi_{w,\alpha}$.
And if $\alpha \in w(-\Pi)$, then we already know by Lemma \ref{antisimpleexistsadmissible} (\ref{antisimpleadmissible})
that $\Phi^+ \cap \sigma_\alpha w \Phi^- = (\Phi^+ \cap w \Phi^-) \setminus \{\alpha\}$.
So, $\psi_{w,\alpha}$ is in fact a permutation of the finite partially ordered set 
$(\Phi^+ \cap w \Phi^-) \setminus \{\alpha\}$, and $\psi_{w,\alpha}$ is also a non-increasing map, so it must be the identity permutation.
\end{proof}

\begin{example}\label{admissibleexample}
If $G=SL_{r+1}$, then $W=S_{r+1}$. 
If $w=[s_1,\ldots,s_{r+1}]$, then 
$\sigma_\alpha w \xrightarrow{\alpha} w$ if and only if
$\sigma_\alpha$ is a transposition
$(i\leftrightarrow j)$ such that $i<j$, $s_i>s_j$, and there are no indices $k$ such that $i<k<j$ and $s_i>s_k>s_j$.
\end{example}

\begin{sideremark}
In Lemmas \ref{antisimpleexistsadmissible} and \ref{admissibilitycrit} (and in Definition \ref{phipm}), 
we speak about
edges \emph{ending} at a fixed element $w \in W$.
We will not need this later, but after a slight modification, it is possible to have a version of these lemmas and definition for 
edges \emph{beginning} at a fixed $w \in W$, i.e. for edges of the form $w \xrightarrow{\alpha} \sigma_\alpha w$.
For example, if we have such an edge, then $\alpha \in \Phi^+ \cap w \Phi^+$, and in Definition \ref{phipm}, 
instead of $\phi_{w,\alpha}$, we could define, say, $\xi_{w,\alpha} \colon \Phi^+ \cap w \Phi^- \to \Phi$
by replacing ``$(\beta, \alpha) > 0$'' with ``$(\beta, \alpha) < 0$'' (and keeping the rest of the definition unchanged).
Then this $\xi_{w,\alpha}$ would establish a bijection between $\Phi^+ \cap w \Phi^-$ and $(\Phi^+ \cap \sigma_\alpha w \Phi^-) \setminus \{\alpha\}$.
In fact, we would even get $\xi_{w,\alpha}=\psi_{\sigma_\alpha w,\alpha}^{-1}$, but we will not need all of that.
\end{sideremark}

Now we are ready to see how the set $\Phi^+ \cap w \Phi^-$ is related to the set of labels in a path in the Bruhat graph.


\begin{lemma}\label{antisimplelabelsisphiwphi}
\begin{enumerate}
\item\label{antisimplepathexists} For each $w \in W$, there exists an antisimple (oriented) path in the Bruhat graph from $1_W$ to $w$.
\item\label{antisimplepathinequality} Let $v, w \in W$, let 
$v \xrightarrow{\beta_1} \ldots \xrightarrow{\beta_l} w$ 
be a path in the Bruhat graph. 
Then there exists 
a bijection 
$h \colon (\Phi^+ \cap v \Phi^-) \sqcup \{ 1, \ldots, l\} \to \Phi^+ \cap w \Phi^-$ such that 
$h(l) = \beta_l$,
$h(i) \succeq \beta_i$
for $1 \le i \le l-1$,
and $h(\alpha) \succeq \alpha$ for all $\alpha \in \Phi^+ \cap v \Phi^-$. 

Moreover, if the path 
$v \xrightarrow{\beta_1} \ldots \xrightarrow{\beta_l} w$ 
is antisimple, then we actually have 
$h(i) = \beta_i$ for $1 \le i \le l$, $h(\alpha) = \alpha$ for all $\alpha \in \Phi^+ \cap v \Phi^-$, 
and 
$\{ \beta_1, \ldots, \beta_l \} \sqcup (\Phi^+ \cap v \Phi^-) = \Phi^+ \cap w \Phi^-$.
For $v=1_W$ (still assuming the path is antisimple), we simply get 
$\{ \beta_1, \ldots, \beta_l \} = \Phi^+ \cap w \Phi^-$.
\end{enumerate}
\end{lemma}
\begin{proof}
Part \ref{antisimplepathexists} follows from Lemma \ref{antisimpleexistsadmissible} (\ref{antisimpleexistsbgg}) directly by induction on $\ell(w)$.

For part \ref{antisimplepathinequality}, we also use induction on $\ell(w)$ (with fixed $v$).
For $w=v$, we take the identity map as $h$. 
If $w \ne v$, then $\beta_l \in \Phi^+\cap w\Phi^-$ by Lemma \ref{admissibilitycrit}, 
and we set $h(l) = \beta_l$.

Next, by the induction hypothesis 
(applied to the path $v \xrightarrow{\beta_1} \ldots \xrightarrow{\beta_{l-1}} \sigma_{\beta_l} w$),
there exists 
a bijection 
$g \colon (\Phi^+ \cap v \Phi^-) \sqcup \{ 1, \ldots, l-1 \} \to \Phi^+\cap \sigma_{\beta_l} w \Phi^-$
such that $g(i) \succeq \beta_i$ for $1 \le i \le l-1$ and $g(\alpha) \succeq \alpha$ for all $\alpha \in \Phi^+ \cap v \Phi^-$.

By Lemma \ref{admissibilitycrit}, $\psi_{\beta_l}$ is a bijection between 
$(\Phi^+\cap w\Phi^-)\setminus \{ \beta_l \}$ 
and $\Phi^+\cap \sigma_{\beta_l} w \Phi^-$.
For $1 \le i \le l-1$ (resp.\ $\alpha \in \Phi^+ \cap v \Phi^-$), set
$h(i) = \psi_{\beta_l}^{-1} (g(i))$ (resp.\ $h(\alpha) = \psi_{\beta_l}^{-1} (g(\alpha))$).
Then the bijectivity of $h$ follows directly by construction.
And $\beta_i \preceq g(i) \preceq \psi_{\beta_l}^{-1} (g(i)) = h(i)$ 
(resp.\ $\alpha \preceq g(\alpha) \preceq h(\alpha)$)
by Lemma \ref{admissibilitycrit}.

Finally, if the path 
$v \xrightarrow{\beta_1} \ldots \xrightarrow{\beta_l} w$ 
is antisimple, then 
the first inequality in these formulas becomes equality by the induction hypothesis (since the first $l-1$ edges are antisimple), 
and the second inequality becomes equality by Lemma \ref{admissibilitycrit} (since the last edge is antisimple).
And the equality $\{ \beta_1, \ldots, \beta_l \} \sqcup (\Phi^+ \cap v \Phi^-) = \Phi^+ \cap w \Phi^-$
follows from the preceding claims of the lemma, which we have already proved now.
\end{proof}

With this lemma, one could already formulate a certain answer to the question when $c_{w,n_{\bullet}} = 0$ and when 
$c_{w,n_{\bullet}} > 0$.
Roughly speaking, we are going to choose simple roots $\alpha_{i_k} \in \supp \beta_k \subseteq \supp h(k)$, then 
$\langle \varpi_{i_k}, \beta_k \rangle > 0$.
But before we formulate our final answer in the whole generality, 
let us 
have 
a few definitions and 
combinatorial lemmas.

\begin{definition}
Given a set of positive roots $A\subseteq \Phi^+$,
we 
say that
a \emph{distribution of simple roots} on $A$
is 
a function $f\colon A\to \Pi$
such that 
$f(\alpha)\in \supp \alpha$ for each $\alpha\in A$.

The \emph{vector of D-multiplicities} of such a distribution is
the vector $n_\bullet \in \ZZ_{\ge 0}^r$ such that for each $\alpha_i \in \Pi$,
there are exactly $n_i$ roots $\alpha \in A$
such that $f(\alpha)=\alpha_i$ 
(``$f$ takes each value $\alpha_i$ exactly $n_i$ times'').
\end{definition}

\begin{definition}
We 
say that 
a \emph{configuration}
is
a pair $(A,n_{\bullet})$, 
where $A \subseteq \Phi^+$, $n_\bullet \in \ZZ_{\ge 0}^r$.

Given a subset $A \subseteq \Phi^+$ and subsets $I,J \subseteq \Pi$, we say that the 
\emph{restriction} $R_{I,J} (A)$ is $\{ \alpha \in A : \supp \alpha \cap I \ne \varnothing , \supp \alpha \cap J = \varnothing \}$.
We will often need restrictions with $J = \varnothing$ (in other words, with 
the second condition $\supp \alpha \cap J = \varnothing$ being trivial), 
denote $R_{I} (A) = R_{I,\varnothing} (A)$.

Finally, we call a configuration $(A,n_{\bullet})$ \emph{excessive} (resp.\ \emph{strictly excessive})
if $A = R_{\supp n_{\bullet}} (A)$ (in other words, there are no roots in $A$ with supports entirely outside $\supp n_\bullet$), 
$|A| = |n_\bullet|$, and for each $I \subsetneq \supp n_\bullet$, $I \ne \varnothing$ we have $|R_I(A)| \ge |p_i(n_\bullet)|$
(resp.\ $|R_I(A)| > |p_i(n_\bullet)|$).
\end{definition}

\begin{lemma}[Generalized Hall's representative lemma]\label{hallgenrep}
Let $A_1,\ldots, A_k$ be several finite sets, and let $n_1,\ldots,n_k\in \NN$. 
Suppose that for each subset $I\subseteq \{1,\ldots,k\}$ one has 
$$
\left|\vphantom{\bigcup}\smash{\bigcup_{i\in I}A_i}\right|\ge \sum_{i\in I} n_i.
$$
Then one can choose elements $n_i$ elements in each set $A_i$ 
($a_{i,j} \in A_i$ for $1\le i\le n$, $1 \le j \le n_i$) 
so that all of the chosen elements are distinct ($a_{i,j} \ne a_{i',j'}$ if $i \ne i'$ or $j \ne j'$).
\end{lemma}
\begin{proof}
This lemma follows from the standard combinatorial Hall's representative lemma
applied to the collection of sets consisting of $A_1$ repeated $n_1$ times, $A_2$ repeated $n_2$ times, \ldots, $A_k$ repeated $n_k$ times.
\end{proof}

\begin{lemma}\label{hallintermsofrpre}
Let $A\subseteq \Phi^+$, let $n_\bullet \in \ZZ_{\ge 0}^r$.
The following conditions are equivalent.
\begin{enumerate}
\item\label{halltruer} 
The configuration $(A, n_\bullet)$ is 
excessive.
\item\label{distributionexistsr}
There exists a simple root distribution on $A$ with vector of D-multiplicities $n_\bullet$.
\end{enumerate}
\end{lemma}
\begin{proof}
First, note that for each $I\subseteq \Pi$, by definition of $R_I(A)$, we have
\begin{equation}\label{hallintermsofrformula}
R_I(A)=\bigcup_{\alpha_i \in I} R_{\{ \alpha_i \}}(A).
\end{equation}

$\ref{halltruer}\Rightarrow\ref{distributionexistsr}$.
In view of 
(\ref{hallintermsofrformula}),
the 
excessiveness condition
implies
the hypothesis of generalized Hall representative lemma (Lemma \ref{hallgenrep}) 
applied to the $|\supp n_\bullet|$ sets: $R_{\{ \alpha_i \}}(A)$ for each $\alpha_i \in \supp n_\bullet$.
And Lemma \ref{hallgenrep} says that for each $\alpha_j \in \supp n_\bullet$, 
we can choose $n_j$ elements of $R_{\{ \alpha_j \}}(A)$, i.~e., $n_j$ roots $\alpha\in A$
such that $\alpha_j\in \supp\alpha$, and all chosen roots (for all simple roots $\alpha_j \in \supp n_\bullet$ together) 
are different.

In total, we chose $|n_\bullet|$ roots, and, by the definition an excessive configuration,
$|n_\bullet|=|A|$.
So, each root from $A$ was chosen exactly once, and we can set 
$f(\alpha)=\alpha_j$ if $\alpha$ was chosen as an element of $R_{\{ \alpha_j \}}(A)$. This is a simple root 
distribution on $A$, and it clearly has vector of D-multiplicities $n_\bullet$.

$\ref{distributionexistsr}\Rightarrow\ref{halltruer}$.
Let $f$ be a simple root distribution on $A$ with vector of D-multiplicities $n_\bullet$. 
The equalities $A = R_{\supp n_\bullet}(A)$ and $|A| = |n_\bullet|$ follow directly from the
definition of the vector of D-multiplicities of a distribution.

Also,
for each $\alpha_i \in \supp n_\bullet$, we have
$n_i=|f^{-1}(\alpha_i)|$ and $f^{-1}(\alpha_i)\subseteq R_{\{ \alpha_i \}}(A)$.
So, for each 
$I\subseteq \supp n_\bullet$, 
we have
\begin{equation*}
|p_I(n_\bullet)| = \sum_{\alpha_i \in I}n_i=\left|\vphantom{\bigcup}\smash{\bigcup_{\alpha_i \in I}f^{-1}(\alpha_i)}\right|,
\text{\qquad and (by (\ref{hallintermsofrformula}) again) \qquad}
\bigcup_{\alpha_i \in I}f^{-1}(\alpha_i)\subseteq R_I(A).
\end{equation*}
Therefore, 
$|p_I(n_\bullet)| \le |R_I(A)|$.
\end{proof}

Finally, we are ready to formulate the main result of this section.

\begin{proposition}\label{sortabilitycriterium}
Let $G$ be a split semisimple simply connected algebraic group of arbitrary type (not necessarily $A$, $D$, or $E$).
Let $w \in W$, let $n_{\bullet} \in \ZZ_{\ge 0}^r$. 
If $\ell(w) \ne |n_{\bullet}|$, then $c_{w,n_\bullet} = 0$. 
If $\ell(w) = |n_{\bullet}|$, 
then the following conditions are equivalent.
\begin{enumerate}
\item\label{cgreater0} $c_{w,n_\bullet} > 0$.
\item\label{labeledsortingexists} 
There exists an
enhanced path 
from $1_W$ to $w$
in the Bruhat graph such that 
the vector of D-multiplicities of the enhancement is $n_{\bullet}$.
\item\label{labeledsortingexistforall}
For every enhancement $\alpha_{i_1}, \ldots, \alpha_{i_l}$ with vector of D-multiplicities $n_\bullet$, 
there exists an
enhanced path from $1_W$ to $w$
with enhancement $\alpha_{i_1}, \ldots, \alpha_{i_l}$.
\item\label{distributionexists} 
There exists a simple root distribution on $\Phi^+\cap w\Phi^-$ with vector of D-multiplicities $n_{\bullet}$.
\item\label{excessiveconfiguration}
The configuration $(\Phi^+\cap w\Phi^-, n_\bullet)$ is 
excessive.
\end{enumerate}
If these conditions fail, then $c_{w,n_\bullet} = 0$.
\end{proposition}
\begin{proof}
$\ref{cgreater0}\Leftrightarrow \ref{labeledsortingexists} \Leftrightarrow \ref{labeledsortingexistforall}$. 
This follows directly from Corollary \ref{pieriformulamultiple2}.

$\ref{labeledsortingexists}\Rightarrow \ref{distributionexists}$.
Suppose that there exists 
an 
enhanced path 
$(1_W \xrightarrow{\beta_1} \ldots \xrightarrow{\beta_l} w, \alpha_{i_1}, \ldots, \alpha_{i_l})$
in the Bruhat graph such that 
the vector of D-multiplicities of the enhancement is $n_{\bullet}$.
By the definition of an enhanced path, we have
$\alpha_{i_k} \in \supp \beta_k$ for $1 \le k \le l$.

By Lemma \ref{antisimplelabelsisphiwphi} (\ref{antisimplepathinequality}), 
there exists 
a bijection 
$h \colon \{ 1, \ldots, l\} \to \Phi^+ \cap w \Phi^-$ such that 
$\beta_k \preceq h(k)$
for $1 \le k \le l$. 
For each $\alpha \in \Phi^+ \cap w \Phi^-$, set $f(\alpha) = \alpha_{i_{h^{-1}(\alpha)}}$.
Then $f(\alpha) \in \supp \beta_{h^{-1}(\alpha)} \subseteq \supp \alpha$, so $f$ is a distribution of simple roots on $\Phi^+ \cap w \Phi^-$.
And the number of times each $\alpha_j \in \Pi$ occurs as $f(\alpha)$ for $\alpha \in \Phi^+ \cap w \Phi^-$
equals the number of times $j$ occurs among $i_1, \ldots, i_l$, because $h$ is a bijection.
In other words, the vectors of D-multiplicities of the enhancement $\alpha_{i_1}, \ldots, \alpha_{i_l}$
and of the distribution $f$ coincide.

$\ref{distributionexists}\Rightarrow \ref{labeledsortingexists}$.
Let $f\colon \Phi^+\cap w\Phi^-\to \Pi$ be
a simple root distribution on $\Phi^+\cap w\Phi^-$
with vector of D-multiplicities $n_\bullet$.

By Lemma \ref{antisimplelabelsisphiwphi} (\ref{antisimplepathexists}), 
there exists an antisimple path from $1_W$ to $w$ in the Bruhat graph, denote it by
$1_W \xrightarrow{\beta_1} \ldots \xrightarrow{\beta_l} w$.
Then Lemma \ref{antisimplelabelsisphiwphi} (\ref{antisimplepathinequality})
also says that
$\{\beta_1,\ldots,\beta_{\ell(w)}\}=\Phi^+\cap w\Phi^-$.
Let us define an enhancement $\alpha_{i_1}, \ldots, \alpha_{i_l}$
by setting $\alpha_{i_k} = f(\beta_k)$.
By the definition of a distribution, $f(\beta_k) \in \supp \beta_k$, 
so we really get an enhanced path, not just a pre-enhanced one.
The vectors of D-multiplicities of $\alpha_{i_1}, \ldots, \alpha_{i_l}$ and of $f$ also obviously coincide.

$\ref{distributionexists} \Leftrightarrow \ref{excessiveconfiguration}$.
This follows directly from Lemma \ref{hallintermsofrpre}
applied to $A = \Phi^+ \cap w \Phi^-$.

Finally, the last equality $c_{w,n_\bullet} = 0$ (if the 5 equivalent conditions fail) follows from the previously known fact 
that $c_{w,n_\bullet}$ cannot be negative (or from Corollary \ref{pieriformulamultiple2}: 
it is clear from the definition of a Z-multiplicity that it cannot be negative).
\end{proof}

\begin{corollary}\label{anydistributiongives}
Let $w \in W$, let $n_{\bullet} \in \ZZ_{\ge 0}^r$. 
Suppose that the 
5
equivalent conditions in the statement of Proposition \ref{sortabilitycriterium}
hold (or just one of them holds). 
Then the path from equivalent condition \ref{labeledsortingexists} can actually be made antisimple.

In other words and moreover, 
let $f \colon \Phi^+ \cap w \Phi^- \to \Pi$ be a distribution 
of simple roots with vector of D-multiplicities $n_{\bullet}$.
Then there exists an \emph{antisimple} 
enhanced path 
of the form $(1_W \xrightarrow{\beta_1} \ldots \xrightarrow{\beta_l} w, f(\beta_1), \ldots, f(\beta_l))$
(by Lemma \ref{antisimplelabelsisphiwphi}, $\{ \beta_1, \ldots, \beta_l \} = \Phi^+ \cap w \Phi^-$, 
so it makes sense to speak about $f(\beta_j)$; clearly, the vector of D-multiplicities of 
the enhancement
is then $n_\bullet$).\qed
\end{corollary}

\begin{sideremark}\label{listoflabelscannotfindantisimple}
In Proposition \ref{sortabilitycriterium}, if $c_{w, n_\bullet} > 0$, 
we construct an antisimple 
enhanced path with a prescribed vector of D-multiplicities of the enhancement.
Or, we construct such a path 
with the enhancement coinciding with a prescribed distribution of simple roots on $\Phi^+ \cap w \Phi^-$
(in Corollary \ref{anydistributiongives}, we speak about this kind of path).
However, if we start with fixing a sequence $\alpha_{i_1}, \ldots, \alpha_{i_l}$ of simple roots 
(with vector of D-multiplicities $n_\bullet$), it is not always possible to have an antisimple 
enhanced path 
from $1_W$ to $w$ with this enhancement (in this order). 
In other words, we cannot say ``antisimple enhanced path'' in condition \ref{labeledsortingexistforall} of Proposition \ref{sortabilitycriterium}.

For example, if we take $G=SL_3$, $w=[3,1,2]$ and $n_\bullet=(1,1)$, then $\Phi^+ \cap w \Phi^-=\{ \alpha_1, \alpha_1 + \alpha_2 \}$, 
and we see immediately from Proposition \ref{sortabilitycriterium} (\ref{distributionexists}) that $c_{w,n_\bullet} > 0$.
However, if we want to use the sequence $\alpha_2,\alpha_1$ as the enhancement, then 
a path from equivalent condition \ref{labeledsortingexistforall} exists (it is 
$1_W \xrightarrow{\alpha_2} [1,3,2] \xrightarrow{\alpha_1} w$), but it
is really not antisimple.
(Note 
that in path, all labels on edges actually coincide with the enhancement, but is by far not universal: 
an enhancement is always a sequence of simple roots, 
and edges in the Bruhat graph can be labeled by arbitrary positive roots as long as the corresponding reflection changes the length by 1.)
%
\end{sideremark}

Before the end of this section, let us prove one more lemma, about equivalent characterizations of antisimple 
edges.
We will only need it in the next sections, where we are going to assume that all roots have equal 
lengths (type $A$, $D$, or $E$). But the proof is simple enough without the equal lengths assumptions, and the lemma itself may be of independent interest.
It implies, for example, that even without the simply-laced assumption, an antisimple 
edge
is the only case when $\psi_{\alpha}$ is the identity map.

\begin{lemma}\label{antisimpleisminimal}
Let $w\in W$, $\alpha\in \Phi^+\cap w\Phi^-$.
\kpar
The following conditions are equivalent:
\begin{enumerate}
\item\label{alphaantisimple} $w^{-1}\alpha\in -\Pi$.
\item\label{alphaminimal} $\alpha$ is a maximal 
element of 
the set $\Phi^+\cap w\Phi^-$ 
with respect to the order $\prec_w$.
\item\label{alphaminimalsimplecrit} It is impossible to find roots $\beta,\gamma \in \Phi^+\cap w\Phi^-$ 
such that $\alpha=\beta+\gamma$,
and
it is impossible to find a root $\beta\in \Phi^+\cap w\Phi^-$ such that
$\beta-\alpha\in \Phi^+\cap w\Phi^+$.
\item\label{alphaminimalbcfgcrit} It is impossible to find roots $\beta,\gamma \in (\Phi^+\cap w\Phi^-) \setminus \{ \alpha \}$ 
such that $\gamma = -\sigma_\alpha \beta$,
and it is impossible to find a root $\beta\in \Phi^+\cap w\Phi^-$ such that
$(\beta,\alpha) > 0$ 
and $\sigma_\alpha \beta \in \Phi^+\cap w\Phi^+$.
\end{enumerate}
\end{lemma}
\begin{proof}
$\ref{alphaantisimple}\Rightarrow \ref{alphaminimal}$.
Let $\alpha\in \Phi^+\cap w(-\Pi)$.
Assume that $\beta\in \Phi^+\cap w\Phi^-$, $\alpha \prec_w\beta$.
Then, by the definition of $\prec_w$, $w^{-1}\alpha\prec w^{-1}\beta$. 
But $w^{-1}\beta\in \Phi^-$, $w^{-1}\alpha\in -\Pi$, 
a contradiction.

$\ref{alphaminimal}\Rightarrow \ref{alphaminimalsimplecrit}$.
Let $\alpha$ be a maximal element of $\Phi^+\cap w\Phi^-$ with respect to $\prec_w$.
If there exist $\beta,\gamma \in \Phi^+\cap w\Phi^-$ 
such that $\alpha=\beta+\gamma$ or 
if there exists $\beta\in \Phi^+\cap w\Phi^-$ such that
$\beta-\alpha\in \Phi^+\cap w\Phi^+$,
then 
$w^{-1}\beta-w^{-1}\alpha\in \Phi^+$ (as $-w^{-1} \gamma$ or directly).
So, $\alpha\prec_w \beta$, a contradiction with maximality of $\alpha$.


$\ref{alphaminimalsimplecrit}\Rightarrow \ref{alphaantisimple}$.
Assume that $w^{-1}\alpha\notin -\Pi$.
Then, since $w^{-1}\alpha\in \Phi^-$, by \cite[\S 10.1, Theorem$'$]{humps}, it is possible to decompose $w^{-1}\alpha=\beta'+\gamma'$, 
where $\beta',\gamma' \in \Phi^-$.
We have $w\beta'+w\gamma'=\alpha \in \Phi^+$.
The roots $w\beta'$ and $w\gamma'$ cannot be both negative.
Without loss of generality, suppose $w\beta'\in \Phi^+$.
Set $\beta=w\beta'$, $\gamma=w\gamma'$.

If $\gamma \in \Phi^-$, then 
$\beta - \alpha = -\gamma \in \Phi^+\cap w\Phi^+$,
a contradiction.
If $\gamma\in \Phi^+$, then $\beta,\gamma \in \Phi^+\cap w\Phi^-$ and $\alpha=\beta+\gamma$,
a contradiction again.

$\ref{alphaantisimple}\Rightarrow \ref{alphaminimalbcfgcrit}$: a direct proof using properties of roots would not be very easy, especially in type $G_2$, 
let us use Lemmas \ref{antisimpleexistsadmissible} and \ref{admissibilitycrit} instead, once we already know them (with the help of \cite{bgg}).
Let $\alpha\in \Phi^+\cap w(-\Pi)$.
By Lemma \ref{antisimpleexistsadmissible} (\ref{antisimpleadmissible}), $\sigma_\alpha w \xrightarrow{\alpha} w$. 

If there exist $\beta, \gamma \in (\Phi^+ \cap w \Phi^-) \setminus \{ \alpha \}$ such that $\gamma = -\sigma_\alpha \beta$, 
then we can write $\beta + \gamma = \langle \beta, \alpha \rangle \alpha$. Moreover, $\langle \beta, \alpha \rangle > 0$
since $\alpha, \beta, \gamma \in \Phi^+$. We have a contradiction with the equivalent condition for $\sigma_\alpha w \xrightarrow{\alpha} w$
from Lemma \ref{admissibilitycrit}. And if there exists $\beta \in \Phi^+ \cap w \Phi^-$ such that $(\beta,\alpha) > 0$ and 
$\sigma_\alpha \beta \in \Phi^+\cap w\Phi^+$, then $\psi_{w,\alpha}$ is not the identity map, which also contradicts 
Lemma \ref{admissibilitycrit} for $\alpha \in w(-\Pi)$.

$\ref{alphaminimalbcfgcrit}\Rightarrow \ref{alphaantisimple}$.
Assume that $-w^{-1} \alpha \notin \Pi$, denote $\alpha' = -w^{-1} \alpha \in \Phi^+$. 
By the proof of \cite[Lemma 10.2 A]{humps}, there exists a simple root $\alpha_i$ such that $(\alpha', \alpha_i) > 0$.
Then $k := \langle \alpha_i, \alpha' \rangle > 0$, and $\sigma_{\alpha'} \alpha_i = \alpha_i - k \alpha'$ cannot be 
a positive root. So, $\sigma_{\alpha'} \alpha_i \in \Phi^-$. Then $-w \alpha_i, w \sigma_{\alpha'} \alpha_i \in w \Phi^-$, 
and the operator $-\sigma_\alpha$ interchanges these two roots.

We also have $-w \alpha_i + w \sigma_{\alpha'} \alpha_i = k \alpha$, 
so the roots $-w \alpha_i$ and $w \sigma_{\alpha'} \alpha_i$ cannot be both negative. 
Choose a positive root from $\{ -w \alpha_i, w \sigma_{\alpha'} \alpha_i \}$ and denote it by $\beta$, denote the other of these two roots by $\gamma$.
Then $\beta \in \Phi^+ \cap w \Phi^-$, $\gamma \in w \Phi^-$, and $-\sigma_\alpha$ interchanges $\beta$ and $\gamma$.

If $\gamma \in \Phi^-$, then $-\gamma = \sigma_\alpha \beta \in \Phi^+ \cap w \Phi^+$, a contradiction.
If $\gamma \in \Phi^+$, then we also get a contradiction with the assumption of condition \ref{alphaminimalbcfgcrit}
once we check that $\beta \ne \alpha$, $\gamma \ne \alpha$.
But $-\sigma_\alpha$ interchanges $\beta$ and $\gamma$, so if one of the roots $\beta$ and $\gamma$ equals $\alpha$, 
then the other root also equals $\alpha$. Then in particular $-w \alpha_i = \alpha$, a contradiction with $-w^{-1} \alpha \notin \Pi$.
\end{proof}

\section{Necessary conditions for $c_{w,n_\bullet}=1$}\label{sectionuniqsortnes}

In this section, we are going to find some necessary conditions for $c_{w,n_\bullet}=1$.
Later we will prove that all these conditions together are also sufficient for $c_{w,n_\bullet}=1$.
Starting from this section, 
\emph{we always assume that the Dynkin diagram is simply-laced ($G$ is of type $A$, $D$, or $E$, all roots have equal lengths)}.
In fact, some intermediate results will also hold for all types of $G$, but the proofs for arbitrary types are more complicated, 
and the intermediate results don't look very important on their own.

\subsection{Basic necessary conditions for $c_{w,n_\bullet}=1$}
We start with a few basic lemmas. For a while, we are going to use the following lemma for paths consisting of a single edge only, 
but at some point we will need it in the whole generality.

%


\begin{lemma}\label{sortingprefixmulttwopropagates}
Let $v,w \in W$, $m_\bullet \in \ZZ_{\ge 0}^r$.
Suppose that there exists an 
enhanced path in the Bruhat graph from $v \in W$ to $w \in W$
such that the vector of D-multiplicities of the enhancement equals $m_\bullet$.
Let $n_\bullet \in \ZZ_{\ge 0}^r$, $|n_\bullet|=\ell(v)$.
Then $c_{w,n_\bullet + m_\bullet}\ge c_{v, n_\bullet}$.

In particular, if 
$c_{v, n_\bullet} \ge 2$,
then 
$c_{w,n_\bullet + m_\bullet} \ge 2$.
\end{lemma}
\begin{proof}
%
Follows directly from Corollary \ref{pieriformulacwnmultiple}.
\end{proof}
Unfortunately, we cannot claim here that $c_{w, n_\bullet + m_\bullet}$ is divisible by $c_{v, n_\bullet}$:
in (\ref{pieriformulacwnmultiplerel}),
it is not necessarily true that all paths in the sum begin at the same vertex $v$. 

\begin{corollary}\label{antireducedsortingprefixmulttwopropagates}
Let $w\in W$, 
let $n_\bullet\in \ZZ_{\ge 0}^r$, $|n_\bullet|=\ell(w)$.
\kpar
Let $f$ be a simple root distribution on $\Phi^+\cap w\Phi^-$ with vector of D-multiplicities 
$n_\bullet$.
\kpar
Let $v \xrightarrow{\beta_1} \ldots \xrightarrow{\beta_l} w$ be an antisimple path in the Bruhat graph.
\kpar
Let $n'_\bullet$ be the vector of D-multiplicities of the sequence $f(\beta_1), \ldots, f(\beta_l)$ of simple roots 
($f$ is defined on $\beta_i$s since $\{\beta_1, \ldots, \beta_k\} \subseteq \Phi^+\cap w\Phi^-$ 
by Lemma \ref{antisimplelabelsisphiwphi} (\ref{antisimplepathinequality})).
\kpar
Then $c_{w,n_\bullet}\ge c_{v,n_\bullet-n'_\bullet}>0$.

In particular, if 
$c_{v,n_\bullet-n'_\bullet}\ge 2$,
then 
$c_{w,n_\bullet}\ge 2$.
\end{corollary}
\begin{proof}
By Lemma \ref{antisimplelabelsisphiwphi} (\ref{antisimplepathinequality}), 
$\Phi^+\cap w\Phi^- = (\Phi^+\cap v\Phi^-) \sqcup \{\beta_1, \ldots, \beta_l\}$.
So, it is possible to restrict of $f$ onto $\Phi^+\cap v\Phi^-$, and the vector of D-multiplicities of this restriction is $n_\bullet - n'_\bullet$.
The fact that $c_{v,n_\bullet - n'_\bullet} > 0$ 
follows from the existence of $f|_{\Phi^+\cap v\Phi^-}$ and Proposition \ref{sortabilitycriterium} ($\ref{distributionexists} \Rightarrow \ref{cgreater0}$).

For the rest of the claim,
we have an enhanced path 
$(v \xrightarrow{\beta_1} \ldots \xrightarrow{\beta_l} w, f(\beta_1), \ldots, f(\beta_l))$ 
(by the definition of a distribution of simple roots, $f(\beta_i) \in \supp \beta_i$, so this is really an enhanced path),
and the claim follows from Lemma \ref{sortingprefixmulttwopropagates}.
\end{proof}

From now on, we are going to pay particular attention to strictly excessive configurations for a while. 
We start with the following nice property.
\begin{lemma}[``free first choice'']\label{excessiveisfreechoice}
Let $w \in W$ and $n_\bullet \in \ZZ_{\ge 0}^r$ be such that the configuration $(\Phi^+ \cap w \Phi^-, n_\bullet)$
is strictly excessive.
Let $\alpha \in \Phi^+ \cap w \Phi^-$, let $\alpha_i \in \supp \alpha \cap \supp n_\bullet$.
Then there exists a distribution of simple roots $f$ on $\Phi^+ \cap w \Phi^-$ with vector of D-multiplicities $n_\bullet$
and such that $f(\alpha) = \alpha_i$.
\end{lemma}
\begin{proof}
Set $A=(\Phi^+\cap w\Phi^-)\setminus \{ \alpha \}$.
Let us check that the configuration $(A, n_\bullet - e_i)$ is 
excessive 
(but not necessarily strictly excessive; recall that $e_i$ denotes the $i$th standard basis vector of the lattice $\ZZ^r$).

Let $I\subseteq \supp (n_\bullet - e_i) \subseteq \supp n_\bullet$. 
Clearly, $|p_I (n_\bullet)|\ge |p_I(n_\bullet - e_i)|$ and $|R_I(A)|\ge |R_I(\Phi^+ \cap w \Phi^-)|-1$. 
If $I\ne \supp n_\bullet$, $I \ne \varnothing$, then 
$$
|R_I(A)|\ge |R_I(\Phi^+ \cap w \Phi^-)|-1>|p_I (n_\bullet)|-1\ge |p_I(n_\bullet - e_i)|-1.
$$
Since all number here are integers, we get
$|R_I(A)|\ge |p_I(n_\bullet - e_i)|$.
If $I= \supp n_\bullet$, then $|p_I(n_\bullet - e_i)|=|p_I (n_\bullet)|-1$, and 
$$
|R_I(A)| = |R_I(\Phi^+ \cap w \Phi^-)|-1 = |p_I (n_\bullet)|-1= |p_I(n_\bullet - e_i)|.
$$
If $I=\varnothing$, then $|R_I(A)|= |p_I(n_\bullet - e_i)|=0$.
So, for all $I\subseteq \supp (n_\bullet - e_i)$ we have $|R_I(A)|\ge |p_I(n_\bullet - e_i)|$.

For $I = \supp (n_\bullet - e_i)$ (whether it equals $\supp n_\bullet$ or not) we get 
$|R_I(A)|\ge |p_I(n_\bullet - e_i)|= |n_\bullet - e_i| = |n_\bullet| - 1 = |A|$ (the last equality follows from the strict excessiveness 
of $(\Phi^+ \cap w \Phi^-, n_\bullet)$).
But $R_I(A) \subseteq A$, so in fact $R_I(A) = A$, which completes the check that $(A, n_\bullet - e_i)$
is 
an excessive configuration.

By Lemma \ref{hallintermsofrpre}, there exists a simple root distribution $g$ on $A$ 
with vector of D-multiplicities $n_\bullet - e_i$.
Set $f(\alpha)=\alpha_i$ and $f(\beta)=g(\beta)$ for $\beta \in A$. 
\end{proof}

The following definition may look trivial for now, but later we will need it in exactly this form.
\begin{definition}
Let $(A, n_\bullet)$ be a configuration.
We say that it \emph{has large essential coordinates} if there exists
$\alpha\in A$ and $\alpha_i\in \supp n_\bullet$ 
such that $\alpha_i^*(\alpha) \ge 2$. 
Otherwise we say that $(A, n_\bullet)$ \emph{has small essential coordinates}.
\end{definition}

\begin{lemma}\label{freechoicetwo}
Let $w\in W$ and $n_\bullet \in \ZZ_{\ge 0}^r$ be such that 
the configuration $(\Phi^+ \cap w \Phi^-, n_\bullet)$ is strictly excessive and has large essential coordinates.
Then $c_{w,n_\bullet}\ge 2$.
\end{lemma}
\begin{proof}
Let
$\alpha\in \Phi^+\cap w\Phi^-$ and $\alpha_i\in \supp n_\bullet$ be
such that $\alpha_i^*(\alpha) \ge 2$. 
By Lemma \ref{excessiveisfreechoice},
there exists a 
simple root distribution $f$ on $\Phi^+\cap w\Phi^-$
with vector of D-multiplicities $n_\bullet$ and
such that $f(\alpha)=\alpha_i$.

By Corollary \ref{anydistributiongives}, there exists 
an antisimple enhanced path 
of the form $(1_W \xrightarrow{\beta_1} \ldots \xrightarrow{\beta_l} w, f(\beta_1), \ldots, f(\beta_l))$
(by Lemma \ref{antisimplelabelsisphiwphi}, $\{ \beta_1, \ldots, \beta_l \} = \Phi^+ \cap w \Phi^-$, 
so it makes sense to speak about $f(\beta_j)$).
In particular, the edge $\xrightarrow{\alpha}$, wherever it appears in the path, has enhancement $f(\alpha)=\alpha_i$.
The Z-multiplicity of this path is positive since it is enhanced, not just pre-enhanced, and is divisible by $\alpha_i^*(\alpha) \ge 2$.
The sum of Z-multiplicities of all enhanced paths from $1_W$ to $w$ with enhancement 
$f(\beta_1), \ldots, f(\beta_l)$
is therefore at least 2 and (on the other hand) equals $c_{w,n_\bullet}$ by Corollary \ref{pieriformulamultiple2}.
\end{proof}

In this lemma, it is really important to have a strictly excessive configuration.
Of course, it follows from the proof that we could replace the strict excessiveness condition 
with the existence of a simple root distribution $f$ such that $f(\alpha)=\alpha_i$.
But we cannot remove the strict excessiveness condition entirely or even replace it with 
(not necessarily strict) excessiveness, 
as we will see in Example \ref{largeessntialoutsideinvolved} later.

We conclude this subsection with a corollary about pairs of roots with scalar product $-1$.
For a necessary condition for $c_{w,n_\bullet}=1$, we could leave the ``moreover'' part only, but we will need 
the first part, without the strictly excessiveness assumption, on its own as an intermediate step a few times later.
\begin{corollary}\label{freechoiceonetwozero}
Let $w \in W$, let $n_\bullet \in \ZZ_{\ge 0}^r$.
Suppose that there exist roots $\alpha,\beta\in \Phi^+\cap w\Phi^-$ 
such that $(\alpha,\beta)=-1$
and $\supp \alpha \cap \supp \beta \cap \supp n_\bullet \ne \varnothing$.
Then the configuration $(\Phi^+ \cap w \Phi^-, n_\bullet)$ has large essential coordinates.

Moreover, if, additionally, the configuration $(\Phi^+ \cap w \Phi^-, n_\bullet)$ is strictly excessive, then $c_{w,n_\bullet} \ge 2$.
\end{corollary}
\begin{proof}
We have $(\alpha,\beta)=-1$, so $\alpha+\beta\in \Phi$ by Lemma \ref{sumexists}, and further,
$\alpha+\beta \in \Phi^+ \cap w \Phi^-$ since both $\alpha$ and $\beta$ are in $\Phi^+ \cap w \Phi^-$.
Let $\alpha_i \in \supp \alpha \cap \supp \beta \cap \supp n_\bullet$. 
Then $\alpha_i^*(\alpha) \ge 1$, $\alpha_i^*(\beta) \ge 1$, so 
$\alpha_i^*(\alpha + \beta) \ge 2$, and the configuration $(\Phi^+ \cap w \Phi^-, n_\bullet)$ has large essential coordinates.
The ``moreover'' part now follows from Lemma \ref{freechoicetwo}.
\end{proof}

\subsection{Orthogonal roots in an excessive configuration and $c_{w,n_\bullet}=1$}
Our goal in this subsection is to prove the following 
proposition about pairs of orthogonal roots in $\Phi^+ \cap w \Phi^-$.
\begin{proposition}\label{freechoiceorthogonalnooverlap}
Let $G$ be a semisimple split algebraic group with simply laced Dynkin diagram.
Let $w \in W$ and $n_\bullet \in \ZZ_{\ge 0}^r$ be such that 
$(\Phi^+ \cap w \Phi^-, n_\bullet)$ is a strictly excessive configuration.
If there exist two roots $\alpha, \beta \in \Phi^+ \cap w \Phi^-$ such that
$(\alpha, \beta)=0$ and $\supp\alpha \cap \supp\beta \cap \supp n_\bullet \ne \varnothing$,
then $c_{w,n_\bullet} \ge 2$.
\end{proposition}
%
Roughly speaking, we are going to prove this proposition by induction on $\ell(w)$ 
using Corollary \ref{antireducedsortingprefixmulttwopropagates}
for single antisimple edges.

So far, we paid most attention to antisimple edges and paths, because passing along them changes the set $\Phi^+ \cap w \Phi^-$ in a simple way.
We are still going to use antisimple paths a lot later, 
but now, for what is going to be the base of our induction, we will need a broader class of edges.
We have the following definition.
\begin{definition}
Let $w \in W$, and let $f \colon \Phi^+\cap w\Phi^- \to \Pi$ be a distribution of simple roots.
An edge $\sigma_\alpha w \xrightarrow{\alpha} w$ in the Bruhat graph
is called 
\emph{compatible with $f$}
if 
for every $\beta\in \Phi^+\cap w\Phi^-$ such that $\beta-\alpha\in \Phi^+\cap w\Phi^+$, 
we have
$f(\beta) \notin \supp\alpha$.
\end{definition}
\begin{remark}\label{antisimplecompatible}
If for some $w \in W$ we have $\alpha \in -w\Pi$, then it follows directly from Lemma \ref{antisimpleexistsadmissible} (\ref{antisimpleadmissible})
and
Lemma \ref{antisimpleisminimal} (\ref{alphaminimalsimplecrit})
that $\sigma_\alpha w \xrightarrow{\alpha} w$ is an edge in the Bruhat graph compatible with any distribution on $\Phi^+ \cap w \Phi^-$.
\end{remark}
\begin{example}\label{std4examplecompatible}
Let $G=SL_4$, $w = [3,4,1,2]$.
Then $\Phi^+ \cap w \Phi^- = \{\alpha_2, \alpha_1+\alpha_2, \alpha_2+\alpha_3, 
\alpha_1+\alpha_2+\alpha_3\}$. 
Let $f(\alpha_2) = f(\alpha_1+\alpha_2+\alpha_3) = \alpha_2$, 
$f(\alpha_1+\alpha_2) = \alpha_1$, and $f(\alpha_2+\alpha_3) = \alpha_3$.
By Example \ref{admissibleexample}, $[3,1,4,2] \xrightarrow{\alpha_2} [3,4,1,2]$ is an edge in the Bruhat graph.
Let us check that it is compatible with $f$.
Indeed, 
there are two roots $\beta \in \Phi^+\cap w\Phi^-$ such that $\beta-\alpha_2 \in \Phi^+\cap w\Phi^+$:
$\beta = \alpha_1+\alpha_2$ and $\beta=\alpha_2+\alpha_3$. For each of them, we have $f (\beta) \notin \supp \alpha_2$.

Also, by Example \ref{antisimplesln}, $\alpha_1+\alpha_2+\alpha_3 \in -w\Pi$, 
$\alpha_1+\alpha_2+\alpha_3 = -w (\alpha_2)$. So, 
$[2,4,1,3] \xrightarrow{\alpha_1 + \alpha_2 + \alpha_3} [3,4,1,2]$ is compatible with $f$ by Remark \ref{antisimplecompatible}.
\end{example}
We will need the following property of edges compatible with distributions.
\begin{lemma}\label{compatibilitycorrect}
Let $w\in W$, 
let $f$ be a simple root distribution on $\Phi^+ \cap w \Phi^-$ with vector of D-multiplicities $n_\bullet$,
let $\alpha\in \Phi^+\cap w\Phi^-$ be a root 
such that $\sigma_\alpha w \xrightarrow{\alpha} w$ is an edge 
compatible with $f$,
and let $\alpha_i = f(\alpha)$.
Then $c_{\sigma_\alpha w, n_\bullet - e_i} \ge 1$.
\end{lemma}
\begin{proof}
The definitions of an $f$-compatible root and of $\psi_{\alpha}$ imply that 
$f(\beta) \in \supp \psi_\alpha(\beta)$ for every 
$\beta \in (\Phi^+ \cap w \Phi^-) \setminus \{ \alpha \}$.
Hence, by Lemma \ref{admissibilitycrit}, 
$f|_{(\Phi^+ \cap w \Phi^-) \setminus \{ \alpha \}} \circ \psi_\alpha^{-1}$
is a simple root distribution on $\Phi^+ \cap \sigma_\alpha w \Phi^-$ with vector of D-multiplicities $n_\bullet - e_i$.
By Proposition \ref{sortabilitycriterium} (\ref{distributionexists}), $c_{\sigma_\alpha w, n_\bullet - e_i} \ge 1$.
\end{proof}
The following lemma will work as the induction base in the proof of Proposition \ref{freechoiceorthogonalnooverlap}.
\begin{lemma}\label{twodiffcompatible}
Let $w\in W$, 
let $n_\bullet \in \ZZ_{\ge 0}$, $|n_\bullet|=\ell(w)$.
Let $\alpha,\beta\in \Phi^+\cap w\Phi^-$ be two different (but not necessarily orthogonal) roots
such that $\sigma_\alpha w \xrightarrow{\alpha} w$, $\sigma_\beta w \xrightarrow{\beta} w$.

Suppose that there exist
two (not necessarily different) simple root distributions $f,g\colon \Phi^+\cap w\Phi^-\to \Pi$, 
both with vector of D-multiplicities $n_\bullet$,
such that $f(\alpha)=g(\beta)$ and
$\sigma_\alpha w \xrightarrow{\alpha} w$ (resp.\ $\sigma_\beta w \xrightarrow{\beta} w$) is 
compatible with $f$ (resp.\ with $g$).
Then $c_{w,n_\bullet}\ge 2$.
\end{lemma}
\begin{proof}
Denote $\alpha_i=f(\alpha)=g(\beta)$.
By Lemma \ref{compatibilitycorrect}, we have $c_{\sigma_\alpha w, n_\bullet - e_i} \ge 1$ and $c_{\sigma_\beta w, n_\bullet - e_i} \ge 1$.
By the definition of a distribution, we have $\alpha_i^*(\alpha) > 0$, $\alpha_i^*(\beta)>0$. 
The claim follows from Corollary \ref{pieriformulacwnmultiple}.
\end{proof}
\begin{example}\label{std4example}
Let $G$, $w$, and $f$ be as in Example \ref{std4examplecompatible}, and let $n_\bullet=(1,2,1)$ be the vector of D-multiplicities of $f$. 
Then, on one hand, Lemma \ref{twodiffcompatible} applied to the two roots $\alpha_2$ and $\alpha_1+\alpha_2+\alpha_3$ mentioned in 
Example \ref{std4examplecompatible} shows that $c_{w,n_\bullet} \ge 2$ (in fact, $c_{w,n_\bullet} = 2$, but this requires more analysis).

On the other hand, there is only one antisimple edge ending at $w$: $w' := [2,4,1,3] \to [3,4,1,2]$.
The label at this edge is $\alpha_1 + \alpha_2 + \alpha_3$, so 
it even becomes an enhanced path if we choose any of the three simple roots $\alpha_1$, $\alpha_2$, or $\alpha_3$ as the enhancement.
But unfortunately, one can check that $c_{w', 0,2,1}=c_{w',1,1,1} = c_{w',1,2,0}=1$.
From here, it is easy to see that whatever sequence $L$ with vector of D-multiplicities $n_\bullet$ we try to use as the enhancement
for paths from $1_W$ to $w$, we will not get more than one enhanced (and not just pre-enhanced) \emph{antisimple} path 
with enhancement $L$. In other words, antisimple enhanced paths are not enough to prove that $c_{w,n_\bullet}\ge 2$ in this case,
we really need Lemma \ref{twodiffcompatible} and compatible edges.
\end{example}
For the induction step in the proof of Proposition \ref{freechoiceorthogonalnooverlap}, we will need a few preliminary technical lemmas.
\begin{lemma}\label{compatibleisminimalp}
Let $w\in W$, $\alpha\in \Phi^+\cap w\Phi^-$.
Let $f$ be simple root distribution 
on $\Phi^+\cap w\Phi^-$.
Then at least one of the following is true:
\begin{enumerate}
\item\label{compatibleisminimalpcomp} We have $\sigma_\alpha w \xrightarrow{\alpha} w$, and this edge is compatible with $f$.
\item\label{compatibleisminimalpmin} There exists $\beta\in \Phi^+\cap w\Phi^-$ 
such that $\alpha \prec_w\beta$, $(\alpha,\beta)=1$, $f(\beta)\in\supp \alpha$, 
and $f(\alpha)\in\supp \beta$.
\end{enumerate}
\end{lemma}
\begin{proof}
Suppose that condition \ref{compatibleisminimalpcomp} fails, let us prove condition \ref{compatibleisminimalpmin}.

First, suppose that there is no edge $\sigma_\alpha w \xrightarrow{\alpha} w$ in the Bruhat graph at all.
By Lemma \ref{admissibilitycrit},
there exist $\beta,\gamma\in \Phi^+\cap w\Phi^-$ such that 
$\beta+\gamma=\alpha$. We have $\alpha,\beta,\gamma\in \Phi^+$, so $\supp \alpha=\supp\beta \cup \supp\gamma$.
We also have
$f(\alpha)\in \supp\alpha$, so we may assume without loss of generality (after a possible interchange of $\beta$ and $\gamma$)
that $f(\alpha)\in\supp \beta$.

By Lemma \ref{sumexists}, 
$(\alpha,\beta)=1$. Since $\gamma=\alpha-\beta \in \Phi^+ \cap w\Phi^-$, 
we get $\alpha\prec_w\beta$ and $\beta\prec\alpha$. In particular, 
$\supp \beta\subseteq\supp \alpha$.
We also have $f(\beta)\in \supp\beta$, so $f(\beta)\in \supp \alpha$, and condition \ref{compatibleisminimalpmin} holds.

Now suppose that $\sigma_\alpha w \xrightarrow{\alpha} w$, but this edge is not compatible with $f$.
By definition, this means that there exists $\beta \in \Phi^+ \cap w \Phi^-$ such that 
$\beta - \alpha \in \Phi^+ \cap w \Phi^+$ and $f(\beta) \in \supp \alpha$.
It follows from $\beta - \alpha \in \Phi^+ \cap w \Phi^+$ that 
$\alpha\prec_w\beta$ and $\alpha \prec \beta$ (in particular, $\supp \alpha \subseteq \supp \beta$).
We have $f(\alpha) \in \supp \alpha$, so $f(\alpha) \in \supp \beta$.
Finally, by Lemma \ref{sumexists}, $(\alpha,\beta)=1$, and condition \ref{compatibleisminimalpmin} holds.
\end{proof}
In fact, conditions \ref{compatibleisminimalpcomp} and \ref{compatibleisminimalpmin} in this lemma are mutually exclusive, 
in other words, the failure of condition \ref{compatibleisminimalpmin} implies compatibility, but we will not need this.
\begin{lemma}\label{uniqueminimalsupportcovers}
Let $w\in W$. Suppose that $\Phi^+\cap w\Phi^-$ 
contains exactly one root $\alpha$ such that $w^{-1}\alpha\in -\Pi$.
Then for every $\beta\in \Phi^+\cap w\Phi^-$, 
we have $\supp \beta\subseteq \supp \alpha$.
\end{lemma}
\begin{proof}
%
%
%
%
Fix $\beta\in \Phi^+\cap w\Phi^-$.
Denote $w^{-1}\alpha=-\alpha_i$ and $w^{-1}\beta=-\sum_{1 \le j \le r} a_j\alpha_j$ (here, $a_j \ge 0$).

We have $w (-\alpha_i) \in \Phi^+$ and, by the uniqueness of $\alpha$, $w (-\alpha_j) \in \Phi^-$ for $j \ne i$. 
We can write 
\begin{equation*}
\beta = a_i w(-\alpha_i) + \Big( \sum_{j \ne i} a_j w (-\alpha_j) \Big),
\end{equation*}
then 
$\beta \preceq a_i w(-\alpha_i) = a_i \alpha$.
Therefore, $\supp \beta\subseteq \supp (a_i\alpha)=\supp \alpha$.
\end{proof}

\begin{lemma}\label{rhombuspreconstruction}
Let $w \in W$ and $n_\bullet \in \ZZ_{\ge 0}^r$ be such that $\ell(w)=|n_\bullet|$ and the configuration 
$(\Phi^+ \cap w \Phi^-, n_\bullet)$ has small essential coordinates. Let $\alpha \in \Phi^+ \cap w \Phi^-$, 
let $\alpha_i \in \supp \alpha \cap \supp n_\bullet$.
Suppose that the set $A=\{ \beta' \in \Phi^+\cap w\Phi^- : (\alpha, \beta') = 0, \alpha_i \in \supp \beta' \}$
is nonempty, and let $\beta$ be a $\prec_w$-maximal element of $A$. 
Finally, suppose that we have $\gamma \in \Phi^+ \cap w \Phi^-$ 
such that $(\gamma,\beta)=1$, $\alpha_i \in \supp \gamma$, and $\beta \prec_w \gamma$.

Then $(\alpha, \gamma)=1$.
\end{lemma}
\begin{proof}
First, we cannot have $(\alpha, \gamma)=0$, this would contradict the $\prec_w$-maximality of $\beta$.
Second, we have $\alpha_i \in \supp \alpha \cap \supp \gamma \supp n_\bullet$, so, by Corollary \ref{freechoiceonetwozero}, 
the equality $(\alpha, \gamma) = -1$ would contradict the assumption of small essential coordinates.
Third, we have 
$(\beta,\gamma)=1$
and
$(\beta,\alpha)=0$, so $\gamma\ne \alpha$ and $(\alpha, \gamma) \ne 2$. 
Finally, both $\alpha$ and $\gamma$ are positive roots, so 
$(\alpha, \gamma) \ne -2$. The only remaining possibility is $(\alpha, \gamma)=1$.
\end{proof}

\begin{lemma}\label{rhombusroot}
Let $\alpha,\beta,\gamma\in \Phi$ be such that 
$(\alpha,\beta)=0$,
$(\beta,\gamma)=1$,
$(\alpha,\gamma)=1$. Denote $\delta=\alpha-\gamma+\beta$. Then:
\begin{enumerate}
\item\label{rhombusrootexists} We have $\delta \in \Phi$,
and $(\alpha,\delta)=1$,
$(\beta,\delta)=1$,
$(\gamma,\delta)=0$.
\item\label{rhombuspositive} If there exists a simple root $\alpha_i$ 
such that $\alpha_i^*(\alpha)=\alpha_i^*(\beta)=\alpha_i^*(\gamma)=1$, 
then $\alpha_i^*(\delta)=1$ and $\delta \in \Phi^+$.
\item\label{rhombusinitinegative}
If ($\alpha\prec_w \gamma$ and $\beta \in w\Phi^-$) or ($\beta\prec_w \gamma$ and $\alpha \in w\Phi^-$),
then $\delta \in w\Phi^-$.
\end{enumerate}
\end{lemma}
\begin{proof}
Direct computation of scalar products. To see that $\delta \in \Phi$, we use Lemma \ref{sumexists}.
\end{proof}
\begin{sideremark}
In this lemma, we really need a group $G$ of type $A$, $D$, or $E$.
If, for example, $G$ were of type $B_2$, this lemma or its analogue doesn't work, as the following example shows.

Let $\Phi$ be a root system of type $B_2$. Recall that we enumerate simple roots as in \cite{bou}, so $\alpha_1$ is a long root,
and $\alpha_2$ is a short root. Set $\alpha = \alpha_1 + \alpha_2$, $\beta=\alpha_2$, $\gamma=\alpha_1+2\alpha_2$.
Then $(\alpha, \beta) = 0$, 
$(\alpha, \gamma) > 0$ (and even $(\alpha, \gamma)=1$ if we choose the scalar product so that the lengths of all short roots equal 1), 
and 
$(\beta, \gamma) > 0$. But $\delta = \alpha+\beta - \gamma = 0 \notin \Phi$.

Even if we replace the formula for $\delta$ with $\delta = - \sigma_\beta \sigma_\alpha \gamma$
(for $G$ of type $A$, $D$, or $E$, this equals $\alpha+\beta -\gamma$), then 
we will of course get $\delta \in \Phi$, but in fact
(in the same example of type $B_2$), we will get $\delta=\gamma$, so, for example, the formula $(\gamma, \delta)=0$ still fails.

And later we will see that the whole Proposition \ref{freechoiceorthogonalnooverlap} is also wrong in type $B_2$, see Example \ref{b2orthogonalexample}.
\end{sideremark}
As we are going to use Corollary \ref{antireducedsortingprefixmulttwopropagates} in the induction step, 
we will need distributions of simple roots, not just configurations.
The following definition explains a property of such a distribution we are going to maintain.
\begin{definition}
Let $w\in W$.
We 
say that
a 
simple root distribution $f$ on $\Phi^+\cap w\Phi^-$
is
\emph{flexible}
if
there exist roots $\alpha,\beta\in \Phi^+\cap w\Phi^-$
such that $(\alpha,\beta)=0$, $f(\beta)\in \supp\alpha$, and $f(\alpha)\in \supp \beta$.
\end{definition}
Sometimes we will also need to modify the distribution of simple roots before we can do an induction step.
In the following lemma, we explain how to do such a modification.
Or, it can also turn out during our modification attempts 
that everything reduces to the induction base and we can conclude that $c_{w, n_\bullet} \ge 2$ right away.
\begin{lemma}\label{freechoiceorthogonalpre}
Let $w\in W$ and $n_\bullet \in \ZZ_{\ge 0}$ be such that $|n_\bullet|=\ell(w)$ and
the configuration $(\Phi^+ \cap w \Phi^-, n_\bullet)$ has small essential coordinates.
Suppose that $\Phi^+\cap w\Phi^-$ 
contains exactly one root $\alpha$ such that $w^{-1}\alpha\in -\Pi$.
Suppose that there exist a simple root distribution $f\colon \Phi^+\cap w\Phi^-\to \Pi$ 
with vector of D-multiplicities $n_\bullet$ 
and a root $\beta\in \Phi^+\cap w\Phi^-$ such that
$(\alpha,\beta)=0$ and $f(\alpha)\in \supp\beta$.
Then at least one of the following is true:
\begin{enumerate}
\item $c_{w,n_\bullet}\ge 2$.
\item 
One can define a 
simple root distribution $g\colon \Phi^+\cap w\Phi^-\to \Pi$ 
with vector of D-multiplicities $n_\bullet$
so that the restriction of $g$ to 
$\Phi^+\cap (\sigma_{\alpha}w)\Phi^-=(\Phi^+\cap w\Phi^-)\setminus \{ \alpha \}$ 
is flexible.
\end{enumerate}
\end{lemma}
\begin{proof}
First, without loss of generality we may suppose that $\beta$ is a $\prec_w$-maximal
root satisfying the preconditions of the lemma. 
In other words, we will suppose that $\beta$ is a $\prec_w$-maximal
element of the set 
$\{ \beta' \in \Phi^+\cap w\Phi^- : (\alpha, \beta') = 0, f(\alpha) \in \supp \beta' \}$.
Also, denote $\alpha_i=f(\alpha)$.
By Lemma \ref{uniqueminimalsupportcovers}, 
$\supp \beta\subseteq\supp \alpha$, so $f(\beta)\in \supp\alpha$. 
Consider the following function $h \colon \Phi^+\cap w\Phi^- \to \Pi$:
\begin{align*}
&h(\alpha)=f(\beta),\\
&h(\beta)=f(\alpha) = \alpha_i,\\ 
&h(\gamma)=f(\gamma) \text{ for all other }\gamma\in \Phi^+\cap w\Phi^-
\end{align*}
Since $f(\alpha)\in \supp\beta$ and (as we suppose from the very beginning) $f(\beta)\in \supp\alpha$, this function $h$ is really a 
simple root distribution. Clearly, its also has 
vector of D-multiplicities $n_\bullet$.
By Remark \ref{antisimplecompatible}, $\sigma_\alpha w \xrightarrow{\alpha} w$ is an $f$-compatible edge.
If there is an $h$-compatible edge $\sigma_\beta w \xrightarrow{\beta} w$ in the Bruhat graph, 
then by Lemma \ref{twodiffcompatible}, $c_{w,n_\bullet}\ge 2$, and we are done.

In what follows, we suppose that there is no 
$h$-compatible edge $\sigma_\beta w \xrightarrow{\beta} w$ in the Bruhat graph.
Then, by Lemma \ref{compatibleisminimalp}, 
there exists a root $\gamma\in \Phi^+\cap w\Phi^-$
such that $\beta\prec_w\gamma$, $(\beta,\gamma)=1$, $h(\gamma)\in \supp\beta$,
and $h(\beta)\in \supp \gamma$. 
We will not need the distribution $h$ anymore, and in terms of $f$, we have
$f(\gamma) \in \supp \beta$, $f(\alpha) = \alpha_i \in \supp \gamma$.

Let us apply Lemma \ref{rhombuspreconstruction}. The only condition in its statement we haven't explicitly checked yet is 
that $\alpha_i \in \supp n_\bullet$. 
But $n_\bullet$ is the vector of D-multiplicities of $f$, and $f(\alpha)=\alpha_i$, so $n_i > 0$, 
and $\alpha_i \in \supp n_\bullet$. 
So, by Lemma \ref{rhombuspreconstruction}, $(\alpha, \gamma)=1$.

Next, we are going to apply Lemma \ref{rhombusroot}.
Recall that $(\alpha,\beta)=0$, $(\beta, \gamma)=1$.
We also have 
$\alpha_i \in \supp \alpha$,
$\alpha_i \in \supp \beta$, $\alpha_i \in \supp \gamma$, and $\alpha_i \in \supp n_\bullet$, 
so the small essential coordinates assumption implies that 
$\alpha_i^*(\alpha)=\alpha_i^*(\beta)=\alpha_i^*(\gamma)=1$.
Finally, we have $\beta \prec_w \gamma$ and $\alpha \in w\Phi^-$.
By Lemma \ref{rhombusroot}, 
we get $\delta:=\alpha-\gamma+\beta \in \Phi^+ \cap w \Phi^-$, $(\delta,\gamma)=0$, 
$(\delta,\alpha)=1$ (in particular, $\delta \ne \alpha$), 
$(\delta,\beta)=1$, and $\alpha_i^*(\delta) = 1$ (in particular, $\alpha_i \in \supp \delta$).

Since $(\delta,\beta)=1$, the roots $\beta$ and $\delta$ are $\prec$-comparable. Let us consider 2 cases: 
$\beta \prec \delta$ and $\delta \prec \beta$. Also, let us note immediately that by Lemma \ref{uniqueminimalsupportcovers}, 
$\supp \gamma \subseteq\supp \alpha$ and $\supp \delta \subseteq\supp \alpha$.
\begin{enumerate}
\item $\beta \prec \delta$. 
Set $g(\delta)=f(\alpha)=\alpha_i$, $g(\alpha)=f(\delta)$, and 
$g(\varepsilon)=f(\varepsilon)$ for 
all other $\varepsilon \in \Phi^+\cap w\Phi^-$.
Recall that $\alpha_i \in \supp \delta$. We also have $f(\delta) \in \supp \delta \subseteq \supp \alpha$.
So, $g$ is really a distribution of simple roots.
Its vector of D-multiplicities is clearly $n_\bullet$.

Recall that $g(\gamma)=f(\gamma) \in \supp \beta$. We have $\beta \prec \delta$, so $\supp \beta \subseteq \supp \delta$, 
hence $g(\gamma) \in \supp \delta$. Also, recall that $g(\delta) = \alpha_i \in \supp \gamma$. 
So, the restriction of $g$ to $(\Phi^+\cap w\Phi^-)\setminus \{ \alpha \}$ is flexible.

\item $\delta \prec \beta$. 
Then we set $g(\gamma)=f(\alpha)=\alpha_i$, $g(\alpha)=f(\gamma)$, and 
$g(\varepsilon)=f(\varepsilon)$ for 
all other $\varepsilon \in \Phi^+\cap w\Phi^-$.
Now we have $\alpha_i \in \supp \gamma$, and also 
$f(\gamma) \in \supp \gamma \subseteq \supp \alpha$.
So, $g$ is again a distribution of simple roots with 
vector of D-multiplicities $n_\bullet$.

For the flexibility, note that $\gamma = \alpha + \beta - \delta \succ \alpha$, and $\supp \alpha \subseteq \supp \gamma$.
We also have $g(\delta)=f(\delta) \in \supp \delta$ and $\supp \delta \subseteq \supp \alpha$. Hence, $g(\delta) \in \supp \gamma$.
Also, recall that $g(\gamma)=\alpha_i \in \supp \delta$.
Again, the restriction of $g$ to $(\Phi^+\cap w\Phi^-)\setminus \{ \alpha \}$ is flexible.
\end{enumerate}
\end{proof}
\begin{example}
Let $G=SL_4$, $w = [3,4,1,2]$ (like in Example \ref{std4examplecompatible}), 
$\alpha=\alpha_1+\alpha_2+\alpha_3$, $\beta=\alpha_2$. This time, let $f(\alpha_1+\alpha_2+\alpha_3)=f(\alpha_2+\alpha_3)=f(\alpha_2)=\alpha_2$, 
$f(\alpha_1+\alpha_2)=\alpha_1$.

Then one checks directly that $\alpha$ is the only element of $\Phi^+ \cap w (-\Pi)$, and Lemma \ref{freechoiceorthogonalpre}
can be applied. In the proof we have just seen, we get $h=f$, and 
this time, $[3,1,4,2] \xrightarrow{\beta} [3,4,1,2]$ is not an $h$-compatible edge.
Following the proof further, we get $\gamma=\alpha_2+\alpha_3$, $\delta=\alpha_1+\alpha_2$, 
and $g(\alpha_1+\alpha_2)=g(\alpha_2+\alpha_3)=g(\alpha_2)=\alpha_2$, $g(\alpha_1+\alpha_2+\alpha_3)=\alpha_1$.
The restriction of $g$ to $(\Phi^+\cap w\Phi^-)\setminus \{ \alpha \}$ is indeed flexible:
$g(\gamma)=\alpha_2 \in \supp \delta$, $g(\delta)=\alpha_2 \in \supp \gamma$.
\end{example}

Now we are ready to do the induction.
\begin{lemma}\label{freechoiceflexible}
Let $w\in W$ and $n_\bullet \in \ZZ_{\ge 0}$ be such that $|n_\bullet|=\ell(w)$ and
the configuration $(\Phi^+ \cap w \Phi^-, n_\bullet)$ has small essential coordinates.
If there exists a flexible simple root distribution $f\colon \Phi^+\cap w\Phi^-\to \Pi$ 
with vector of D-multiplicities $n_\bullet$,
then $c_{w,n_\bullet}\ge 2$.
\end{lemma}
\begin{proof}
We do induction on $\ell(w)$.

First, let us consider the case when, for the definition of a flexible distribution, 
we can find two roots corresponding to \emph{antisimple} edges. Precisely, 
suppose that there exist two roots $\alpha,\beta \in \Phi^+\cap w \Phi^-$ such that
$w^{-1} \alpha,w^{-1} \beta \in -\Pi$, $f(\alpha)\in \supp \beta$, $f(\beta)\in \supp \alpha$, 
and $(\alpha,\beta)=0$.

Then we can consider (possibly) another simple roots distribution $g$ on $\Phi^+\cap w\Phi^-$:
$g(\alpha)=f(\beta)$, $g(\beta)=f(\alpha)$, and $g(\gamma)=f(\gamma)$ for all other $\gamma\in \Phi^+\cap w\Phi^-$.
Clearly, it also has 
vector of D-multiplicities $n_\bullet$.
By Remark \ref{antisimplecompatible}, $\sigma_\alpha w \xrightarrow{\alpha} w$ (resp.\ $\sigma_\beta w \xrightarrow{\beta} w$)
is an edge in the Bruhat graph compatible with $f$ (resp.\ with $g$). By Lemma \ref{twodiffcompatible}, $c_{w,n_\bullet} \ge 2$.

In what follows, we suppose that there are no pairs of roots $\alpha,\beta \in \Phi^+\cap w (-\Pi)$
such that $f(\alpha)\in \supp \beta$, $f(\beta)\in \supp \alpha$, 
and $(\alpha,\beta)=0$ (inside the proof, we refer to this assumption as ``assumption $(*)$'').
Our next goal is to find a root $\delta \in \Phi^+ \cap w(-\Pi)$ and a distribution $g \colon \Phi^+ \cap w \Phi^- \to \Pi$
with vector of D-multiplicities $n_\bullet$ such that the restriction of 
$g$ to $\Phi^+\cap \sigma_\delta w\Phi^- = (\Phi^+\cap w\Phi^-)\setminus \{ \delta \}$ is flexible.

If there exist three \emph{different} roots
$\alpha, \beta, \gamma \in \Phi^+\cap w\Phi^-$
such that $f(\alpha)\in \supp \beta$, $f(\beta)\in \supp \alpha$, 
and $(\alpha,\beta)=0$, and $w^{-1}\gamma \in -\Pi$, then we simply set $g=f$ and $\delta=\gamma$.

Suppose there are no such triples (``assumption $(**)$''). Since $f$ is flexible, we still have a pair of roots 
$\alpha, \beta \in \Phi^+\cap w\Phi^-$
such that $f(\alpha)\in \supp \beta$, $f(\beta)\in \supp \alpha$, 
and $(\alpha,\beta)=0$. 
And there still exists a root $\gamma \in \Phi^+ \cap w (-\Pi)$ by Lemma \ref{antisimpleexistsadmissible} (\ref{antisimpleexistsbgg}).
So, the only possibility that agrees with our current assumptions is that $\alpha = \gamma$ or $\beta = \gamma$.
Without loss of generality, $\alpha=\gamma$, $\alpha \in w (-\Pi)$.
Then it follows directly from our assumptions that $\alpha$ is the only element of 
$\Phi^+ \cap w \Phi^-$ that belongs to $-w\Pi$: $(*)$ excludes $\beta$, and $(**)$ excludes all other elements of $\Phi^+ \cap w \Phi^-$.
We can apply Lemma \ref{freechoiceorthogonalpre}.
Lemma \ref{freechoiceorthogonalpre} may tell us 
that $c_{w, n_\bullet}\ge 2$. In this case, we are done (even though we didn't construct the desired $g$).
Otherwise, set $\delta=\alpha$ and take $g$ from Lemma \ref{freechoiceorthogonalpre}.

So, now we have a root $\delta \in \Phi^+ \cap w(-\Pi)$ and a distribution $g \colon \Phi^+ \cap w \Phi^- \to \Pi$
with vector of D-multiplicities $n_\bullet$ such that the restriction of 
$g$ to 
$\Phi^+\cap \sigma_\delta w\Phi^-$
is flexible.
Let $\alpha_i = g(\delta)$.
Then the vector of D-multiplicities of 
$g|_{\Phi^+\cap \sigma_\delta w\Phi^-}$ is $n_\bullet - e_i$.
It is also clear that the configuration 
$(\Phi^+ \cap \sigma_\delta w \Phi^-, n_\bullet - e_i)=((\Phi^+ \cap w \Phi^-) \setminus \{ \delta \}, n_\bullet - e_i)$ 
has small essential coordinates.
By the induction hypothesis, $c_{\sigma_\delta w, n_\bullet - e_i} \ge 2$.
By Corollary \ref{antireducedsortingprefixmulttwopropagates} (applied to $(\sigma_\delta w \xrightarrow{\delta} w, \alpha_i)$), we get
$c_{w,n_\bullet}\ge 2$.
\end{proof}
\begin{sideremark}
In this proof, we used Lemma \ref{twodiffcompatible} directly for antisimple edges only. 
But there is no contradiction with 
Example \ref{std4example}, because in this proof we also used Lemma \ref{freechoiceorthogonalpre},
and in the proof of Lemma \ref{freechoiceorthogonalpre}, we also used Lemma \ref{twodiffcompatible}, 
and that time, we used it for not necessarily antisimple edges.
\end{sideremark}

Now, to finish the proof of Proposition \ref{freechoiceorthogonalnooverlap}, we have to construct a flexible 
simple root distribution on $\Phi^+ \cap w \Phi^-$. For that, we will need one more preliminary lemma.
\begin{lemma}\label{orthogonalpairrechoice}
Let $w\in W$ and $n_\bullet \in \ZZ_{\ge 0}$ be such that $|n_\bullet|=\ell(w)$ and
the configuration $(\Phi^+ \cap w \Phi^-, n_\bullet)$ has small essential coordinates.
Suppose that there exist roots $\alpha,\beta \in \Phi^+ \cap w \Phi^-$ and $\alpha_i \in \Pi$ such that 
$(\alpha,\beta)=0$ and $\alpha_i \in \supp \alpha \cap \supp \beta \cap \supp n_\bullet$.

Then there exist roots $\alpha',\beta' \in \Phi^+ \cap w \Phi^-$ such that 
$(\alpha',\beta')=0$, $\alpha_i \in \supp \alpha' \cap \supp \beta' \cap \supp n_\bullet$, 
and one of the following also holds:
\begin{enumerate}
\item $\alpha'=\alpha$ and $\beta' \in -w\Pi$;
\item $\beta' \prec \alpha'$.
\end{enumerate}
\end{lemma}
\begin{proof}
First, similarly to the beginning of the proof of Lemma \ref{freechoiceorthogonalpre}, 
without loss of generality, let us replace $\beta$ with a $\prec_w$-maximal element of the set 
$\{ \beta'' \in \Phi^+\cap w\Phi^- : (\alpha, \beta'') = 0, \alpha_i \in \supp \beta'' \}$.
If now $\beta \in -w\Pi$, we are done: set $\alpha'=\alpha$, $\beta'=\beta$.

Note that $\alpha$ and $\beta$ are $\prec$-comparable, then we are also done: 
we can set $\alpha'=\alpha$, $\beta'=\beta$ if $\beta \prec \alpha$, and 
$\alpha'=\beta$, $\beta'=\alpha$ if $\alpha \prec \beta$.

Now suppose that $\alpha$ and $\beta$ are not $\prec$-comparable and that $\beta \notin -w\Pi$.
By Lemma \ref{antisimpleisminimal} (\ref{alphaminimalsimplecrit}), 
this means that either there exist two roots $\gamma, \gamma' \in \Phi^+ \cap w \Phi^-$ such that $\gamma + \gamma' = \beta$, 
or there exists $\gamma \in \Phi^+ \cap w \Phi^-$ such that $\gamma - \beta \in \Phi^+ \cap w \Phi^+$.
Let us consider these two cases separately.
\begin{enumerate}
\item\label{inside:splits}
If there exist two roots $\gamma, \gamma' \in \Phi^+ \cap w \Phi^-$ such that $\gamma + \gamma' = \beta$, then 
$\supp \gamma \cup \supp \gamma' = \supp \beta$, and $\alpha_i \in \supp \beta$, so without loss of generality
(after a possible interchange of $\gamma$ and $\gamma'$), we may assume that $\alpha_i \in \supp \gamma$.
Also, by Lemma \ref{sumexists}, $(\beta,\gamma)=1$, and $\gamma' \in w\Phi^-$ means that $\beta \prec_w \gamma$.
\item\label{inside:covers}
If there exists $\gamma \in \Phi^+ \cap w \Phi^-$ such that $\gamma - \beta \in \Phi^+ \cap w \Phi^+$, then 
$\beta \prec \gamma$, so $\alpha_i \in \supp \beta$ implies $\alpha_i \in \supp \gamma$.
Also, we again have $(\beta,\gamma)=1$ by Lemma \ref{sumexists}, 
and $\gamma - \beta \in w\Phi^+$ again means that $\beta \prec_w \gamma$.
\end{enumerate}
So, in each of the cases \ref{inside:splits} and \ref{inside:covers}, we have a root $\gamma \in \Phi^+ \cap w \Phi^-$
such that $\alpha_i \in \supp \gamma$, $(\beta,\gamma)=1$, and $\beta \prec_w \gamma$.
By Lemma \ref{rhombuspreconstruction}, we have $(\alpha, \gamma)=1$.

Since the configuration $(\Phi^+ \cap w \Phi^-, n_\bullet)$ has small essential coordinates, 
we have $\alpha_i^*(\alpha)=\alpha_i^*(\beta)=\alpha_i^*(\gamma)=1$. By Lemma \ref{rhombusroot}, 
$\delta:= \alpha-\gamma+\beta \in \Phi^+ \cap w \Phi^-$, and we have $(\gamma,\delta)=0$ 
and $\alpha_i^* (\delta)=1$ 
(this time, we will not need the conclusions $(\alpha,\delta)=1$ and $(\beta,\delta)=1$ of Lemma \ref{rhombusroot}).

We have $(\alpha, \gamma)=1$ and $(\beta, \gamma)=1$, so both $\alpha$ and $\beta$ are $\prec$-comparable with $\gamma$.
But we are assuming that $\alpha$ is not $\prec$-comparable with $\beta$, so 
either $\alpha, \beta \prec \gamma$, or $\alpha,\beta \succ \gamma$.
If $\alpha, \beta \prec \gamma$, then $\delta = (\alpha - \gamma) +\beta \prec \beta \prec \gamma$, and we can set $\alpha'=\gamma$
and $\beta' = \delta$. 
Similarly, if $\alpha, \beta \succ \gamma$, then $\delta = (\alpha - \gamma) +\beta \succ \beta \succ \gamma$, and we can set $\alpha'=\delta$
and $\beta' = \gamma$. 
\end{proof}
\begin{corollary}\label{orthogonalpairrechoiceboth}
Let $w$, $n_\bullet$, $\alpha$, $\beta$, and $\alpha_i$ be as in Lemma \ref{orthogonalpairrechoice}.
Then there exist roots $\alpha'',\beta'' \in \Phi^+ \cap w \Phi^-$ such that 
$(\alpha'',\beta'')=0$, $\alpha_i \in \supp \alpha'' \cap \supp \beta'' \cap \supp n_\bullet$, 
and one of the following also holds:
\begin{enumerate}
\item both $\alpha''$ and $\beta''$ are in $-w\Pi$;
\item $\beta'' \prec \alpha''$.
\end{enumerate}
\end{corollary}
\begin{proof}
First, let us apply Lemma \ref{orthogonalpairrechoice} to the pair $\alpha,\beta$.
If it produces $\alpha', \beta'$ such that $\beta' \prec \alpha'$, we are done.

Otherwise, it produces a pair $\alpha',\beta'$, where $\alpha'=\alpha$ and $\beta' \in -w\Pi$.
Then let us apply Lemma \ref{orthogonalpairrechoice} again, to the pair $\beta', \alpha'$.
It can either produce $\alpha'', \beta''$ such that $\beta'' \prec \alpha''$ (and we are done), 
or $\alpha'', \beta''$ such that $\alpha'' = \beta'$ and $\beta'' \in -w\Pi$ (and this also finishes the proof since $\beta' \in -w\Pi$).
\end{proof}
Now we are ready to finish the proof of Proposition \ref{freechoiceorthogonalnooverlap}.
\begin{proof}[Last steps of the proof of Proposition \ref{freechoiceorthogonalnooverlap}]
If the configuration $(\Phi^+ \cap w \Phi^-, n_\bullet)$ has large essential coordinates, then $c_{w,n_\bullet} \ge 2$ by Lemma \ref{freechoicetwo}.

Suppose that the configuration $(\Phi^+ \cap w \Phi^-, n_\bullet)$ has small essential coordinates.
Then by Corollary \ref{orthogonalpairrechoiceboth}, we may assume without loss of generality 
that either $\alpha,\beta \in -w\Pi$, or $\beta \prec \alpha$. Let us consider these two cases separately.
Fix a root $\alpha_i \in \supp \alpha \cap \supp \beta \cap \supp n_\bullet$.

If $\alpha,\beta \in -w\Pi$, then by Lemma \ref{excessiveisfreechoice} (applied twice), 
there exist simple root distributions $f$ and $g$ on $\Phi^+ \cap w \Phi^-$, both with vector of D-multiplicities $n_\bullet$, 
such that $f(\alpha)=g(\beta)=\alpha_i$.
By Remark \ref{antisimplecompatible}, $\sigma_\alpha w \xrightarrow{\alpha} w$ (resp.\ $\sigma_\beta w \xrightarrow{\beta} w$)
is an edge in the Bruhat graph compatible with $f$ (resp.\ with $g$).
By Lemma \ref{twodiffcompatible}, $c_{w,n_\bullet} \ge 2$.

If $\beta \prec \alpha$, then also by Lemma \ref{excessiveisfreechoice} (applied once this time), 
there exists a simple root distribution $f$ on $\Phi^+ \cap w \Phi^-$ such that its 
vector of D-multiplicities is $n_\bullet$ and such that $f(\alpha)=\alpha_i$.
Then $f(\alpha)=\alpha_i \in \supp \beta$ by the choice of $\alpha_i$, and 
$f(\beta) \in \supp \alpha$ since $\beta \prec \alpha$ and $\supp \beta \subseteq \supp \alpha$.
So, $f$ is flexible, and $c_{w,n_\bullet} \ge 2$ by Lemma \ref{freechoiceflexible}.
\end{proof}
\begin{sideremark}
One could ask why, in this proof, we used Corollary \ref{antireducedsortingprefixmulttwopropagates}, 
which required us to deal with distributions of simple roots, instead of trying to use Lemma \ref{sortingprefixmulttwopropagates}, 
which simply needs configurations.
The problem is that if we tried to do induction on $\ell(w)$ using Lemma \ref{sortingprefixmulttwopropagates}, then,
even if we start with a 
strictly excessive configuration $(\Phi^+ \cap w \Phi^-, n_\bullet)$, 
then 
a configuration of the form $(\Phi^+ \cap \sigma_{\alpha} w \Phi^-, n_\bullet-e_i)$ appearing in Lemma \ref{sortingprefixmulttwopropagates}
doesn't have to be strictly excessive anymore.
And, as we will see in Example \ref{largeessntialoutsideinvolved} later, 
Proposition \ref{freechoiceorthogonalnooverlap} becomes wrong if we drop the strict excessiveness condition.

This also explains why we introduced the definitions of small essential coordinates and a flexible distribution:
while we could not maintain excessiveness if we tried to do induction using Lemma \ref{sortingprefixmulttwopropagates}, 
we still can maintain some properties of the configuration and distribution when we do induction 
using Corollary \ref{antireducedsortingprefixmulttwopropagates}, and we needed some terminology for these properties.
(Speaking of flexibility, the maintenance was not immediate, we had to modify the distribution sometimes, 
but it was possible, as we have seen in Lemma \ref{freechoiceorthogonalpre}.)
\end{sideremark}
\begin{example}\label{b2orthogonalexample}
Proposition \ref{freechoiceorthogonalnooverlap} becomes wrong if we drop the assumption that $G$ is of type $A$, $D$, or $E$.
Let $G$ be a group of type $B_2$, $w=\sigma_{\alpha_1+2\alpha_2} = \sigma_2 \sigma_1 \sigma_2$.
Then $\Phi^+ \cap w \Phi^-=\{\alpha_1+\alpha_2, \alpha_1+2\alpha_2, \alpha_2\}$. 
We have the following (antisimple) path in the Bruhat graph:
\begin{equation}\label{b2path}
1_W \xrightarrow{\alpha_2} \sigma_2 \xrightarrow{\alpha_1+2\alpha_2} \sigma_{\alpha_1+2\alpha_2} \sigma_2 = \sigma_2 \sigma_1 
\xrightarrow{\alpha_1+\alpha_2}
\sigma_{\alpha_1+\alpha_2} \sigma_{\alpha_1+2\alpha_2} \sigma_2 = w
\end{equation}
We have $\langle \varpi_2, \alpha_2 \rangle = \langle \varpi_2, \alpha_1+2\alpha_2 \rangle = \langle \varpi_2, \alpha_1 + \alpha_2 \rangle = 1$, 
so the Z-multiplicity of (\ref{b2path}) for enhancement $\alpha_2, \alpha_2, \alpha_2$ equals 1. 
And in fact, (\ref{b2path}) is the only enhanced path from $1_W$ to $w$ with enhancement $\alpha_2, \alpha_2, \alpha_2$.
One can check this by a case-by-case analysis of edges 
leading from vertices not in (\ref{b2path}) to vertices in (\ref{b2path})
using Lemma \ref{admissibilitycrit} and Proposition \ref{sortabilitycriterium} 
(recall that the results of Section \ref{sectioncriterionzero} hold for groups of all types).
For example, there is an edge $\sigma_2 w \xrightarrow{\alpha_2} w$, but by Lemma \ref{admissibilitycrit}, 
we get $\Phi^+ \cap \sigma_2 w \Phi^-=\{\alpha_1+\alpha_2, \alpha_1\}$, and 
by Proposition \ref{sortabilitycriterium} (\ref{distributionexists}), 
we get $c_{\sigma_2 w, 2e_2}=0$, so, there are no enhanced paths from $1_W$ to $w$ that have enhancement $\alpha_2, \alpha_2, \alpha_2$
and pass through $\sigma_2 w$.

This way we can prove that $c_{w,3e_2}=1$. But in $\Phi^+ \cap w \Phi^-$, there are 
orthogonal roots $\alpha=\alpha_1+\alpha_2$ and $\beta=\alpha_2$, and $\supp \alpha \cap \supp \beta \cap \supp (3e_2)=\{\alpha_2\} \ne \varnothing$.
\end{example}

\subsection{Obtuse angles between roots in an excessive configuration and $c_{w,n_\bullet}=1$}
In this subsection, we will prove the following proposition about 
pairs of roots in $\Phi^+ \cap w \Phi^-$ making an obtuse angle.
\begin{proposition}\label{freechoicenosplit}
Let $G$ be a semisimple split algebraic group with simply laced Dynkin diagram.
Let $w \in W$ and $n_\bullet \in \ZZ_{\ge 0}^r$ be such that 
$(\Phi^+ \cap w \Phi^-, n_\bullet)$ is a strictly excessive configuration.
If there exist two roots $\alpha, \beta \in \Phi^+ \cap w \Phi^-$ such that
$(\alpha, \beta)=-1$,
then $c_{w,n_\bullet} \ge 2$.
\end{proposition}
(Note that in this proposition, unlike in Proposition \ref{freechoiceorthogonalnooverlap}, 
we don't require that $\supp \alpha \cap \supp \beta \cap \supp n_\bullet \ne \varnothing$.
Moreover, in fact, if we added the condition $\supp \alpha \cap \supp \beta \cap \supp n_\bullet \ne \varnothing$, 
then we would already know that $c_{w,n_\bullet} \ge 2$ from Corollary \ref{freechoiceonetwozero}.)

We start with the following remark and corollary about 
what we already know.
\begin{remark}\label{excessivesimpleproperties}
\begin{enumerate}
\item\label{excessiveintersectsupp} Let $(A, n_\bullet)$ be 
an excessive configuration. 
Then the condition $R_{\supp n_\bullet}(A)=A$ can be reformulated as follows: for every $\alpha \in A$, 
we have $\supp \alpha \cap \supp n_\bullet \ne \varnothing$.
\item\label{restrictionunion} Let $A \subseteq \Phi^+$, $I,J \subseteq \Pi$. Then $R_{I \cup J}(A) = R_I(A) \cup R_J(A)$.
\end{enumerate}
\end{remark}
\begin{corollary}\label{excessive1product1}
Let $w \in W$ and $n_\bullet \in \ZZ_{\ge 0}^r$ be such that 
$(\Phi^+ \cap w \Phi^-, n_\bullet)$ is a strictly excessive configuration and $c_{w,n_1,\ldots,n_r}=1$.
If $\alpha, \beta \in \Phi^+ \cap w \Phi^-$, $\alpha \ne \beta$, and $\supp \alpha \cap \supp \beta \cap \supp n_\bullet \ne \varnothing$, 
then $(\alpha, \beta)=1$.
\end{corollary}
\begin{proof}
Follows directly from Corollary \ref{freechoiceonetwozero} and 
Proposition \ref{freechoiceorthogonalnooverlap}.
\end{proof}
Then, we will need a lemma about strictly excessive configurations.
\begin{lemma}\label{strictlyexcessivesplittingimpossible}
Let $(A, n_\bullet)$ be a strictly excessive configuration. Then it is impossible to split $\supp n_\bullet$
into a disjoint union of nonempty subsets $\supp n_\bullet = I \sqcup J$ ($I \ne \varnothing$, $J \ne \varnothing$)
so that $|R_I(A) \cap R_J(A)| \le 1$.
\end{lemma}
\begin{proof}
Assume the contrary. We have 
\begin{multline*}
|n_\bullet| = |A| = \\
\shoveleft{\text{(by the definition of an excessive configuration)}} \\
|R_{\supp n_\bullet}(A)| = \\
\text{(by Remark \ref{excessivesimpleproperties} (\ref{restrictionunion}))} \\
|R_I(A) \cup R_J(A)| = |R_I(A)| + |R_J(A)| - |R_I(A) \cap R_J(A)| \ge |R_I(A)| + |R_J(A)| - 1 \ge \\ 
\text{(by the definition of a strictly excessive configuration,}\\
\text{keeping in mind that all numbers here are integers)} \\
|p_I(n_\bullet)| + 1 + |p_J(n_\bullet)| + 1 - 1 = \\
\Big(\sum_{i:\alpha_i \in I} n_i\Big) + \Big(\sum_{i:\alpha_i \in J} n_i\Big) + 1 = |n_\bullet|+1.
\end{multline*}
A contradiction.
\end{proof}
We will also need the following property of strictly excessive configurations $(\Phi^+ \cap w \Phi^-, n_\bullet)$ with $c_{w,n_\bullet}=1$.
\begin{lemma}\label{freechoiceuniquemax}
Let $w \in W$ and $n_\bullet \in \ZZ_{\ge 0}^r$ be such that 
$(\Phi^+ \cap w \Phi^-, n_\bullet)$ is a strictly excessive configuration and $c_{w,n_1,\ldots,n_r}=1$.
Then there is a unique $\prec$-maximal element $\alpha \in \Phi^+\cap w\Phi^-$.
Moreover, $\supp n_\bullet \subseteq \supp \alpha$, 
and for every other element $\beta \in \Phi^+\cap w\Phi^-$, we have $(\alpha, \beta)=1$.
\end{lemma}
\begin{proof}
Let $\alpha \in \Phi^+\cap w\Phi^-$ be a $\prec$-maximal element.
Let $I = \supp \alpha \cap \supp n_\bullet$. By Remark \ref{excessivesimpleproperties} (\ref{excessiveintersectsupp}), $I \ne \varnothing$.

First, consider an arbitrary root $\beta \in R_I(\Phi^+ \cap w \Phi^-)$, $\beta \ne \alpha$.
By Corollary \ref{excessive1product1},
$(\alpha, \beta)=1$, and $\alpha$ and $\beta$ are $\prec$-comparable.
Since $\alpha$ is a $\prec$-maximal element of $\Phi^+\cap w\Phi^-$, we in fact have $\beta \prec \alpha$, 
$\supp \beta \subseteq \supp \alpha$, and $\supp \beta \cap \supp n_\bullet \subseteq I$.

Next, denote $J = \supp n_\bullet \setminus I$. 
Let us check that $R_I(\Phi^+ \cap w \Phi^-) \cap R_J(\Phi^+ \cap w \Phi^-) = \varnothing$.
Indeed, if we have a root
$\beta \in R_I(\Phi^+ \cap w \Phi^-)$, 
then the analysis of the previous paragraph shows that 
$\supp \beta \cap \supp n_\bullet \subseteq I$, 
so 
$\supp \beta \cap J = \varnothing$, 
and $\beta \notin R_J(\Phi^+ \cap w \Phi^-)$.
So, $R_J(\Phi^+ \cap w \Phi^-) \cap R_I(\Phi^+ \cap w \Phi^-) = \varnothing$.
By Lemma \ref{strictlyexcessivesplittingimpossible}, we cannot have $J \ne \varnothing$.

So, $J = \varnothing$, and $\supp n_\bullet = I \subseteq \supp \alpha$.
By the definition of an excessive configuration, $R_I (\Phi^+ \cap w \Phi^-)=\Phi^+ \cap w \Phi^-$.
And in the beginning of the proof, we have already checked the conclusion of the lemma for all roots from $R_I(\Phi^+ \cap w \Phi^-)$, 
so we are done.
\end{proof}
Now we are ready to proof Proposition \ref{freechoicenosplit}.
\begin{proof}[Proof of Proposition \ref{freechoicenosplit}]
First, $\gamma := \alpha+\beta\in \Phi$ by Lemma \ref{sumexists}.
Also, $\alpha, \beta \in \Phi^+ \cap w \Phi^-$, so 
$\gamma \in \Phi^+ \cap w \Phi^-$.

Let us check that $\gamma$ is the (unique by Lemma \ref{freechoiceuniquemax}) $\prec$-maximal element of $\Phi^+ \cap w \Phi^-$.
Indeed, let $\gamma'$ be the $\prec$-maximal element of $\Phi^+ \cap w \Phi^-$.
Clearly, $\alpha$ or $\beta$ cannot be equal to $\gamma'$, because $\alpha, \beta \prec \gamma$.
Then by Lemma \ref{freechoiceuniquemax}, $(\gamma', \alpha)=(\gamma',\beta)=1$, so $(\gamma', \alpha+\beta)=2$, 
and $\gamma'=\gamma$.

Denote $I=\supp \alpha \cap \supp n_\bullet$ and $J=\supp \beta \cap \supp n_\bullet$. 
Clearly, $\supp \alpha \cup \supp \beta = \supp \gamma$, and by Lemma \ref{freechoiceuniquemax}, 
$\supp n_\bullet \subseteq \supp \gamma$. So, $I \cup J = \supp n_\bullet$. 
Also, $I \ne \varnothing$ and $J \ne \varnothing$ by Remark \ref{excessivesimpleproperties} (\ref{excessiveintersectsupp}).
If $I \cap J \ne \varnothing$, then $c_{w,n_\bullet} \ge 2$ by Corollary \ref{freechoiceonetwozero}, and we are done. 
In what follows, let us suppose that $I \cap J = \varnothing$.

Assume that $c_{w,n_\bullet} \le 1$. 
Then actually, we have $c_{w,n_\bullet}=1$ by Proposition \ref{sortabilitycriterium} (\ref{excessiveconfiguration}).

Let us find $R_I(\Phi^+ \cap w \Phi^-) \cap R_J(\Phi^+ \cap w \Phi^-)$. 
Clearly, $\gamma \in R_I(\Phi^+ \cap w \Phi^-) \cap R_J(\Phi^+ \cap w \Phi^-)$. 
On the other hand, let $\delta \in R_I(\Phi^+ \cap w \Phi^-) \cap R_J(\Phi^+ \cap w \Phi^-)$.
Since $I \cap J = \varnothing$, we have $\alpha \notin R_J(\Phi^+ \cap w \Phi^-)$ and $\beta \notin R_I(\Phi^+ \cap w \Phi^-)$, 
so $\delta \ne \alpha$, $\delta \ne \beta$. 
Since we are assuming $c_{w,n_\bullet}=1$, 
by Corollary \ref{excessive1product1}, we get $(\delta,\alpha)=(\delta,\beta)=1$. So $(\delta, \alpha+\beta)=2$, and $\delta=\gamma$.
Hence, $R_I(\Phi^+ \cap w \Phi^-) \cap R_J(\Phi^+ \cap w \Phi^-)=\{\gamma\}$.

In particular, $|R_I(\Phi^+ \cap w \Phi^-) \cap R_J(\Phi^+ \cap w \Phi^-)|=1$, which means that we have a contradiction with 
Lemma \ref{strictlyexcessivesplittingimpossible}. Therefore, $c_{w,n_\bullet} \ge 2$.
\end{proof}

As we will see later, the conditions for $c_{w,n_\bullet}=1$ in 
Lemma \ref{freechoicetwo}, in Proposition \ref{freechoiceorthogonalnooverlap}, and in Proposition \ref{freechoicenosplit}
together also give a sufficient, not just necessary condition 
for $c_{w,n_\bullet}=1$ if we have a strictly excessive configuration. 
So, for now, let us combine these three statements into a single definition and a single proposition.
\begin{definition}\label{clusterdef}
A strictly excessive configuration $(A, n_\bullet)$ is called a \emph{cluster}
if the following conditions hold:
\begin{enumerate}
\item\label{clusterone} The configuration $(A, n_\bullet)$ has small essential coordinates. 
In other words, if $\alpha\in A$ and $n_i > 0$, then $\alpha_i^*(\alpha) \le 1$.
\item\label{clusternoonetwozero} If $\alpha,\beta\in A$, then $(\alpha,\beta) \ge 0$.
\item\label{clusternoinvolvedort} If $\alpha,\beta\in A$ and $(\alpha,\beta)=0$, then $\supp \alpha\cap \supp \beta\cap \supp n_\bullet=\varnothing$.
\end{enumerate}
\end{definition}
\begin{proposition}\label{freechoiceClusterNecessary}
Let $G$ be a semisimple split algebraic group with simply laced Dynkin diagram.
Let $w \in W$ and $n_\bullet \in \ZZ_{\ge 0}^r$ be such that 
$(\Phi^+ \cap w \Phi^-, n_\bullet)$ is a strictly excessive configuration and $c_{w,n_\bullet}=1$.
Then $(\Phi^+ \cap w \Phi^-, n_\bullet)$ is a cluster.
\end{proposition}
\begin{proof}
Follows from Lemma \ref{freechoicetwo}, Proposition \ref{freechoiceorthogonalnooverlap}, and Proposition \ref{freechoicenosplit}.
\end{proof}

\subsection{Reduction of the general case to the case of excessive configuration}

In this subsection we will prove a necessary condition (which will turn out to be also sufficient later) 
for $c_{w,n_\bullet}=1$ without any strict excessiveness assumption. 
Let us start with the corresponding definition. Now, unlike all subsections before, we will need restrictions 
of the form $R_{I,J}(A)$ with $J \ne \varnothing$.
Recall that if $A \subseteq \Phi^+$, then we define 
$R_{I,J} (A)$ as $\{ \alpha \in A : \supp \alpha \cap I \ne \varnothing , \supp \alpha \cap J = \varnothing \}$.
\begin{definition}
Let $A \subseteq \Phi^+$, $n_\bullet \in \ZZ_{\ge 0}^r$.
We call the configuration $(A,n_\bullet)$ \emph{clusterable} if 
$A = R_{\supp n_\bullet} (A)$ and
there exists an integer $k \ge 0$
and a decomposition $\supp n_\bullet = I_1 \sqcup \ldots \sqcup I_k$, where all $I_m$s are nonempty, such that
all configurations
\begin{equation}
\label{clusterization}
(R_{I_1, \varnothing}(A), p_{I_1} (n_\bullet)), (R_{I_2, I_1}(A), p_{I_2}(n_\bullet)), 
(R_{I_3, I_1 \cup I_2}(A), p_{I_3}(n_\bullet)), \ldots, 
(R_{I_k, I_1 \cup \ldots \cup I_{k-1}}(A), p_{I_k}(n_\bullet))
\end{equation}
are clusters. Note that we do call $(\varnothing, 0)$ a clusterable configuration (with $k=0$).
\end{definition}
%
\begin{sideremark}
It might look natural to include the 
excessiveness condition in the definition of a clusterable configuration.
But in fact, 
it can be proved 
(using Lemma \ref{hallintermsofrpre}, Lemma \ref{clusterizableinductive} (\ref{clusterizableinductivesingle}) below, 
and the strict excessiveness of clusters) that
the conditions we have already imposed
imply 
that every clusterable configuration $(A,n_\bullet)$ is 
excessive.
We will not use this fact (so we don't prove it precisely), but we will need (and will prove in Corollary \ref{clusterizableequal})
the weaker statement that $|A|=|n_\bullet|$.
\end{sideremark}
Now we are ready to formulate the main result of this subsection exactly.
\begin{proposition}\label{ClusterNecessary}
Let $G$ be a semisimple split algebraic group with simply laced Dynkin diagram.
Let 
$w \in W$ and $n_\bullet \in \ZZ_{\ge 0}^r$ be such that
$c_{w,n_\bullet}=1$.
Then 
the configuration $(\Phi^+\cap w\Phi^-, n_\bullet)$ is 
clusterable.
\end{proposition}
To prove this proposition, we start with an equivalent characterization of clusterable configurations.
%
\begin{lemma}\label{clusterizableinductive}
Let $(A,n_\bullet)$ be a configuration.
Then the following conditions are equivalent.
\begin{enumerate}
\item\label{clusterizableinductiveclusterizable} The configuration $(A,n_\bullet)$ is clusterable.
\item\label{clusterizableinductivesingle} 
\begin{enumerate}
\item Either $n_\bullet=0$ and $A=\varnothing$, 
\item or
$A = R_{\supp n_\bullet} (A)$ and 
there exists a nonempty subset $I \subseteq \supp n_\bullet$
such that $(R_{I}(A), p_I(n_\bullet))$ is a cluster and $(R_{\supp n_\bullet \setminus I, I}(A), p_{\supp n_\bullet \setminus I}(n_\bullet))$
is clusterable.
\end{enumerate}
\item\label{clusterizableinductivemultiple}
$A = R_{\supp n_\bullet} (A)$ and
there exists a 
(possibly empty)
subset $I \subseteq \supp n_\bullet$
such that $(R_{I}(A), p_I(n_\bullet))$ and $(R_{\supp n_\bullet \setminus I, I}(A), p_{\supp n_\bullet \setminus I}(n_\bullet))$
are clusterable.
\end{enumerate}
More precisely, if in condition \ref{clusterizableinductiveclusterizable}, $A \ne \varnothing$ 
and the definition of a clusterable configuration 
holds for a decomposition $\supp n_\bullet = I_1 \sqcup \ldots \sqcup I_k$, then in condition \ref{clusterizableinductivesingle},
we can take $I=I_1$, and the definition of a clusterable configuration 
will hold for $\supp n_\bullet \setminus I = I_2 \sqcup \ldots \sqcup I_k$.
And if the definitions of clusterable configurations in condition \ref{clusterizableinductivemultiple}
hold for decompositions $I = I_1 \sqcup \ldots \sqcup I_k$ and 
$\supp n_\bullet \setminus I = J_1 \sqcup \ldots \sqcup J_l$, then 
the definition of a clusterable configuration in condition \ref{clusterizableinductiveclusterizable}
holds for the decomposition
$\supp n_\bullet = I_1 \sqcup \ldots \sqcup I_k \sqcup J_1 \sqcup \ldots \sqcup J_l$.
\end{lemma}
\begin{proof}
For $\ref{clusterizableinductiveclusterizable} \Rightarrow \ref{clusterizableinductivesingle}$, note first that 
if $n_\bullet=0$, then $R_{\supp n_\bullet}(A)=\varnothing$, so $A=\varnothing$.
The rest of the proof ($\ref{clusterizableinductiveclusterizable} \Rightarrow \ref{clusterizableinductivesingle} 
\Rightarrow \ref{clusterizableinductivemultiple} 
\Rightarrow \ref{clusterizableinductiveclusterizable}$) 
follows directly from the definition of a clusterable configuration 
and the following three equalities, which in turn follow directly from the definition of a restriction.
\begin{enumerate}
\item If $I \subseteq \supp n_\bullet$ and $\supp n_\bullet \setminus I = J_1 \sqcup \ldots \sqcup J_l$, then
for every $m \in \{1,\ldots, l\}$ we have
$R_{J_m, J_1 \cup \ldots \cup J_{m-1}}(R_{\supp n_\bullet \setminus I,I} (A))=
R_{J_m, I \cup J_1 \cup \ldots \cup J_{m-1}}(A)$.
\item If $I \subseteq \supp n_\bullet$ and $I = I_1 \sqcup \ldots \sqcup I_k$, then 
for every $m \in \{1,\ldots, k\}$ we have
$R_{I_m, I_1 \cup \ldots \cup I_{m-1}}(A)=
R_{I_m, I_1 \cup \ldots \cup I_{m-1}}(R_I(A))$.
\item If $I \subseteq \supp n_\bullet$, $I = I_1 \sqcup \ldots \sqcup I_k$, and 
$\supp n_\bullet \setminus I = J_1 \sqcup \ldots \sqcup J_l$, then 
for every $m \in \{1,\ldots, l\}$ we have
$R_{J_m, I_1 \cup \ldots \cup I_k \cup J_1 \cup \ldots \cup J_{m-1}}(A)=
R_{J_m, J_1 \cup \ldots \cup J_{m-1}}(R_{\supp n_\bullet \setminus I,I} (A))$.
\end{enumerate}
\end{proof}
\begin{corollary}\label{clusterizableequal}
If $(A,n_\bullet)$ is a clusterable configuration, then $|A|=|n_\bullet|$.
\end{corollary}
\begin{proof}
Induction on $|n_\bullet|$ using Lemma \ref{clusterizableinductive} (\ref{clusterizableinductivesingle}). 
If $n_\bullet=0$, 
everything is clear.
If $|n_\bullet| > 0$, then 
let $I \subseteq \supp n_\bullet$ be a nonempty subset such that
such that $(R_{I}(A), p_I(n_\bullet))$ is a cluster and $(R_{\supp n_\bullet \setminus I, I}(A), p_{\supp n_\bullet \setminus I}(n_\bullet))$
is clusterable.
A cluster is a strictly excessive configuration by definition, so 
$|R_{I}(A)|= |p_I(n_\bullet)|$.

We have $n_\bullet=p_I(n_\bullet)+p_{\supp n_\bullet \setminus I}(n_\bullet)$ and 
$I \subseteq \supp n_\bullet$, $I \ne \varnothing$, so $|p_I(n_\bullet)|>0$
and $|p_{\supp n_\bullet \setminus I}(n_\bullet)|<|n_\bullet|$.
By the induction hypothesis, 
$|R_{\supp n_\bullet \setminus I, I}(A)|=|p_{\supp n_\bullet \setminus I}(n_\bullet)|$.

We always have $R_{I}(A) \cap R_{\supp n_\bullet \setminus I, I}(A) = \varnothing$ and 
$R_{\supp n_\bullet}(A)=R_{I}(A) \cup R_{\supp n_\bullet \setminus I, I}(A)$.
Also, now we have $A=R_{\supp n_\bullet}(A)$, so 
$|A|=|R_{I}(A)| + |R_{\supp n_\bullet \setminus I, I}(A)|$.
\end{proof}
The rest of the proof of Proposition \ref{ClusterNecessary} consists of two parts.
First, we are going to prove (for a certain subset $I \subseteq \supp n_\bullet$) 
that $(R_{I}(\Phi^+\cap w\Phi^-), p_I(n_\bullet))$ is a cluster, and then 
we will prove that 
$(R_{\supp n_\bullet \setminus I, I}(\Phi^+\cap w\Phi^-), p_{n_\bullet \setminus I}(n_\bullet))$
is clusterable.
In turn, for the first of these two parts, we need a few preliminary lemmas.
Recall that for any $J \subseteq \Pi$, we denote by $\Phi_J$ the set of roots whose supports are contained in $J$, 
and we denote by $W_J$ the subgroup of $W$ generated by the reflections corresponding to the roots from $\Phi_J$.
\begin{lemma}\label{firstsortsmall}
Let $w \in W$ and $n_\bullet \in \ZZ_{\ge 0}^r$ be such that $c_{w,n_\bullet} > 0$ (not necessarily $c_{w, n_\bullet}=1$).
Let $I \subseteq \supp n_\bullet$ be such that $|R_I(\Phi^+ \cap w \Phi^-)| = |p_I(n_\bullet)|$.
Let $J \subseteq \Pi$ be a subset such that $\Phi^+ \cap w \Phi^- \subseteq \Phi_J$.
Then there exists an element $v \in W$ and an enhanced path from $v$ to $w$ in the Bruhat graph with the following properties.
\begin{enumerate}
\item\label{highsortingdmult} The vector of D-multiplicities of the enhancement is $p_{\supp n_\bullet \setminus I} (n_\bullet)$.
\item\label{highsortingnoi} For each label $\beta_i$ of an edge in the path, we have 
$\supp \beta_i \subseteq J \setminus I$.
In particular, we have 
$w v^{-1} \in W_{J \setminus I}$.
\item\label{lowsortingpositive} $c_{v, p_I(n_\bullet)} > 0$.
\end{enumerate}
\end{lemma}
\begin{proof}
Choose an arbitrary sequence $S$ of simple roots with vector of D-multiplicities 
$n_\bullet - p_I(n_\bullet) = p_{\supp n_\bullet \setminus I}(n_\bullet)$.
Let us apply Corollary \ref{pieriformulacwnmultiple} to $c_{w,n_\bullet}$ and to this sequence $S$.
It follows from the positivity of the right-hand side of 
(\ref{pieriformulacwnmultiplerel})
that 
there exists an enhanced path of the form
\begin{equation*}
\left( 
P = v \xrightarrow{\beta_1} \ldots \xrightarrow{\beta_l} w, S
\right),
\end{equation*}
where $l=|p_{\supp n_\bullet \setminus I}(n_\bullet)|$, and $v \in W$ is such that $c_{v, p_I(n_\bullet)} > 0$.
So, claims \ref{highsortingdmult} and \ref{lowsortingpositive}
hold for these $v$ and $(P,S)$.

Let us check that claim \ref{highsortingnoi} also holds for $v$ and $(P,S)$.
By Lemma \ref{antisimplelabelsisphiwphi} (\ref{antisimplepathinequality}), 
there exists a bijection $h \colon (\Phi^+ \cap v \Phi^-) \sqcup \{ 1, \ldots, l\} \to \Phi^+ \cap w \Phi^-$ 
such that $h(i) \succeq \beta_i$ for $1 \le i \le l$ and $h(\alpha) \succeq \alpha$ for all $\alpha \in \Phi^+ \cap v \Phi^-$.
By Proposition \ref{sortabilitycriterium} (\ref{excessiveconfiguration}),
$(\Phi^+ \cap v \Phi^-, p_I(n_\bullet))$ is an excessive configuration, 
and by Remark \ref{excessivesimpleproperties} (\ref{excessiveintersectsupp}), 
we have $\supp \alpha \cap I \ne \varnothing$ for all $\alpha \in \Phi^+ \cap v \Phi^-$ (clearly, $\supp p_I(n_\bullet)=I$).
Then also $\supp h(\alpha) \cap I \ne \varnothing$, and $h(\alpha) \in R_I(\Phi^+ \cap w \Phi^-)$
for all $\alpha \in \Phi^+ \cap v \Phi^-$.

Recall that $|R_I(\Phi^+ \cap w \Phi^-)| = |p_I(n_\bullet)|$. Also, $|p_I(n_\bullet)|=\ell(v)=|\Phi^+ \cap v \Phi^-|$, 
so
$h(\Phi^+ \cap v \Phi^-)$
is actually the whole set $R_I(\Phi^+ \cap w \Phi^-)$, and $\{h(1), \ldots, h(l)\} = 
(\Phi^+ \cap w \Phi^-) \setminus R_I(\Phi^+ \cap w \Phi^-)$.
Hence, $\supp h(i) \cap I = \varnothing$ for all if $i \in \{1, \ldots, l\}$. 
Recall that $\Phi^+ \cap w \Phi^- \subseteq \Phi_J$, so $\supp h(i) \subseteq J$.
Since $\beta_i \preceq h(i)$, we also have 
$\supp \beta_i \subseteq J \setminus I$ 
for all labels $\beta_i$ of the edges of $P$. 
Finally, $w = \sigma_{\beta_l} \ldots \sigma_{\beta_{l_0+1}} v$, so 
$w v^{-1} = \sigma_{\beta_l} \ldots \sigma_{\beta_{l_0+1}} \in W_{J \setminus I}$.
\end{proof}
In this section, we are going to use 
Lemma \ref{firstsortsmall}
for $J=\Pi$ only (but still for arbitrary $I$), 
but later we will need it for arbitrary $J$ also.
\begin{lemma}\label{smallsortingpreservesclusters}
Let $v, w \in W$, $n_\bullet \in \ZZ_{\ge 0}^r$, $I \subseteq \supp n_\bullet$.
Suppose that $w v^{-1} \in W_{\Pi \setminus I}$. Then:
\begin{enumerate}
\item\label{restrictionmoved} We have $R_I(\Phi^+ \cap w \Phi^-) = wv^{-1} R_I(\Phi^+ \cap v \Phi^-)$.
\item\label{excessiveiff} The configuration $(R_I(\Phi^+ \cap w \Phi^-), p_I(n_\bullet))$ is strictly excessive if and only if 
the configuration $(R_I(\Phi^+ \cap v \Phi^-), p_I(n_\bullet))$ is strictly excessive.
\item\label{clusteriff} The configuration $(R_I(\Phi^+ \cap w \Phi^-), p_I(n_\bullet))$ is a cluster if and only if 
the configuration $(R_I(\Phi^+ \cap v \Phi^-), p_I(n_\bullet))$ is a cluster.
\end{enumerate}
\end{lemma}
\begin{proof}
Note first that since $w v^{-1} \in W_{\Pi \setminus I}$, 
for any $\alpha \in \Phi$ and for any $\alpha_i \in I$ we have $\alpha_i^* (\alpha) = \alpha_i^* (w v^{-1} \alpha)$.
Also, note that the statement of the lemma is symmetric in $v$ and $w$. 

For claim \ref{restrictionmoved}, let us prove that 
$wv^{-1} R_I(\Phi^+ \cap v \Phi^-) \subseteq R_I(\Phi^+ \cap w \Phi^-)$, 
then the other inclusion will follow after we interchange $v$ and $w$.
Let $\alpha \in R_I(\Phi^+ \cap v \Phi^-)$. 
Then $\supp \alpha \cap I \ne \varnothing$, 
choose an arbitrary $\alpha_i \in \supp \alpha \cap I$ (then $\alpha_i^*(\alpha) > 0$).
Then $\alpha_i^*(w v^{-1} \alpha) >0$ by the remark at the beginning of the proof,
so $w v^{-1} \alpha \in \Phi^+$ and $\supp(w v^{-1} \alpha) \cap I \ne \varnothing$.
We also have $\alpha \in v \Phi^-$, so $v^{-1} \alpha \in \Phi^-$, and $w v^{-1}\alpha \in w \Phi^-$.
Hence, $w v^{-1}\alpha \in R_I(\Phi^+ \cap w \Phi^-)$, which proves claim \ref{restrictionmoved}.

For claim 
\ref{excessiveiff}, 
let us also suppose that $(R_I(\Phi^+ \cap w \Phi^-), p_I(n_\bullet))$ is a strictly excessive configuration
and prove that $(R_I(\Phi^+ \cap v \Phi^-), p_I(n_\bullet))$ is strictly excessive, 
the other direction will follow by interchange of $v$ and $w$.

We clearly have $\supp p_I(n_\bullet)= I$ (recall that $I \subseteq \supp n_\bullet$)
and $R_I(R_I(\Phi^+ \cap v \Phi^-))=R_I(\Phi^+ \cap v \Phi^-)$.
The rest of the definition of 
strict excessiveness is formulated in 
terms of the cardinalities of the sets $R_J(R_I(\Phi^+ \cap v \Phi^-))$ and the numbers $|p_J(p_I(n_\bullet))|$
for all subsets $J \subseteq I$.
Clearly, $R_J(R_I(\Phi^+ \cap v \Phi^-))=R_J(\Phi^+ \cap v \Phi^-)$ (and same for $w$)
and $p_J(p_I(n_\bullet))=p_J(n_\bullet)$ (and these numbers don't depend on $v$ or $w$).
So, the strict excessiveness of $(R_I(\Phi^+ \cap v \Phi^-), p_I(n_\bullet))$
follows from claim \ref{restrictionmoved} applied to all subsets $J \subseteq I$, and we are done with claim \ref{excessiveiff}.

Claim \ref{clusteriff}: again, let us suppose that $(R_I(\Phi^+ \cap w \Phi^-), p_I(n_\bullet))$
is a cluster and prove that $(R_I(\Phi^+ \cap v \Phi^-), p_I(n_\bullet))$ is a cluster.
We have already proved that $(R_I(\Phi^+ \cap v \Phi^-), p_I(n_\bullet))$ is strictly excessive.
The remaining conditions in the definition of a cluster are formulated in terms of scalar products of roots
and of their coordinates (with respect to the basis $\Pi$) corresponding to the simple roots from $I$.
By claim \ref{restrictionmoved}, $wv^{-1}$ can be viewed as 
a bijection $R_I(\Phi^+ \cap v \Phi^-) \to R_I(\Phi^+ \cap w \Phi^-)$, and it preserves scalar products (clearly)
and these coordinates (the remark at the beginning of the proof). So, $(R_I(\Phi^+ \cap v \Phi^-), p_I(n_\bullet))$ is a cluster.
\end{proof}
\begin{lemma}\label{minstrictlyexcessive}
Let $(A,n_\bullet)$ be an excessive configuration.
Let $I \subseteq \supp n_\bullet$ be a minimal by inclusion nonempty subset
such that $|R_I(A)| = |p_I(n_\bullet)|$. Then $(R_I(A), p_I(n_\bullet))$ is a strictly excessive configuration.
\end{lemma}
\begin{proof}
Clearly, $\supp p_I(n_\bullet)=I$, and $R_I(R_I(A)) = R_I(A)$.
By the definition of an excessive configuration, for every $J \subseteq I$, we have
$|R_J(R_I(A))| = |R_J(A)| \ge |p_J(n_\bullet)| = |p_J(p_I(n_\bullet))|$. If $J \ne I$ and $J \ne \varnothing$, 
then this inequality becomes strict by the choice of $I$.
\end{proof}
We are ready to finish the first part of the proof of Proposition \ref{ClusterNecessary}.
\begin{lemma}\label{excessivesubsetiscluster}
Let 
$w \in W$ and $n_\bullet \in \ZZ_{\ge 0}^r$ be such that
$c_{w,n_\bullet}=1$.
Let $I \subseteq \supp n_\bullet$ be a minimal by inclusion nonempty subset
such that $|R_I(\Phi^+ \cap w \Phi^-)| = |p_I(n_\bullet)|$.
Then $(R_I(\Phi^+ \cap w \Phi^-), p_I(n_\bullet))$ is a cluster.
\end{lemma}
\begin{proof}
By Lemma \ref{firstsortsmall}, there exists an element an element $v \in W$ and an enhanced path from $v$ to $w$ 
with the following properties:
the vector of D-multiplicities of the enhancement is $p_{\supp n_\bullet \setminus I}(n_\bullet)$, 
$wv^{-1} \in W_{\Pi \setminus I}$, and $c_{v, p_I(n_\bullet)} > 0$.
By Proposition \ref{admissibilitycrit} (\ref{excessiveconfiguration}),
$(\Phi^+ \cap v \Phi^-, p_I(n_\bullet))$ is an excessive configuration. In particular, 
$R_I(\Phi^+ \cap v \Phi^-) = \Phi^+ \cap v \Phi^-$.

We also clearly have $p_{\supp n_\bullet \setminus I}(n_\bullet)+p_I(n_\bullet)=n_\bullet$, so by Lemma \ref{sortingprefixmulttwopropagates}, 
$c_{w,n_\bullet} \ge c_{v, p_I(n_\bullet)}$. Hence, the only possibility for $c_{v, p_I(n_\bullet)}$ is that $c_{v, p_I(n_\bullet)}=1$.

By Lemma \ref{minstrictlyexcessive}, $(R_I(\Phi^+ \cap w \Phi^-), p_I(n_\bullet))$ is a strictly excessive configuration.
By Lemma \ref{smallsortingpreservesclusters} (\ref{excessiveiff}), 
$(R_I(\Phi^+ \cap v \Phi^-), p_I(n_\bullet))$ is also a strictly excessive configuration.
We already know that $R_I(\Phi^+ \cap v \Phi^-) = \Phi^+ \cap v \Phi^-$, so we can say that 
$(\Phi^+ \cap v \Phi^-, p_I(n_\bullet))$ is a strictly excessive configuration.
By Proposition \ref{freechoiceClusterNecessary}, $(\Phi^+ \cap v \Phi^-, p_I(n_\bullet))$ is a cluster.
By Lemma \ref{smallsortingpreservesclusters} (\ref{clusteriff}),  
$(R_I(\Phi^+ \cap w \Phi^-), p_I(n_\bullet))$ is a cluster.
\end{proof}

Now, to continue the proof of Proposition \ref{ClusterNecessary}, we need the following lemma about existence of antisimple edges.
Previously, when we needed to find an antisimple edge, Lemma \ref{antisimpleadmissible} (\ref{antisimpleexistsbgg}) was enough for us,
but now we need to find antisimple edges in a more precise way.
\begin{lemma}\label{antisimplesuppexists}
Let $w\in W$
and $I \subseteq \Pi$ be such that
$R_I(\Phi^+\cap w\Phi^-) \ne \varnothing$.
Then $R_I(\Phi^+\cap w\Phi^-) \cap w(-\Pi) \ne \varnothing$.
In other words, there exists a root $\beta$ such that $\sigma_\beta w \xrightarrow{\beta} w$ is an antisimple edge
and $\supp \beta \cap I \ne \varnothing$.
\end{lemma}
\begin{proof}
Let 
$\alpha$ 
be a $\prec_w$-maximal element of 
$R_I(\Phi^+\cap w\Phi^-)$.
Assume that 
$w^{-1}\alpha \notin -\Pi$. 
Then by Lemma \ref{antisimpleisminimal}, 
either there exists roots 
$\beta,\gamma \in \Phi^+ \cap w\Phi^-$
such that 
$\alpha=\beta+\gamma$, 
or there exists 
$\beta \in \Phi^+ \cap w\Phi^-$
such that 
$\beta-\alpha \in \Phi^+ \cap w\Phi^+$.

If there exist roots 
$\beta,\gamma \in \Phi^+ \cap w\Phi^-$
such that 
$\alpha=\beta+\gamma$, 
then 
$\supp \alpha = \supp \beta \cup \supp \gamma$,
so $\supp \beta \cap I \ne \varnothing$ or $\supp \gamma \cap I \ne \varnothing$.
Without loss of generality, 
$\supp \beta \cap I \ne \varnothing$, $\beta \in R_I(\Phi^+\cap w\Phi^-)$.
We have 
$\alpha - \beta = \gamma \in w\Phi^-$, 
so 
$\alpha \prec_w \beta$. 
A contradiction with the $\prec_w$-maximality of 
$\alpha$.

If there exists 
$\beta \in \Phi^+ \cap w\Phi^-$
such that 
$\beta-\alpha \in \Phi^+ \cap w\Phi^+$,
then 
$\alpha \prec_w \beta$.
Also, 
$\alpha \prec \beta$,
so 
$\supp \alpha \subseteq \supp \beta$, 
and 
$\beta \in R_I(\Phi^+\cap w\Phi^-)$.
We again have a contradiction
with the $\prec_w$-maximality of 
$\alpha$.
\end{proof}
\begin{corollary}\label{antireducedsortingprefixexists}
Let $w\in W$, let $I \subseteq \Pi$, let $l = |R_I(\Phi^+ \cap w \Phi^-)|$.
Then there exists an antisimple path $u \xrightarrow{\beta_1} \ldots \xrightarrow{\beta_l} w$ 
such that
$\{\beta_1, \ldots, \beta_l\}= R_I(\Phi^+ \cap w \Phi^-)$.
\end{corollary}
\begin{proof}
Induction on $l$. If $l=0$, there is nothing to do. If $l>0$, 
then we can find the last edge $\sigma_{\beta_l} w \xrightarrow{\beta_l} w$
for the path
by Lemma \ref{antisimplesuppexists}.
Then, by 
Lemma \ref{admissibilitycrit},
$\Phi^+ \cap w \Phi^- = \{\beta_l\} \sqcup (\Phi^+ \cap \sigma_{\beta_l}w \Phi^-)$, 
so
$R_I(\Phi^+ \cap w \Phi^-) = \{\beta_l\} \sqcup R_I(\Phi^+ \cap \sigma_{\beta_l}w \Phi^-)$.
We can construct the rest of the path by the induction hypothesis.
\end{proof}

\begin{lemma}\label{subsetequalimpliesprojection}
Let $A \subseteq \Phi^+$, 
let $f$ be a simple root distribution on $A$ with vector of D-multiplicities $n_\bullet$.
If $I \subseteq A$ is a subset such that $|R_I(A)| = |p_I(n_\bullet)|$, then 
the vector of D-multiplicities of $f|_{R_I(A)}$ is $p_I(n_\bullet)$.
\end{lemma}
\begin{proof}
Denote the vector of D-multiplicities of $f|_{R_I(A)}$ by $m_\bullet$.
By the definition of the vector of D-multiplicities,
for every $\alpha_i \in I$, $f$ takes value $\alpha_i$ on $A$ exactly $n_i$ times.
By the definitions of a distribution and of a restriction, if $\alpha_i \in I$, then $f$ cannot take value $\alpha_i$ 
outside $R_I(A)$. Hence, if $\alpha_i \in I$, then $f$ takes value $\alpha_i$ on $R_I(A)$ also exactly $n_i$ times.
In other words, $m_i = n_i$ for all $\alpha_i \in I$. 

We have $|R_I(A)|=|p_I(n_\bullet)|=\sum_{i: \alpha_i \in I} n_i = \sum_{i: \alpha_i \in I} m_i$.
By the definition of the vector of D-multiplicities, we also have 
$|R_I(A)|=\sum_{1 \le i \le r} m_i$. Hence, $m_i = 0$ if $\alpha_i \notin I$, and $m_\bullet = p_I(n_\bullet)$.
\end{proof}

\begin{proof}[Last steps of the proof of Proposition \ref{ClusterNecessary}]
We do induction on $|n_\bullet|$. If $n_\bullet = 0$, then, since $\CH^*(G/B)$ is a graded ring, we can only have $c_{w,n_\bullet}=1$ for $w=1_W$.
In this case, $\Phi^+ \cap w \Phi^- = \varnothing$, and $(\varnothing, 0)$ is a clusterable configuration.

If $n_\bullet \ne 0$, then $\supp n_\bullet \ne \varnothing$. 
By Proposition \ref{admissibilitycrit} (\ref{excessiveconfiguration}),
$(\Phi^+ \cap w \Phi^-, n_\bullet)$ is an excessive configuration. 
In particular, $\Phi^+ \cap w \Phi^-=R_{\supp n_\bullet}(\Phi^+ \cap w \Phi^-)$ and 
$|\Phi^+ \cap w \Phi^-|=|n_\bullet|=|p_{\supp n_\bullet}(n_\bullet)|$.
So, there exists at least one nonempty subset $I' \subseteq \supp n_\bullet$ (namely, $I'=\supp n_\bullet$) such that 
$|R_{I'}(\Phi^+ \cap w \Phi^-)|=|p_{I'}(n_\bullet)|$.
Hence, we can choose a minimal by inclusion nonempty subset $I \subseteq n_\bullet$ such that 
$|R_I(\Phi^+ \cap w \Phi^-)|=|p_I(n_\bullet)|$.
Then by Lemma \ref{excessivesubsetiscluster}, 
$(R_I(\Phi^+ \cap w \Phi^-), p_I(n_\bullet))$ is a cluster.

Next, by Corollary \ref{antireducedsortingprefixexists}, 
there exists an antisimple path $u \xrightarrow{\beta_1} \ldots \xrightarrow{\beta_l} w$ 
such that
$\{\beta_1, \ldots, \beta_l\}= R_I(\Phi^+ \cap w \Phi^-)$.
We are going to apply Corollary \ref{antireducedsortingprefixmulttwopropagates}.

By Proposition \ref{sortabilitycriterium} (\ref{distributionexists}), there exists a 
simple root distribution $f$ on $\Phi^+ \cap w \Phi^-$ with vector of D-multiplicities $n_\bullet$.
By Lemma \ref{subsetequalimpliesprojection}, the restriction of $f$ to 
$R_I(\Phi^+ \cap w \Phi^-) = \{\beta_1, \ldots, \beta_l\}$
has vector of D-multiplicities $p_I(n_\bullet)$.
By Corollary \ref{antireducedsortingprefixmulttwopropagates}, 
$c_{w,n_\bullet} \ge c_{u,n_\bullet - p_I(n_\bullet)} > 0$.
Note that $n_\bullet - p_I(n_\bullet) = p_{\supp n_\bullet \setminus I}(n_\bullet)$.
We have $c_{w,n_\bullet}=1$, so the only possibility for 
$c_{u,n_\bullet - p_I(n_\bullet)} = c_{u,p_{\supp n_\bullet \setminus I}(n_\bullet)}$ is 
$c_{u,p_{\supp n_\bullet \setminus I}(n_\bullet)}=1$.

By the induction hypothesis, 
$(\Phi^+ \cap u \Phi^-, p_{\supp n_\bullet \setminus I}(n_\bullet))$
is clusterable.
By Lemma \ref{antisimplelabelsisphiwphi} (\ref{antisimplepathinequality}), 
$\Phi^+ \cap w \Phi^- = \{\beta_1, \ldots, \beta_l\} \sqcup (\Phi^+ \cap u \Phi^-)$, 
so $\Phi^+ \cap u \Phi^- = (\Phi^+ \cap w \Phi^-) \setminus R_I(\Phi^+ \cap w \Phi^-) = 
R_{\supp n_\bullet \setminus I,I}(\Phi^+ \cap w \Phi^-)$
(recall that for every $\alpha \in \Phi^+ \cap w \Phi^-$
we have $\supp \alpha \cap \supp n_\bullet \ne \varnothing$, see Remark \ref{excessivesimpleproperties} (\ref{excessiveintersectsupp})).
So, 
the configuration
$(R_{\supp n\setminus I,I}(\Phi^+ \cap w \Phi^-), p_{\supp n_\bullet \setminus I}(n_\bullet))$ is clusterable.
By Lemma \ref{clusterizableinductive} (\ref{clusterizableinductivesingle}), 
$(\Phi^+ \cap w \Phi^-, n_\bullet)$ is clusterable.
\end{proof}

\section{Sufficient condition for $c_{w,n_\bullet}=1$}
Our goal in this section is to prove the following proposition, which says that the necessary condition for $c_{w,n_\bullet}=1$ 
we have found in the previous section is also sufficient.
\begin{proposition}\label{ClusterSufficient}
Let $G$ be a semisimple split algebraic group with simply laced Dynkin diagram.
Let $w \in W$ and $n_\bullet \in \ZZ_{\ge 0}^r$
be such that the configuration $(\Phi^+ \cap w \Phi^-, n_\bullet)$ is clusterable.
Then $c_{w,n_\bullet}=1$.
\end{proposition}
In the proof of this proposition, instead of using enhanced paths,
it is more convenient to use 
Corollary \ref{pieriformulacwn} directly.

The proof will consist of two parts. First, we are going to prove that
if we choose $\alpha$ and $\alpha_i$ the right way, then we will have only one nonzero summand in (\ref{pieriformulacwnformula}).
Second, we will check that the configuration $(\Phi^+ \cap \sigma_\alpha w \Phi^-, n_\bullet - e_i)$
corresponding to this summand is clusterable, which will allow us to prove Proposition \ref{ClusterSufficient} by induction on $\ell(w)$.

The first part of the proof begins
with the following lemma, which is quite similar to Lemma \ref{freechoiceuniquemax}.
\begin{lemma}\label{clusteruniquewmax}
Let $w \in W$, $n_\bullet \in \ZZ_{\ge 0}^r$, and $I \subseteq \supp n_\bullet$ be such that 
$(R_I(\Phi^+ \cap w \Phi^-), p_I(n_\bullet))$ is a nonempty cluster.
Then there is a unique $\prec_w$-maximal element $\alpha \in R_I(\Phi^+ \cap w \Phi^-)$.
Moreover, $I \subseteq \supp \alpha$, 
and $\alpha \in -w \Pi$.
\end{lemma}
\begin{proof}
Let $\alpha \in R_I(\Phi^+ \cap w \Phi^-)$ be a $\prec_w$-maximal element.
Let $I' = \supp \alpha \cap I$. By Remark \ref{excessivesimpleproperties} (\ref{excessiveintersectsupp}), $I' \ne \varnothing$.

Like in the proof of Lemma \ref{freechoiceuniquemax}, let us first consider the set 
$R_{I'}(\Phi^+ \cap w \Phi^-)$. Let $\beta \in R_{I'}(\Phi^+ \cap w \Phi^-)$, $\beta \ne \alpha$.
Clearly, $\supp \beta \cap \supp \alpha \cap \supp p_I(n_\bullet) \ne \varnothing$, so by the definition of a cluster,
$(\alpha, \beta)=1$. Hence, $(w^{-1}\alpha, w^{-1}\beta)=1$, $\alpha$ and $\beta$ are $\prec_w$-comparable, and
$\beta \prec_w \alpha$ by the $\prec_w$-maximality of $\alpha$.

Also like in the proof of Lemma \ref{freechoiceuniquemax}, let us denote $J = I \setminus I'$ and check 
that $R_{I'}(\Phi^+ \cap w \Phi^-) \cap R_J(\Phi^+ \cap w \Phi^-) = \varnothing$.
Assume that there exists a root 
$\beta \in R_{I'}(\Phi^+ \cap w \Phi^-) \cap R_J(\Phi^+ \cap w \Phi^-)$. 
Then the analysis of the previous paragraph shows that 
$\beta \prec_w \alpha$ and $(\alpha, \beta)=1$.
Then $\beta -\alpha \in w \Phi^-$.

We 
have $\beta \in R_J(\Phi^+ \cap w \Phi^-)$, so, there exists a simple root $\alpha_j \in \supp \beta \cap J$.
In other words, $\alpha_j^*(\beta) > 0$. Also, $\supp \alpha \cap I = I'$. so $\supp \alpha \cap J = \varnothing$, and $\alpha_j^*(\alpha)=0$.
Hence, $\alpha_j^*(\beta - \alpha)>0$, $\beta - \alpha \in \Phi^+$, and $\beta - \alpha \in \Phi^+ \cap w \Phi^-$.
We also have $\alpha_j \in \supp (\beta - \alpha)$, $\alpha_j \in I$, so $\beta - \alpha \in R_I(\Phi^+ \cap w \Phi^-)$.
Finally, $(\alpha, \beta - \alpha)=-1$, and we have a contradiction with the definition of a cluster.
So, $R_{I'}(\Phi^+ \cap w \Phi^-) \cap R_J(\Phi^+ \cap w \Phi^-) = \varnothing$.
By Lemma \ref{strictlyexcessivesplittingimpossible}, we cannot have $J \ne \varnothing$.

So, $J = \varnothing$, $I = I'$ (which also implies $I \subseteq \supp \alpha$), 
and in the beginning of the proof, we have already checked that $\alpha$
is the unique $\prec_w$-maximal element of $R_{I'} (\Phi^+ \cap w \Phi^-)=R_I(\Phi^+ \cap w \Phi^-)$.

Finally, by Lemma \ref{antisimplesuppexists}, there exists a root $\alpha' \in R_I(\Phi^+ \cap w \Phi^-) \cap w(-\Pi)$.
By Lemma \ref{antisimpleisminimal} (\ref{alphaminimal}), $\alpha'$ is a $\prec_w$-maximal element of the whole 
$\Phi^+ \cap w \Phi^-$. Since we already know that a $\prec_w$-maximal element in 
$R_I(\Phi^+ \cap w \Phi^-)$ is unique, we get $\alpha'=\alpha$, $\alpha \in -w\Pi$.
\end{proof}

\begin{example}
Despite 
Lemma \ref{clusteruniquewmax}
and its proof are very similar to Lemma \ref{freechoiceuniquemax}, 
the root $\alpha$ found here doesn't have to be a $\prec$-maximal element of $R_I(\Phi^+ \cap w \Phi^-)$
if $I \ne \supp n_\bullet$ (if $I = \supp n_\bullet$, then $\alpha$ is in fact $\prec$-maximal, but we don't need this).
For example, if $G=SL_3$, $w=[3,2,1]$, $I=\{\alpha_1\}$, then $R_I(\Phi^+ \cap w \Phi^-)=\{\alpha_1, \alpha_1+\alpha_2\}$, 
and the $\prec_w$-maximal element of this set is $\alpha_1$, while the $\prec$-maximal element is $\alpha_1 + \alpha_2$.
\end{example}

\begin{lemma}\label{bcbig}
Let $w \in W$, $n_\bullet \in \ZZ_{\ge 0}^r$, and $I \subseteq \supp n_\bullet$ be such that 
$(R_I(\Phi^+ \cap w \Phi^-), p_I(n_\bullet))$ is a nonempty cluster.
Suppose that 
$\beta \in R_I(\Phi^+ \cap w \Phi^-)$ 
is \emph{not} the 
$\prec_w$-maximal element of $R_I(\Phi^+ \cap w \Phi^-)$.

Denote by $A_1$ (resp.\ $A_2$) the set of roots 
$\gamma \in R_I(\Phi^+ \cap w \Phi^-)$
such that $\beta \prec_w \gamma$
(resp.\ $(\beta, \gamma)=0$).
Then $|A_1\cup A_2| > |p_{I \setminus \supp \beta}(n_\bullet)|$.
\end{lemma}
\begin{proof}
Let us consider two cases: $I \not\subseteq \supp \beta$ and $I \subseteq \supp \beta$.

\emph{First}, suppose that 
$I \not\subseteq \supp \beta$, in other words, $I \setminus \supp \beta \ne \varnothing$.
In this case, let us check that $A_1 \cup A_2$ contains 
$R_{I\setminus \supp \beta} (\Phi^+ \cap w \Phi^-)$.
Assume the contrary: assume that there exists a root 
$\gamma \in R_{I\setminus \supp \beta} (\Phi^+ \cap w \Phi^-)$
such that 
$\gamma \notin A_1 \cup A_2$.
Fix a simple root 
$\alpha_i \in \supp \gamma \cap (I\setminus \supp \beta)$.
Clearly, 
$\gamma \in R_I(\Phi^+ \cap w \Phi^-)$, 
so by the definition of a cluster, 
$(\beta,\gamma) \ge 0$. 
Since 
$\gamma \notin A_2$, 
we have 
$(\beta,\gamma) > 0$.
Also, 
$\alpha_i \notin \supp \beta$, $\alpha_i \in \supp \gamma$, so $\beta \ne \gamma$ and $(\beta, \gamma) \ne 2$.

So, the only remaining possibility for 
$(\beta,\gamma)$ is $(\beta,\gamma)=1$. 
By Lemma \ref{sumexists}, 
$\gamma - \beta \in \Phi$,
and
$\beta$ and $\gamma$ 
are $\prec_w$-comparable.
Since $\gamma \notin A_1$, we have $\gamma \prec_w \beta$ and 
$\gamma - \beta \in w \Phi^-$.
Next, 
$\alpha_i^*(\gamma - \beta)>0$, so $\gamma - \beta \in \Phi^+$, and $\gamma-\beta \in R_I(\Phi^+ \cap w \Phi^-)$.
But 
$(\beta,\gamma-\beta)=-1$, 
a contradiction with the definition of a cluster.


Therefore, 
$R_{I\setminus \supp \beta} (\Phi^+ \cap w \Phi^-) \subseteq A_1 \cup A_2$, 
$|A_1 \cup A_2| \ge |R_{I\setminus \supp \beta} (\Phi^+ \cap w \Phi^-)|$. 
We cannot have 
$I\setminus \supp \beta = I$ since $\beta \in R_I(\Phi^+ \cap w \Phi^-)$.
So, by the definition of a strictly excessive configuration, 
$|R_{I\setminus \supp \beta} (\Phi^+ \cap w \Phi^-)| > |p_{I \setminus \supp \beta}(n_\bullet)|$, 
and we are done if 
$I \not\subseteq \supp \beta$.

\emph{Second}, if 
$I \subseteq \supp \beta$, then $I \setminus \supp \beta = \varnothing$ 
and 
$|p_{I \setminus \supp \beta}(n_\bullet)|=0$. 
On the other hand, by Lemma \ref{clusteruniquewmax}, 
there exists a unique $\prec_w$-maximal element 
$\alpha \in R_I(\Phi^+ \cap w \Phi^-)$.
We have supposed that $\alpha \ne \beta$, so 
$\beta \prec_w \alpha$, 
$\alpha \in A_2$, and 
$|A_1 \cup A_2| > 0$.
\end{proof}

Recall that if 
$\sigma_\beta w \xrightarrow{\beta} w$ 
is an edge in the Bruhat graph, then 
in Definition \ref{phipm} and in Lemma \ref{admissibilitycrit}, we have constructed a bijection 
$\psi_{\beta} \colon (\Phi^+ \cap w \Phi^-) \setminus \{\beta\} \to \Phi^+ \cap \sigma_\beta w \Phi^-$.
\begin{lemma}\label{continuenotbc}
Let $w \in W$, $n_\bullet \in \ZZ_{\ge 0}^r$, and $I \subseteq \supp n_\bullet$ be such that 
$(R_I(\Phi^+ \cap w \Phi^-), p_I(n_\bullet))$ is a nonempty cluster.
Suppose we have $\sigma_\beta w \xrightarrow{\beta} w$ 
for some $\beta\in R_I(\Phi^+ \cap w \Phi^-)$.

Denote by $A_1$ (resp.\ $A_2$) the set of roots 
$\gamma \in R_I(\Phi^+ \cap w \Phi^-)$
such that $\beta \prec_w \gamma$
(resp.\ $(\beta, \gamma)=0$).
Then 
$R_{\supp \beta \cap I} (\Phi^+ \cap \sigma_\beta w \Phi^-) \cap \psi_{\beta} (A_1 \cup A_2)=\varnothing$.
\end{lemma}
\begin{proof}
If $\gamma \in A_2$, then 
$(\gamma, \beta)=0$ and $\psi_{\beta}(\gamma)=\gamma$.
By the definition of a cluster, 
$\supp \beta \cap \supp \gamma \cap I = \varnothing$.
Therefore, 
$\psi_{\beta}(\gamma)=\gamma \notin R_{\supp \beta \cap I}(\Phi^+ \cap \sigma_\beta w \Phi^-)$.

Now consider a root 
$\gamma \in A_1 \setminus A_2$.
By the definition of a cluster, 
$(\beta,\gamma) \ge 0$, so
it follows from 
$\gamma \in A_1 \setminus A_2$ 
that
$(\beta,\gamma)=1$.
Then 
$\gamma-\beta \in \Phi$, and 
$\gamma \in A_1$ means that
$\beta - \gamma \in w\Phi^-$.
If $\beta-\gamma \in \Phi^+$, then $\beta-\gamma \in \Phi^+\cap w\Phi^-$, $\gamma \in \Phi^+\cap w\Phi^-$, 
and we cannot have 
$\sigma_\beta w \xrightarrow{\beta} w$ 
by Lemma \ref{admissibilitycrit}, a contradiction.
So, 
$\beta-\gamma \in \Phi^-$, 
$\sigma_\beta \gamma = \gamma - \beta \in \Phi^+ \cap w \Phi^+$,
and 
$\psi_{\beta}(\gamma)=\gamma-\beta$.

Assume that 
$\gamma-\beta \in R_{\supp \beta \cap I}(\Phi^+ \cap \sigma_\beta w \Phi^-)$.
Then there exists a simple root 
$\alpha_i \in \supp (\gamma - \beta) \cap \supp \beta \cap I$.
Then $\alpha_i^*(\gamma - \beta) \ge 1$, $\alpha_i^*(\beta) \ge 1$, 
so $\alpha_i^*(\gamma) \ge 2$, 
and we have 
a contradiction with the definition of 
a cluster.
\end{proof}

\begin{lemma}\label{sortonlywithprecwmax}
Let $w \in W$, $n_\bullet \in \ZZ_{\ge 0}^r$, and $I \subseteq \supp n_\bullet$ be such that 
$(R_I(\Phi^+ \cap w \Phi^-), p_I(n_\bullet))$ is a nonempty cluster.
Suppose we have 
$\sigma_\beta w \xrightarrow{\beta} w$ 
for some 
$\beta \in R_I(\Phi^+ \cap w \Phi^-)$, 
but 
$\beta$ 
is not the \emph{not} the 
$\prec_w$-maximal element of $R_I(\Phi^+ \cap w \Phi^-)$.

Let 
$\alpha_i\in \supp \beta \cap I$.
Then 
$c_{\sigma_\beta w, n_\bullet - e_i}=0$.
\end{lemma}
\begin{proof}
By Lemma \ref{admissibilitycrit}, we have 
$\psi_{\beta}(\gamma) \preceq \gamma$ for all 
$\gamma \in (\Phi^+ \cap w \Phi^-) \setminus \{\beta\}$.
So, 
$R_{\supp \beta \cap I} (\Phi^+ \cap \sigma_\beta w \Phi^-) \subseteq 
R_I (\Phi^+ \cap \sigma_\beta w \Phi^-) \subseteq 
\psi_{\beta} (R_I (\Phi^+ \cap w \Phi^-) \setminus \{\beta\})$.

Like in the previous few lemmas, denote by 
$A_1$ (resp.\ $A_2$) the set of roots 
$\gamma \in R_I(\Phi^+ \cap w \Phi^-)$
such that $\beta \prec_w \gamma$
(resp.\ $(\beta, \gamma)=0$).
Clearly, 
$A_1 \cup A_2 \subseteq R_I (\Phi^+ \cap w \Phi^-) \setminus \{\beta\}$.

By the definition of a cluster (and of a strictly excessive configuration), 
we have $|R_I (\Phi^+ \cap w \Phi^-)|=|p_I(n_\bullet)|$, so
$|R_I (\Phi^+ \cap w \Phi^-) \setminus \{\beta\}| = |p_I(n_\bullet)|-1$.
By Lemma \ref{bcbig}, 
$|A_1 \cup A_2| > |p_{I \setminus \supp \beta} (n_\bullet)|$, 
and by Lemma \ref{continuenotbc}, 
$R_{\supp \beta \cap I} (\Phi^+ \cap \sigma_\beta w \Phi^-) \cap \psi_{\beta} (A_1 \cup A_2)=\varnothing$.
Therefore, 
$|R_{\supp \beta \cap I} (\Phi^+ \cap \sigma_\beta w \Phi^-)| < |p_I(n_\bullet)|-1-|p_{I \setminus \supp \beta} (n_\bullet)|
= |p_{\supp \beta \cap I} (n_\bullet)|-1 = |p_{\supp \beta \cap I} (n_\bullet-e_i)|$
(recall that $\alpha_i \in \supp \beta \cap I$).

Finally, denote 
$J = I \cap \supp \beta \cap \supp (n_\bullet - e_i)$. 
In general, $J$ doesn't have to be equal to 
$I \cap \supp \beta$ 
(if $n_i=1$, then $J$ is a proper subset, 
otherwise, they are indeed equal), but we always have 
$J \subseteq \supp \beta \cap I$ and 
$p_J(n_\bullet - e_i) = p_{I \cap \supp \beta} (n_\bullet - e_i)$.
Also, since 
$J \subseteq \supp \beta \cap I$, we have 
$R_{J} (\Phi^+ \cap \sigma_\beta w \Phi^-) \subseteq 
R_{\supp \beta \cap I} (\Phi^+ \cap \sigma_\beta w \Phi^-)$.
Therefore, 
$|R_{J} (\Phi^+ \cap \sigma_\beta w \Phi^-)| < |p_J(n_\bullet - e_i)|$, and
$(\Phi^+ \cap \sigma_\beta w \Phi^-, n_\bullet - e_i)$ 
is not 
an excessive configuration.
By Lemma \ref{sortabilitycriterium} (\ref{excessiveconfiguration}), 
$c_{\sigma_\beta w, n_\bullet - e_i} = 0$.
\end{proof}

For the second part of the proof of Proposition \ref{ClusterSufficient}, we are going to show that if we choose 
$I$ and $\alpha$ as in Lemma \ref{clusteruniquewmax} and an arbitrary $\alpha_i \in \supp \alpha \cap I$, 
then $((\Phi^+ \cap w \Phi^-) \setminus \{\alpha\}, n_\bullet - e_i)$ is a clusterable configuration.
This claim might look obvious intuitively, but a rigorous proof requires some technical efforts.
\begin{lemma}\label{clusterclusterizable}
Let $A \subseteq \Phi^+$, let $f$ be a distribution of simple roots on $A$, and let $n_\bullet$ be the vector of 
D-multiplicities of $f$. 
Suppose that $(A, n_\bullet)$ satisfies all conditions in the definition of a cluster except strict excessiveness.
In other words, suppose that: 
for all $\alpha \in A$ and for all $\alpha_i \in \supp n_\bullet$ we have $\alpha_i^*(\alpha) \le 1$;
for all $\alpha, \beta \in A$ we have $(\alpha, \beta) \ge 0$; 
and if for some $\alpha, \beta \in A$ we have $(\alpha, \beta) = 0$, then $\supp \alpha \cap \supp \beta \cap \supp n_\bullet = \varnothing$.

Then the configuration $(A, n_\bullet)$ is clusterable.
\end{lemma}

\begin{proof}
Induction on $|A|$. For $A=\varnothing$, everything is clear.

By Lemma \ref{hallintermsofrpre}, the configuration $(A, n_\bullet)$ is 
(not necessarily strictly)
excessive.
Let $I$ be a minimal by inclusion nonempty subset of $\supp n_\bullet$ such that 
$|R_I(A)|=|p_I(n_\bullet)|$.
Then it follows directly from the definitions of 
excessive and strictly excessive 
configurations
that $(R_I(A),p_I(n_\bullet))$ is strictly excessive. 
And then it follows from the assumptions of the lemma that 
$(R_I(A),p_I(n_\bullet))$ is a cluster.

By Lemma \ref{subsetequalimpliesprojection}, the vector of D-multiplicities of $f|_{R_I(A)}$ is $p_I(n_\bullet)$.
So, the vector of D-multiplicities of the restriction of $f$ to 
$A \setminus R_I(A)$ 
is 
$n_\bullet - p_I(n_\bullet) = p_{\supp n_\bullet \setminus I} (n_\bullet)$.
By the induction hypothesis, 
$(A \setminus R_I(A), p_{\supp n \setminus I} (n_\bullet))$ 
is clusterable.
By the definition of 
an excessive configuration, 
$A = R_{\supp n_\bullet}(A)$, so 
$A \setminus R_I(A) = R_{\supp n_\bullet \setminus I, I}(A)$.
By Lemma \ref{clusterizableinductive} (\ref{clusterizableinductivesingle}), 
$(A, n_\bullet)$ is clusterable.
\end{proof}

\begin{lemma}\label{clusterizableremoveone}
Let $(A, n_\bullet)$ be a clusterable configuration with $A\ne \varnothing$, $n_\bullet \ne 0$.
Let $I \subseteq \supp n_\bullet$ be a nonempty subset such that
$(R_I(A), p_I(n_\bullet))$ is a cluster and 
$(R_{\supp n_\bullet \setminus I,I} (A), p_{\supp n_\bullet \setminus I}(n_\bullet))$ is a clusterable
configuration.

Let $\alpha\in R_I(A)$ and $\alpha_i \in \supp \alpha \cap I$. 
Then
$(A \setminus\{\alpha\}, n_\bullet - e_i)$ 
is a clusterable configuration.
\end{lemma}

\begin{proof}
By Lemma \ref{excessiveisfreechoice}, there exists a distribution of simple roots $f$ on $R_I(A)$ 
with vector of D-multiplicities $p_I(n_\bullet)$ and such that $f(\alpha)=\alpha_i$.
Then $f|_{R_I(A) \setminus \{ \alpha \}}$ has vector of D-multiplicities
$p_I(n_\bullet-e_i)$.
Denote $n'_\bullet=n_\bullet-e_i$, 
$I'=\supp (p_I(n'_\bullet))$
(like in the proof of Lemma \ref{sortonlywithprecwmax}, 
$I'$ fails to be equal to $I$
if $n_i=1$).
By Lemma \ref{clusterclusterizable}, 
$(R_I(A) \setminus \{ \alpha \}, p_I(n'_\bullet))$
is a clusterable configuration.
In particular,
$R_I(A) \setminus \{ \alpha \}=R_{I'}(R_I(A) \setminus \{ \alpha \})$.
Denote $A'=A \setminus \{\alpha\}$.
By the definition of restriction, we can write
$R_{I'}(R_I(A) \setminus \{ \alpha \})=R_{I'}(A \setminus \{\alpha\})=R_{I'}(A')$.
Finally, $p_I(n'_\bullet)=p_{I'}(p_I(n'_\bullet))=p_{I'}(n'_\bullet)$, and we can say that 
$(R_{I'}(A'), p_{I'}(n'_\bullet))$
is a clusterable configuration.

Next, 
we know that 
$(R_{\supp n_\bullet \setminus I,I} (A), p_{\supp n_\bullet \setminus I}(n_\bullet))$ is a clusterable
configuration, and now we need to slightly rewrite this statement.
Denote $J=\supp n_\bullet \setminus I$. We also have 
$J=\supp (n'_\bullet) \setminus I'$ 
and $p_J(n_\bullet)=p_J(n'_\bullet)$.
Also, recall that $\alpha \in R_I(A)$ and that $R_I(A) \setminus \{ \alpha \}=R_{I'}(A')$.
We have
\begin{equation*}
R_{J,I} (A)=R_{J} (A \setminus R_I(A))=R_{J} ((A \setminus \{\alpha\}) \setminus (R_I(A) \setminus \{\alpha\}))=
R_{J} (A' \setminus R_{I'}(A '))=
R_{J,I'} (A'),
\end{equation*}
and we can say that 
$(R_{J,I'} (A'), p_J(n'_\bullet))$ 
is a clusterable configuration.

Finally, $(A, n_\bullet)$ is clusterable, so $A=R_{\supp n_\bullet}(A)$.
Also, $\alpha \notin R_{J,I} (A)$, and we can write
\begin{multline*}
A'=A \setminus \{\alpha\} = (R_{\supp n_\bullet}(A)) \setminus \{\alpha\} = 
(R_I(A) \cup R_{J,I}(A)) \setminus \{\alpha\} = \\
(R_I(A) \setminus \{\alpha\}) \cup R_{J,I}(A) = 
R_{I'}(A') \cup R_{J,I'}(A') = R_{\supp n'_\bullet} (A').
\end{multline*}
By Lemma \ref{clusterizableinductive} (\ref{clusterizableinductivemultiple}), the configuration
$(A',n'_\bullet)=(A \setminus\{\alpha\}, n_\bullet - e_i)$ 
is clusterable.
\end{proof}

\begin{proof}[Last steps of the proof of Proposition \ref{ClusterSufficient}]
We do induction on $\ell(w)$. If $\ell(w)=0$, then everything is clear.

If $\ell(w) > 0$, then by Lemma \ref{clusterizableinductive} (\ref{clusterizableinductivesingle}), 
there exists a nonempty subset $I \subseteq \supp n_\bullet$
such that 
$(R_I(\Phi^+ \cap w \Phi^-), p_I(n_\bullet))$ is a cluster and 
$(R_{\supp n_\bullet \setminus I,I} (\Phi^+ \cap w \Phi^-), p_{\supp n_\bullet \setminus I}(n_\bullet))$ is a clusterable
configuration.
By Lemma \ref{clusteruniquewmax}, there is a unique $\prec_w$-maximal element $\alpha \in R_I(\Phi^+ \cap w \Phi^-)$.
Lemma \ref{clusteruniquewmax} also says that $I \subseteq \supp \alpha$ and $\alpha \in -w \Pi$.
Fix an arbitrary $\alpha_i \in I$.


We are going to use Corollary \ref{pieriformulacwn}. Let us check that all terms in the right-hand side of (\ref{pieriformulacwnformula}), 
except for the term with $\beta=\alpha$, vanish. Indeed, 
for all terms we have $\beta \in \Phi^+ \cap w \Phi^-$ by Lemma \ref{admissibilitycrit}.
If $\alpha_i \notin \supp \beta$, 
then $\langle \varpi_i, \beta \rangle=\alpha_i^*(\beta)=0$. And if $\alpha_i \in \supp \beta$, 
then $\beta \in R_I (\Phi^+ \cap w \Phi^-)$, so we have $c_{\sigma_\beta w, n_\bullet - e_i} = 0$ by Lemma \ref{sortonlywithprecwmax}.

Finally, for the term with $\beta=\alpha$, we have $\Phi^+ \cap \sigma_\alpha w \Phi^- = (\Phi^+ \cap w \Phi^-) \setminus \{\alpha\}$
by Lemma \ref{admissibilitycrit} since $\alpha \in -w \Pi$.
By Lemma \ref{clusterizableremoveone}, 
the configuration $((\Phi^+ \cap w \Phi^-) \setminus \{\alpha\}, n_\bullet - e_i)$ is clusterable.
By the induction hypothesis, $c_{\sigma_\alpha w, n_\bullet - e_i}=1$. And 
$\langle \varpi_i, \alpha \rangle=\alpha_i^*(\alpha)=1$ by the definition of a cluster.
So, the only nonzero summand in (\ref{pieriformulacwnformula}) equals 1, and $c_{w,n_\bullet}=1$.
\end{proof}

\section{Criterion for $c_{w,n_\bullet}=1$}
\begin{theorem}\label{MainTheorem}
Let $G$ be a semisimple split algebraic group with simply laced Dynkin diagram.
Let $w \in W$, $n_\bullet \in \ZZ_{\ge 0}^r$.
Then $c_{w,n_\bullet}=1$ if and only if the configuration $(\Phi^+ \cap w \Phi^-, n_\bullet)$ is 
clusterable. If these conditions hold, then $\ell(w)=|n_\bullet|$.
\end{theorem}
\begin{proof}
Follows directly from Proposition \ref{ClusterNecessary}, 
Proposition \ref{ClusterSufficient}, and the gradedness of $\CH^*(G/B)$.
\end{proof}

In Lemma \ref{freechoicetwo}, we said that if the configuration $(\Phi^+ \cap w \Phi^-,n_\bullet)$
has large essential coordinates and is strictly excessive, then $c_{w,n_\bullet} \ge 2$.
We have also said in Proposition \ref{freechoiceorthogonalnooverlap} that if $(\Phi^+ \cap w \Phi^-,n_\bullet)$
is again strictly excessive and there are two orthogonal roots $\alpha,\beta \in \Phi^+ \cap w \Phi^-$
such that $\supp \alpha \cap \supp \beta \cap \supp n_\bullet \ne \varnothing$, then also $c_{w,n_\bullet} \ge 2$.
Now we are ready to see why the strict excessiveness condition in these two statements cannot be removed
or even replaced with 
(not necessarily strict)
excessiveness.
\begin{example}\label{largeessntialoutsideinvolved}
Let $G$ be a group of type $D_4$, $w=\sigma_2 \sigma_3 \sigma_4 \sigma_2 \sigma_1$ 
(recall that we enumerate the vertices of the Dynkin diagram as in \cite{bou}).
Then one can check by a direct computation (for example, using \cite[Lemma 2.2]{bgg})
that $\Phi^+ \cap w \Phi^-=\{\alpha_2, \alpha_2+\alpha_3, \alpha_2+\alpha_4, \alpha_2 + \alpha_3 + \alpha_4, \alpha_1 + 2\alpha_2 +\alpha_3 + \alpha_4\}$.
Let $n_\bullet=(1,1,1,2)$.

Let us check that the configuration $(\Phi^+ \cap w \Phi^-, n_\bullet)$ is clusterable.
Let $I_1 = \{\alpha_1\}$, $I_2 = \{\alpha_4\}$, $I_3 = \{\alpha_3\}$, and $I_4=\{\alpha_2\}$.
Then $R_{I_1, \varnothing} (\Phi^+ \cap w \Phi^-)=\{\alpha_1 + 2\alpha_2 +\alpha_3 + \alpha_4\}$, $p_{I_1}(n_\bullet)=e_1$;
$R_{I_2, I_1}(\Phi^+ \cap w \Phi^-)=\{\alpha_2+\alpha_4, \alpha_2+\alpha_3+\alpha_4\}$, $p_{I_2}(n_\bullet)=2e_4$;
$R_{I_3, I_1 \cup I_2} (\Phi^+ \cap w \Phi^-)= \{\alpha_2+\alpha_3\}$, $p_{I_3}(n_\bullet)=e_3$;
$R_{I_4, I_1 \cup I_2 \cup I_3} (\Phi^+ \cap w \Phi^-)= \{\alpha_2\}$, $p_{I_4}(n_\bullet)=e_2$.
So, $(\Phi^+ \cap w \Phi^-, n_\bullet)$ is indeed clusterable, and $c_{w,n_\bullet}=1$ by Theorem \ref{MainTheorem}.

On the other hand, we have $\alpha_2 \in \supp n_\bullet$ and $\alpha_2^* (\alpha_1 + 2\alpha_2 +\alpha_3 + \alpha_4)=2$
(so, the configuration $(\Phi^+ \cap w \Phi^-, n_\bullet)$ has large essential coordinates).
Also, we have $(\alpha_2+\alpha_3, \alpha_2+\alpha_4)=0$ and $\supp (\alpha_2+\alpha_3) \cap \supp (\alpha_2+\alpha_4) \cap \supp n_\bullet=\{\alpha_2\}$.
And $(\Phi^+ \cap w \Phi^-, n_\bullet)$ is even 
excessive 
(this can be checked directly by hand or follows from Proposition \ref{sortabilitycriterium} (\ref{excessiveconfiguration})
since we already know that $c_{w,n_\bullet}>0$). 
But $(\Phi^+ \cap w \Phi^-, n_\bullet)$ is not strictly excessive.
\end{example}
In the definition of a cluster, we have the condition $\supp \alpha \cap \supp \beta \cap \supp n_\bullet = \varnothing$ for pairs of 
orthogonal roots, and this condition might look complicated.
The following example shows that we cannot exclude all pairs of orthogonal roots 
in the definition of a cluster and even replace the condition 
$\supp \alpha \cap \supp \beta \cap \supp n_\bullet = \varnothing$ 
with a simpler-looking condition $\supp \alpha \cap \supp \beta = \varnothing$.
\begin{example}\label{orthogonalcannotbeweakened}
Again, let $G$ be a group of type $D_4$.
Also, like in the previous example, let $w=\sigma_2 \sigma_3 \sigma_4 \sigma_2 \sigma_1$
(then $\Phi^+ \cap w \Phi^-=\{\alpha_2, \alpha_2+\alpha_3, \alpha_2+\alpha_4, \alpha_2 + \alpha_3 + \alpha_4, \alpha_1 + 2\alpha_2 +\alpha_3 + \alpha_4\}$), 
but this time, let $n_\bullet=(0,1,2,2)$.

Let $I_1=\{\alpha_3, \alpha_4\}$, $I_2=\{\alpha_2\}$.
Then 
$R_{I_1, \varnothing} (\Phi^+ \cap w \Phi^-)=\{\alpha_2+\alpha_3, \alpha_2+\alpha_4, \alpha_2 + \alpha_3 + \alpha_4, \alpha_1 + 2\alpha_2 +\alpha_3 + \alpha_4\}$,
$p_{I_1}(n_\bullet)=2e_3+2e_4$;
$R_{I_2, I_1}(\Phi^+ \cap w \Phi^-)=\{\alpha_2\}$, $p_{I_2}(n_\bullet)=e_2$.
We see that the configuration $(\Phi^+ \cap w \Phi^-, n_\bullet)$ is clusterable, 
so, by Theorem \ref{MainTheorem}, $c_{w,n_\bullet}=1$.
But $R_{I_1, \varnothing} (\Phi^+ \cap w \Phi^-)$ contains two orthogonal roots $\alpha_2+\alpha_3$ and $\alpha_2+\alpha_4$, 
and their supports even have non-empty intersection (they intersect at $\alpha_2$).
We have $\alpha_2 \notin p_{I_1}(n_\bullet) = \{\alpha_3,\alpha_4\}$, so 
our definition of a cluster works for 
$(R_{I_1, \varnothing} (\Phi^+ \cap w \Phi^-), p_{I_1}(n_\bullet))$.
But if we replaced ``$\supp \alpha \cap \supp \beta \cap \supp n_\bullet = \varnothing$''
with ``$\supp \alpha \cap \supp \beta = \varnothing$'' in the definition of a cluster, then such a modified definition 
would fail for 
$(R_{I_1, \varnothing} (\Phi^+ \cap w \Phi^-), p_{I_1}(n_\bullet))$
and would give us a wrong conclusion in Theorem \ref{MainTheorem}.
\end{example}
\begin{sideremark}
One could ask why, in Example \ref{orthogonalcannotbeweakened}, we used a group of type $D$ despite 
pairs of orthogonal roots exist in somewhat simpler root systems of type $A$ also
(unlike coordinates 2 with respect to $\Pi$, so in Example \ref{largeessntialoutsideinvolved}, type $D$ was clearly necessary).
The answer is that, in fact, in type $A$, it is not possible to find a cluster containing a pair of orthogonal 
roots. In other words, in type $A$, 
the whole conditions \ref{clusternoonetwozero} and \ref{clusternoinvolvedort} in Definition \ref{clusterdef}
can be replaced with ``if $\alpha,\beta\in A$, then $(\alpha,\beta) > 0$''. 
However, we will not need this fact later, so we don't prove it.
\end{sideremark}


\section{Simple clusters and degrees of multiplicity free monomials}
Theorem \ref{MainTheorem} makes it possible to check if $c_{w,n_\bullet}=1$. 
But clusterable configurations in general still can be too complicated to directly 
find the maximal degree of a multiplicity-free monomial (in other words, the maximal value of $\ell(w)$ or of $|n_\bullet|$
such that $c_{w,n_\bullet}=1$).
For this, we will need the following definitions of a simple cluster and a simply clusterable configuration.
\begin{definition}
A \emph{simple cluster} is a cluster of the form $(A, k e_i)$ ($k \in \ZZ_{\ge 0}$).
In other words, a configuration $(A, n_\bullet)$ is a simple cluster 
if there exists $i \in \{1,\ldots,r\}$ such that: $n_\bullet = n_i e_i$ (at most one nonzero coordinate), 
$|A|=n_i$,
$\alpha_i^*(\alpha)=1$ for all $\alpha \in A$, and 
$(\alpha,\beta)=1$ for all $\alpha,\beta \in A$, $\alpha \ne \beta$.
\end{definition}
\begin{definition}
A configuration $(A, n_\bullet)$ is called \emph{simply clusterable} if $A=R_{\supp n_\bullet}(A)$ and there exists 
an ordering $\supp n_\bullet = \{ \alpha_{i_1}, \ldots, \alpha_{i_k} \}$
such that 
for $I_m:= \{\alpha_{i_m}\}$ ($1 \le m \le k$), 
all configurations in (\ref{clusterization}) are (automatically simple) clusters.
In other words, all configurations 
$(R_{\{\alpha_{i_m}\}, \{\alpha_{i_1}, \ldots, \alpha_{i_{m-1}}\}}(A), p_{\{\alpha_{i_m}\}}(n_\bullet))$
for $1 \le m \le k$  must be simple clusters.
\end{definition}
Clearly, a simply clusterable configuration is clusterable.
The following two lemmas show how simply clusterable configurations help to find the maximal degree of a 
multiplicity-free monomial.
\begin{lemma}\label{clustersimplyclusterizable}
Let $A \subseteq \Phi^+$. Let $I \subseteq \Pi$ be a subset such that $R_I(A)=A$ and 
(like in the definition of a cluster): for every $\alpha_i \in I$ and $\alpha \in A$ we have $\alpha_i^*(\alpha) \le 1$;
for every $\alpha, \beta \in A$ we have $(\alpha,\beta) \ge 0$; and 
for every $\alpha, \beta \in A$ such that $(\alpha,\beta) = 0$ we have $\supp \alpha \cap \supp \beta \cap I = \varnothing$.

Then there exists $n_\bullet \in \ZZ_{\ge 0}^r$ such that $\supp n_\bullet \subseteq I$ 
and $(A, n_\bullet)$ is simply clusterable.
\end{lemma}
\begin{proof}
Induction on $|A|$. If $A=\varnothing$, everything is clear. Suppose $A \ne \varnothing$.

Since $R_I(A) = A \ne \varnothing$,
there exists a simple root $\alpha_i \in I$
such that $R_{\{\alpha_i\}}(A) \ne \varnothing$ (Remark \ref{excessivesimpleproperties} (\ref{restrictionunion})).
Choose one such $\alpha_i \in I$ and denote $k = |R_{\{\alpha_i\}}(A)|$.
Then $(R_{\{\alpha_i\}}(A), k e_i)$ is a simple cluster.

Denote $I'=I \setminus \{\alpha_i\}$ and $A'=R_{I', \{\alpha_i\}} (A)$.
By the induction hypothesis, there exists 
$n'_\bullet \in \ZZ_{\ge 0}^r$ such that $\supp n'_\bullet \subseteq I'$ 
and $(A', n'_\bullet)$ is simply clusterable.
Then $R_{\supp n'_\bullet}(A')=A'$. Set $n_\bullet = ke_i + n'_\bullet$, then 
$\supp n_\bullet = \{\alpha_i\} \sqcup \supp n'_\bullet$.
We have 
\begin{equation*}
A = R_I(A)=R_{\{\alpha_i\}}(A) \cup A' = R_{\{\alpha_i\}}(A) \cup R_{\supp n'_\bullet}(A') \subseteq 
R_{\supp n_\bullet}(A).
\end{equation*}
Also, $R_{\supp n_\bullet}(A) \subseteq A$ by the definition of restriction, so $A = R_{\supp n_\bullet}(A)$.
By Lemma \ref{clusterizableinductive} 
($\ref{clusterizableinductivemultiple} \Rightarrow \ref{clusterizableinductiveclusterizable}$ and the ``more precisely'' part), 
$(A,n_\bullet)$ is a simply clusterable configuration.
\end{proof}

\begin{lemma}\label{excessiveclusterisableissimple}
Let $(A, n_\bullet)$ be a clusterable configuration.
Then there exists $m_\bullet \in \ZZ_{\ge 0}^r$  such that
$\supp m_\bullet \subseteq \supp n_\bullet$, 
$|m_\bullet|=|n_\bullet|$, and 
$(A,m_\bullet)$ is a simply clusterable configuration.
\end{lemma}
\begin{proof}
The proof is very similar to the proof of the previous lemma (although refers to it also)
and also runs by 
induction on $|A|$. If $A=\varnothing$, everything is clear. Suppose $A \ne \varnothing$.

By Lemma \ref{clusterizableinductive} (\ref{clusterizableinductivesingle}),
there exists a nonempty subset $I\subseteq \supp n_\bullet$ such that
$(R_I(A), p_I(n_\bullet))$ is a cluster and 
$(R_{\supp n_\bullet \setminus I, I}(A), p_{\supp n_\bullet \setminus I}(n_\bullet))$ is a clusterable configuration.

By Lemma \ref{clustersimplyclusterizable}, there exists $n'_\bullet \in \ZZ_{\ge 0}^r$ 
such that $\supp n'_\bullet \subseteq I$ and $(R_I(A), n'_\bullet)$ is simply clusterable.
Denote $J = \supp n_\bullet \setminus I$. 
Since $I \subseteq \supp n_\bullet$, $I \ne \varnothing$, and $(R_I(A), p_I(n_\bullet))$ is a cluster, 
we have $|R_I(A)|=|p_I(n_\bullet)|>0$.
Since $A=R_{\supp n_\bullet}(A)$ by the definition of a clusterable configuration, 
we have $R_{J,I}(A) = A \setminus R_I(A)$, and $|R_{J,I}(A)| < |A|$.
So, we can apply the induction hypothesis to $(R_{J,I}(A), p_J(n_\bullet))$ 
and find $m'_\bullet \in \ZZ_{\ge 0}^r$ with 
$\supp m'_\bullet \subseteq J$
such that $(R_{J,I}(A),m'_\bullet)$ is a simply clusterable configuration.

Set $m_\bullet=n'_\bullet+m'_\bullet$. Then $\supp m_\bullet = \supp n'_\bullet \sqcup \supp m'_\bullet \subseteq I \sqcup J = \supp n_\bullet$.
The configuration $(R_I(A), n'_\bullet)$ is simply clusterable, so 
$R_{\supp n'_\bullet}(R_I(A))=R_I(A)$.
Also, $(R_{J,I}(A),m'_\bullet)$ is simply clusterable, so $R_{\supp m'_\bullet}(R_{J,I}(A))=R_{J,I}(A)$.
We have 
\begin{equation*}
A = R_I(A) \cup R_{J,I}(A) = R_{\supp n'_\bullet}(R_I(A)) \cup R_{\supp m'_\bullet}(R_{J,I}(A)) \subseteq 
R_{\supp m_\bullet}(A).
\end{equation*}
Like in the proof of previous lemma, $R_{\supp m_\bullet}(A) \subseteq A$, so in fact $A = R_{\supp m_\bullet}(A)$.
Again by Lemma \ref{clusterizableinductive} 
($\ref{clusterizableinductivemultiple} \Rightarrow \ref{clusterizableinductiveclusterizable}$ and the ``more precisely'' part), 
$(A,m_\bullet)$ simply clusterable.
Finally, by Lemma \ref{clusterizableequal}, $|n_\bullet|=|A|$ and $|m_\bullet|=|A|$.
\end{proof}
The following example shows that the vector $m_\bullet$ constructed in this proof may really have a smaller support than $\supp n_\bullet$
(in particular, the checks that $A = R_{\supp n_\bullet}(A)$ and $A = R_{\supp m_\bullet}(A)$ in the last two proofs were not redundant).
\begin{example}
Let $G=SL_4$, $w=[4,1,2,3]$, $n_\bullet = 2e_1 + e_2$. 
Then $\Phi^+ \cap w \Phi^- = \{\alpha_1, \alpha_1+\alpha_2, \alpha_1+\alpha_2+\alpha_3\}$, 
and $(\Phi^+ \cap w \Phi^-, n_\bullet)$ is a cluster, but not a simple cluster.
If, in the proof of Lemma \ref{clustersimplyclusterizable}, we start with $\alpha_i=\alpha_1$, 
then we will get $m_\bullet = 3e_1$, a vector with a smaller support that $\supp n_\bullet$.
(If we start with $\alpha_i=\alpha_2$ there, we will get a different $m_\bullet$, 
namely, $m_\bullet=e_1+2e_2$, with the same support as $\supp n_\bullet$.)
\end{example}

Now, after we have Lemma \ref{excessiveclusterisableissimple}, let us study simple clusters and simply clusterable configurations more carefully.
We start with a few formulas.
\begin{lemma}\label{differencesformapath}
Let $\beta_1, \ldots, \beta_l \in \Phi$ be roots such that $(\beta_i, \beta_j)=1$ for $i \ne j$.
Let $\gamma_1 = \beta_1$, $\gamma_i = \beta_i - \beta_{i-1}$ for $2 \le i \le l$.
Then $(\gamma_i, \gamma_{i+1})=-1$ for $1 \le i \le l-1$ and 
$(\gamma_i, \gamma_j)=0$ for $1 \le i,j \le l$, $i \le j-2$.
\end{lemma}
\begin{proof}
Direct calculation (separately for $i=1$ and for $i > 1$).
\end{proof}
\begin{lemma}\label{simpleclusterpath}
Let $w \in W$ and $i \in \{1, \ldots, r\}$ be such that $(\Phi^+ \cap w \Phi^-, le_i)$
is a simple cluster ($l = \ell(w)$).
Suppose we have written $\Phi^+ \cap w \Phi^-=\{ \beta_1, \ldots, \beta_l \}$, where 
$\beta_1 \prec \ldots \prec \beta_l$. Then there exists a 
simple path 
$i_1=i, i_2, \ldots, i_l$ in the Dynkin diagram such that $\beta_j = \alpha_{i_1} + \ldots + \alpha_{i_j}$
for all $j \in \{1,\ldots, l\}$.
\end{lemma}
\begin{proof}
If $l=0$, everything is clear (we just take the empty path). Suppose $l > 0$.

Denote $\gamma_1 = \beta_1$, $\gamma_j = \beta_j - \beta_{j-1}$ for $2 \le j \le l$.
By the definition of a simple cluster and by Lemma \ref{sumexists}, we have $\gamma_j \in \Phi^+$ for $1 \le j \le l$.

Let us check that $\gamma_j \in \Pi$ for $1 \le j \le l$. Assume the contrary:
assume that for some $j$, we have $\gamma_j = \gamma' + \gamma''$ for some $\gamma', \gamma'' \in \Phi^+$. 
Note that $\gamma', \gamma'' \prec \gamma_j$.
If $j=1$, we get a contradiction right away: we cannot have $\gamma' \in w \Phi^-$ or $\gamma'' \in w \Phi^-$ since
we are assuming that $\gamma_1 = \beta_1$ is the $\prec$-minimal element of $\Phi^+ \cap w \Phi^-$, 
and $\gamma', \gamma'' \in w \Phi^+$ would imply $\beta_1 = \gamma' + \gamma'' \in w \Phi^+$.

Suppose $j > 1$. We have $(\beta_{j-1},\beta_j)=1$ by the definition of a simple cluster, so $(\beta_{j-1},\gamma_j)=-1$.
We can write
$(\beta_{j-1},\gamma'+\gamma'')=-1$ and $\{(\beta_{j-1},\gamma'), (\beta_{j-1},\gamma'')\}=\{0.-1\}$.
Without loss of generality, suppose $(\beta_{j-1},\gamma')=-1$ and $(\beta_{j-1},\gamma'')=0$.
By Lemma \ref{sumexists}, $\beta_{j-1}+\gamma' \in \Phi^+$. 
We cannot have $\beta_{j-1}+\gamma' \in w \Phi^-$ since 
$\beta_{j-1} \prec \beta_{j-1}+\gamma' \prec \beta_j$, and $\beta_{j-1}$ and $\beta_j$ are two consecutive 
elements of the linearly $\prec$-ordered set 
$\Phi^+ \cap w \Phi^-$. So, $\beta_{j-1}+\gamma' \in w \Phi^+$.

Also, $\alpha_i^* (\beta_{j-1})=\alpha_i^*(\beta_j)=1$ by the definition of a simple cluster, 
so $\alpha_i^*(\gamma_j)=0$, and $\alpha_i^*(\gamma'')=0$ since $\gamma'' \prec \gamma_j$.
So, $\gamma'' \notin \Phi^+ \cap w \Phi^-$, hence $\gamma'' \in w \Phi^+$.
But then $\beta_j = \beta_{j-1}+\gamma' + \gamma'' \in w \Phi^+$, a contradiction.

So, all $\gamma_j$s are simple roots, and we can choose $i_1, \ldots, i_l$ so that $\alpha_{i_j}=\gamma_j$.
Then $\beta_j = \alpha_{i_1}+\ldots+\alpha_{i_j}$ for $1 \le j \le l$.
We have $\alpha_i^*(\beta_1)=1$ by the definition of a simple cluster, and $\alpha_{i_1}=\beta_1$, so $i_1=i$.
Also, by the definition of a simple cluster, we have $(\beta_j, \beta_k)=1$ for $j \ne k$, so 
by Lemma \ref{differencesformapath}, we have $(\alpha_{i_j}, \alpha_{i_{j+1}})=-1$ for $1 \le j \le l-1$
and $(\alpha_{i_j}, \alpha_{i_k})=0$ for $1 \le j,k \le l$, $j \le k-2$.
In terms of the (simply-laced) Dynkin diagram, these values of scalar product
mean exactly that $i_1, \ldots, i_l$ is a simple path in the diagram.
\end{proof}
Now we are ready to fully describe simple clusters of the form $(\Phi^+ \cap w \Phi^-, le_i)$.
We will need the following standard fact about roots systems: if $i_1, \ldots, i_l$ is a simple path in 
the Dynkin diagram, then $\alpha_{i_1}+\ldots+\alpha_{i_l} \in \Phi^+$.
\begin{proposition}\label{simpleclusterdescription}
Let $w \in W$, $l=\ell(w)$, and let $i \in \{1,\ldots, r\}$. The following conditions are equivalent.
\begin{enumerate}
\item\label{simpleclustercoefficient1} $c_{w,le_i}=1$.
\item\label{simpleclustercluster} $(\Phi^+ \cap w \Phi^-, le_i)$ is a simple cluster.
\item\label{simpleclusterpathcluster} There exists a simple path $i_1 = i, \ldots, i_l$ in the Dynkin diagram
such that $\Phi^+ \cap w \Phi^- = \{ \alpha_{i_1}, \alpha_{i_1} + \alpha_{i_2} , \ldots, \alpha_{i_1}+...+\alpha_{i_l}\}$.
\item\label{simpleclusterpathreduced} There exists a simple path $i_1 = i, \ldots, i_l$ in the Dynkin diagram
such that $w=\sigma_{i_1} \ldots \sigma_{i_l}$.
\end{enumerate}
Moreover, the paths in conditions \ref{simpleclusterpathcluster} and \ref{simpleclusterpathreduced} are actually unique and coincide.
Also, we have the following antisimple path in the Bruhat graph:
\begin{equation*}
1_w \xrightarrow{\alpha_{i_1}} \sigma_{i_1} \xrightarrow{\alpha_{i_1}+\alpha_{i_2}} \sigma_{i_1} \sigma_{i_2} 
\xrightarrow{\alpha_{i_1}+\alpha_{i_2}+\alpha_{i_3}} \ldots 
\xrightarrow{\alpha_{i_1}+...+\alpha_{i_l}} 
\sigma_{i_1} ... \sigma_{i_l} = w.
\end{equation*}
\end{proposition}
\begin{proof}
$\ref{simpleclustercoefficient1} \Leftrightarrow \ref{simpleclustercluster}$ follows from Theorem \ref{MainTheorem}.

$\ref{simpleclustercluster} \Rightarrow \ref{simpleclusterpathcluster}$.
All scalar products between different roots in a simple cluster equal 1 by definition, so 
$\Phi^+ \cap w \Phi^-$ is a linearly ordered set.
Then condition \ref{simpleclusterpathcluster} follows from Lemma \ref{simpleclusterpath}.

$\ref{simpleclusterpathcluster} \Rightarrow \ref{simpleclustercluster}$.
The only nontrivial check is that all scalar products between different roots in $\Phi^+ \cap w \Phi^-$ equal 1.
Let $1 \le j < k \le l$. We have 
\begin{multline*}
(\alpha_{i_1}+\ldots+\alpha_{i_j}, \alpha_{i_1}+\ldots+\alpha_{i_k}) = \\
(\alpha_{i_1}+\ldots+\alpha_{i_j}, \alpha_{i_1}+\ldots+\alpha_{i_j}) + (\alpha_{i_1}+\ldots+\alpha_{i_j}, \alpha_{i_{j+1}}+\ldots+\alpha_{i_k})=
2+(-1)=1.
\end{multline*}

$\ref{simpleclusterpathreduced} \Rightarrow \ref{simpleclusterpathcluster}$.
A direct computation shows that $\sigma_{i_1} \ldots \sigma_{i_{j-1}} (\alpha_{i_j}) = \alpha_{i_1} + \ldots + \alpha_{i_j}$.
By \cite[Lemma 2.2]{bgg}, $\Phi^+ \cap w \Phi^- = \{ \alpha_{i_1}, \alpha_{i_1} + \alpha_{i_2} , \ldots, \alpha_{i_1}+...+\alpha_{i_l}\}$.

$\ref{simpleclusterpathcluster} \Rightarrow \ref{simpleclusterpathreduced}$.
Set $w'=\sigma_{i_1} \ldots \sigma_{i_l}$.
By the above calculation, 
$\Phi^+ \cap w' \Phi^- = \{ \alpha_{i_1}, \alpha_{i_1} + \alpha_{i_2} , \ldots, \alpha_{i_1}+...+\alpha_{i_l}\} = 
\Phi^+ \cap w \Phi^-$.
By Lemma \ref{phiwphidetermines}, $w'=w$.

The uniqueness of the path in condition \ref{simpleclusterpathcluster} is trivial, and 
the uniqueness of the path in condition \ref{simpleclusterpathreduced}, 
as well as the equality of the paths in conditions \ref{simpleclusterpathreduced} and \ref{simpleclusterpathcluster}
follow from the calculations in the proof of $\ref{simpleclusterpathreduced} \Rightarrow \ref{simpleclusterpathcluster}$.

Finally, suppose 
condition \ref{simpleclusterpathreduced} holds.
Denote $\beta = \alpha_{i_1}+ \ldots +\alpha_{i_l}$. We already know that 
$\sigma_{i_1} \ldots \sigma_{i_{l-1}} (\alpha_{i_l})=\beta$. So,
$w(-\alpha_{i_l})=\sigma_{i_1} \ldots \sigma_{i_{l-1}} \sigma_{i_l} (-\alpha_{i_l})=\beta$, and 
$\sigma_\beta w \xrightarrow{\beta} w$ is an antisimple edge.
Also, $\sigma_\beta w = w \sigma_{w^{-1}\beta}=w \sigma_{-\alpha_{i_l}} = w \sigma_{i_l}=\sigma_{i_1} \ldots \sigma_{i_{l-1}}$.
From here, we can get the desired antisimple path by induction on $l$.
\end{proof}

Our next goal is to describe all pairs $(w \in W, n_\bullet \in \ZZ_{\ge 0}^r)$ 
such that $(\Phi^+ \cap w \Phi^-, n_\bullet)$ is simply clusterable. 
We start with a few lemmas.
\begin{lemma}\label{phiwphicontainedj}
Let $J \subseteq \Pi$, $u \in W_J$. 
Then $\Phi^+ \cap u \Phi^- \subseteq \Phi_J$.
\end{lemma}
\begin{proof}
Assume the contrary: assume we have $\alpha \in \Phi^+ \cap u \Phi^-$, but $\alpha \notin \Phi_J$.
This means that there exists $\alpha_i \in \Pi \setminus J$ such that $\alpha_i^* (\alpha) > 0$.
But the reflections corresponding to the roots from $\Phi_J$ cannot change the coordinates corresponding to the simple roots outside $J$, 
so $\alpha_i^* (u^{-1}\alpha) = \alpha_i^* (\alpha) > 0$. 
On the other hand, $\alpha \in \Phi^+ \cap u \Phi^-$ 
means
that 
$u^{-1}\alpha \in \Phi^-$, a contradiction.
\end{proof}

\begin{lemma}\label{smallsortinggivescomplement}
Let $v, w \in W$, $I \subseteq \Pi$.
Suppose that $w v^{-1} \in W_{\Pi \setminus I}$ and $\Phi^+ \cap v \Phi^-=R_I(\Phi^+ \cap v \Phi^-)$. 
Then
$(\Phi^+ \cap w \Phi^-) \setminus R_I(\Phi^+ \cap w \Phi^-) = \Phi^+ \cap w v^{-1} \Phi^-$.
\end{lemma}
\begin{proof}
Let $\alpha \in (\Phi^+ \cap w \Phi^-) \setminus R_I(\Phi^+ \cap w \Phi^-)$. Let us check that $(w v^{-1})^{-1} \alpha \in \Phi^-$.
Assume the contrary: $v w^{-1} \alpha \in \Phi^+$. We also have $w^{-1} \alpha \in \Phi^-$, 
so $v w^{-1} \alpha \in v \Phi^-$ and $v w^{-1} \alpha \in \Phi^+ \cap v \Phi^- = R_I(\Phi^+ \cap v \Phi^-)$.
By Lemma \ref{smallsortingpreservesclusters} (\ref{restrictionmoved}), 
$\alpha = w v^{-1} v w^{-1} \alpha \in R_I(\Phi^+ \cap w \Phi^-)$,
a contradiction.
So, $(w v^{-1})^{-1} \alpha \in \Phi^-$, and $\alpha \in \Phi^+ \cap w v^{-1} \Phi^-$.

Now let $\beta \in \Phi^+ \cap w v^{-1} \Phi^-$. Then $-v w^{-1} \beta \in \Phi^+ \cap v w^{-1} \Phi^-$.
We also have $v w^{-1} = (w v^{-1})^{-1} \in W_{\Pi \setminus I}$, so by Lemma \ref{phiwphicontainedj}, 
$-v w^{-1} \beta \in \Phi_{\Pi \setminus I}$.
Then $\supp (-v w^{-1} \beta) \cap I = \varnothing$, 
and $-v w^{-1} \beta \notin R_I(\Phi^+ \cap v \Phi^-) = \Phi^+ \cap v \Phi^-$.
But $-v w^{-1} \beta \in \Phi^+$, so $-v w^{-1} \beta \notin v \Phi^-$, 
$-v w^{-1} \beta \in v \Phi^+$.
Hence, $w^{-1} \beta \in \Phi^-$, and $\beta \in \Phi^+ \cap w \Phi^-$.
Recall that $\beta \in \Phi^+ \cap w v^{-1} \Phi^-$. 
By Lemma \ref{phiwphicontainedj} again (but applied to $w v^{-1}$ this time), 
$\beta \in \Phi_{\Pi \setminus I}$, 
$\supp \beta \cap I = \varnothing$, and 
$\beta \notin R_I (\Phi^+ \cap w \Phi^-)$.
\end{proof}
\begin{corollary}\label{smallsortinggivescomplementdown}
Let $v$, $w$, and $I$ be as in Lemma \ref{smallsortinggivescomplement}.
Let $I' \subseteq \Pi$ be such that $I \subseteq I'$ and $\Phi^+ \cap w \Phi^-=R_{I'} (\Phi^+ \cap w \Phi^-)$.
Then $R_{I' \setminus I, I}(\Phi^+ \cap w \Phi^-) = \Phi^+ \cap w v^{-1} \Phi^-$.
\end{corollary}
\begin{proof}
By the definition of restriction, 
$R_{I' \setminus I, I}(\Phi^+ \cap w \Phi^-) = R_{I'} (\Phi^+ \cap w \Phi^-) \setminus R_I (\Phi^+ \cap w \Phi^-)$.
The claim follows from Lemma \ref{smallsortinggivescomplement}.
\end{proof}
\begin{corollary}\label{smallsortinggivescomplementup}
Let $v$, $w$, and $I$ be as in Lemma \ref{smallsortinggivescomplement}.
Let $I' \subseteq \Pi$ be such that $I \subseteq I'$ and $\Phi^+ \cap w v^{-1} \Phi^-=R_{I' \setminus I} (\Phi^+ \cap w v^{-1} \Phi^-)$.
Then $\Phi^+ \cap w \Phi^-=R_{I'} (\Phi^+ \cap w \Phi^-)$ and
$R_{I' \setminus I, I}(\Phi^+ \cap w \Phi^-) = \Phi^+ \cap w v^{-1} \Phi^-$.
\end{corollary}
\begin{proof}
Let $\alpha \in \Phi^+ \cap w \Phi^-$. Let us check that $\supp \alpha \cap I' \ne \varnothing$.
If $\supp \alpha \cap I \ne \varnothing$, then $\supp \alpha \cap I' \ne \varnothing$ since $I \subseteq I'$.

If $\supp \alpha \cap I = \varnothing$, then by Lemma \ref{smallsortinggivescomplement},
$\alpha \in (\Phi^+ \cap w \Phi^-) \setminus R_I(\Phi^+ \cap w \Phi^-) = 
\Phi^+ \cap w v^{-1} \Phi^- = R_{I' \setminus I} (\Phi^+ \cap w v^{-1} \Phi^-)$.
So, $\supp \alpha \cap (I' \setminus I) \ne \varnothing$, and $\supp \alpha \cap I' \ne \varnothing$.

Therefore, $\Phi^+ \cap w \Phi^-=R_{I'} (\Phi^+ \cap w \Phi^-)$, and the rest of the claim follows from 
Corollary \ref{smallsortinggivescomplementdown}.
\end{proof}

Now we are ready to describe the pairs $(w \in W, n_\bullet \in \ZZ_{\ge 0}^r)$ 
such that $(\Phi^+ \cap w \Phi^-, n_\bullet)$ is simply clusterable.
\begin{definition}
A finite sequence of simple paths in the Dynkin diagram
\begin{equation*}
i_{1,1},..., i_{1,l_1};
\ldots;
i_{m,1},..., i_{m,l_m}
\end{equation*}
($1 \le i_{j,k} \le r$)
is called \emph{tiling} if
for each $j$ ($1 \le j < m$),
the beginning of the $j$th path does not occur anywhere in the $(j+1)$th, \ldots, the $m$th path:
$i_{j,1} \ne i_{j',k}$ if $j < j' \le m$, $1 \le k \le l_{j'}$.
The \emph{total length} of such a tiling sequence is $l_1+\ldots+l_m$.
\end{definition}
\begin{theorem}\label{simplyclusterizabledescription}
Let $G$ be a semisimple split algebraic group with simply laced Dynkin diagram.
Let $J \subseteq \Pi$ and let $w \in W_J$.
Let $\alpha_{i_1}, \ldots, \alpha_{i_m}$ be an arbitrary sequence consisting of some (not necessarily all) roots from $J$ without repetitions,
let $l_1, \ldots, l_m$ be positive integers, and let $n_\bullet = l_1 e_{i_1} + \ldots + l_m e_{i_m}$.
Then the following conditions are equivalent.
\begin{enumerate}
\item\label{simplyclusterizableforsequence}
The configuration 
$(\Phi^+ \cap w \Phi^-, n_\bullet)$ 
is simply clusterable, 
and the definition of a simply clusterable configuration holds for the ordering 
$\supp n_\bullet = \{ \alpha_{i_1}, \ldots, \alpha_{i_m} \}$.
\item\label{reducedexpression}
There exists a tiling sequence of simple paths in the Dynkin diagram of the form
\begin{equation*}
j_{1,1}=i_1,..., j_{1,l_1};
\ldots;
j_{m,1}=i_m,..., j_{m,l_m}
\end{equation*}
such that $w=\sigma_{j_{m,1}}...\sigma_{j_{m,l_m}} \ldots \sigma_{j_{1,1}}...\sigma_{j_{1,l_1}}$
and all vertices in the paths belong to $J$ 
(formally: $\alpha_{j_{k,p}} \in J$ for $1 \le k \le m$, $1 \le p \le l_j$).
\end{enumerate}
If these conditions hold, then the expression for $w$ in condition \ref{reducedexpression}
is reduced, and $c_{w,n_\bullet}=1$.
\end{theorem}
\begin{proof}
$\ref{simplyclusterizableforsequence} \Rightarrow \ref{reducedexpression}$. 
Induction on $m$. If $m=0$, then everything is clear: we have $\Phi^+ \cap w \Phi^- = \varnothing$ and $w = 1_W$.

Let $m > 0$.
By Lemma \ref{phiwphicontainedj}, $\Phi^+ \cap w \Phi^- \subseteq \Phi_J$.
By the definition of a simply clusterable configuration, 
$(R_{\{\alpha_{i_1}\}} (\Phi^+ \cap w \Phi^-), p_{\{\alpha_{i_1}\}}(n_\bullet))$ 
is a simple cluster, 
in particular, $|R_{\{\alpha_{i_1}\}} (\Phi^+ \cap w \Phi^-)|=|p_{\{\alpha_{i_1}\}}(n_\bullet)|$.
Denote $J'=J \setminus \{ \alpha_{i_1} \}$.
By Lemma \ref{firstsortsmall}, there exists $v \in W$ 
such that 
$w v^{-1} \in W_{J'}$ 
and $c_{v, p_{\{\alpha_{i_1}\}}(n_\bullet)} > 0$.

We can rewrite $c_{v, p_{\{\alpha_{i_1}\}}(n_\bullet)} > 0$ as $c_{v, l_1 e_{i_1}} > 0$.
By Proposition \ref{sortabilitycriterium} (\ref{excessiveconfiguration}), 
$(\Phi^+ \cap v \Phi^-, l_1 e_{i_1})$ is 
an 
excessive configuration, 
in particular, $R_{\{\alpha_{i_1}\}} (\Phi^+ \cap v \Phi^-) = \Phi^+ \cap v \Phi^-$.
By Lemma \ref{smallsortingpreservesclusters} (\ref{clusteriff}), 
$(\Phi^+ \cap v \Phi^-, l_1 e_{i_1})$ is a (clearly, simple) cluster. 
By Proposition \ref{simpleclusterdescription}, 
there exists a simple path $j_{1,1}=i_1, \ldots, j_{1,l_1}$ in the Dynkin diagram such that 
$v=\sigma_{j_{1,1}} \ldots \sigma_{j_{1,l_1}}$ and 
$\beta:= \alpha_{j_{1,1}} + \ldots + \alpha_{j_{1,l_1}} \in \Phi^+ \cap v \Phi^-$.
We have $w \in W_J$ and 
$w v^{-1} \in W_{J'}$, 
so $v \in W_J$. 
By Lemma \ref{phiwphicontainedj}, $\beta \in \Phi_J$, 
in other words, $\alpha_{j_{1, p}} \in J$ for $1 \le p \le l_1$.

By Lemma \ref{clusterizableinductive} 
($\ref{clusterizableinductiveclusterizable} \Rightarrow \ref{clusterizableinductivesingle}$ and the ``more precisely'' part), 
$(R_{\supp n_\bullet \setminus \{\alpha_{i_1}\}, \{ \alpha_{i_1} \}} (\Phi^+ \cap w \Phi^-), 
p_{\supp n_\bullet \setminus \{\alpha_{i_1}\}}(n_\bullet))$
is a simply clusterable configuration.
By the definition of a clusterable configuration, $R_{\supp n_\bullet} (\Phi^+ \cap w \Phi^-)=\Phi^+ \cap w \Phi^-$.
By Corollary \ref{smallsortinggivescomplementdown} (applied to $v$, $w$, $\{\alpha_{i_1}\}$, 
and $\supp n_\bullet$), 
$R_{\supp n_\bullet \setminus \{\alpha_{i_1}\}, \{ \alpha_{i_1} \}} (\Phi^+ \cap w \Phi^-) = \Phi^+ \cap w v^{-1} \Phi^-$.

Summarizing, 
$(\Phi^+ \cap w v^{-1} \Phi^-, l_2 e_{i_2} + \ldots + l_m e_{i_m})$
is a simply clusterable configuration 
and $w v^{-1} \in W_{J'}$. 
By the induction hypothesis, 
there exists a tiling sequence of simple paths in the Dynkin diagram of the form
$j_{2,1}=i_2,..., j_{2,l_2};
\ldots;
j_{m,1}=i_m,..., j_{m,l_m}$
such that $w v^{-1}=\sigma_{j_{m,1}}...\sigma_{j_{m,l_m}} \ldots \sigma_{j_{2,1}}...\sigma_{j_{2,l_2}}$
and all vertices in the paths belong to $J'$.
Since all these vertices belong to $J' = J \setminus \{ \alpha_{i_1} \}$, 
the sequence of paths 
$j_{1,1}=i_1, \ldots, j_{1,l_1};
j_{2,1},..., j_{2,l_2};
\ldots;
j_{m,1},..., j_{m,l_m}$
is also tiling.
And $w = w v^{-1} v = 
\sigma_{j_{m,1}}...\sigma_{j_{m,l_m}} \ldots \sigma_{j_{2,1}}...\sigma_{j_{2,l_2}} \sigma_{j_{1,1}} ... \sigma_{j_{1,l_1}}$, 
which finishes the proof of $\ref{simplyclusterizableforsequence} \Rightarrow \ref{reducedexpression}$.

$\ref{reducedexpression} \Rightarrow \ref{simplyclusterizableforsequence}$. 
We again do induction on $m$. Again, if $m=0$, everything is clear.

Let $m > 0$.
Denote $v=\sigma_{j_{1,1}} \ldots \sigma_{j_{1,l_1}}$.
By Proposition \ref{simpleclusterdescription}, 
$(\Phi^+ \cap v \Phi^-, l_1 e_{i_1})$ is a simple cluster, 
in particular, $R_{\{ \alpha_{i_1} \}} (\Phi^+ \cap v \Phi^-) = \Phi^+ \cap v \Phi^-$.
Denote $J' = J \setminus \{ \alpha_{i_1} \}$.
Since the sequence of paths is tiling, we have 
$w v^{-1} = \sigma_{j_{m,1}}...\sigma_{j_{m,l_m}} \ldots \sigma_{j_{2,1}}...\sigma_{j_{2,l_2}} \in W_{J'}$.
By Lemma \ref{smallsortingpreservesclusters} (\ref{clusteriff}), 
$(R_{\{ \alpha_{i_1} \}} (\Phi^+ \cap w \Phi^-), l_1 e_{i_1})$ is a (clearly, simple) cluster.

By the induction hypothesis, the configuration 
$(\Phi^+ \cap w v^{-1} \Phi^-, l_2 e_{i_2} + \ldots + l_m e_{i_m})$ satisfies
the definition of a simply clusterable configuration 
for the ordering
$\supp (l_2 e_{i_2} + \ldots + l_m e_{i_m}) = \{ \alpha_{i_2}, \ldots, \alpha_{i_m} \}$.
In particular, $R_{ \{ \alpha_{i_2}, \ldots, \alpha_{i_m} \} } (\Phi^+ \cap w v^{-1} \Phi^-) = \Phi^+ \cap w v^{-1} \Phi^-$.
By Corollary \ref{smallsortinggivescomplementup}
(again applied to $v$, $w$, $\{\alpha_{i_1}\}$, 
and $\supp n_\bullet$), 
we have $\Phi^+ \cap w \Phi^- = R_{\supp n_\bullet}(\Phi^+ \cap w \Phi^-)$ and
$\Phi^+ \cap w v^{-1} \Phi^- = R_{\supp n_\bullet \setminus \{ \alpha_{i_1} \}, \{ \alpha_{i_1} \}} (\Phi^+ \cap w \Phi^-)$.
Now, condition \ref{simplyclusterizableforsequence} follows from 
Lemma \ref{clusterizableinductive} (formally, 
$\ref{clusterizableinductivemultiple} \Rightarrow \ref{clusterizableinductiveclusterizable}$ and the ``more precisely'' part).

Finally, by Corollary \ref{clusterizableequal}, we have $\ell(w)=|n_\bullet|=l_1+\ldots+l_m$, so the expression
for $w$ in condition \ref{reducedexpression} is indeed reduced.
And $c_{w, n_\bullet}=1$ follows from Theorem \ref{MainTheorem}.
\end{proof}
This theorem looks mostly useful and motivated for $J= \Pi$, but an arbitrary $J \subseteq \Pi$ was necessary 
to make the induction work and to prove the tiling property.
\begin{corollary}\label{tilingsequencemultiplicityfree}
The maximal degree of a multiplicity-free monomial in Schubert divisors equals the maximal total length of a 
tiling sequence of simple paths in the Dynkin diagram.

If
\begin{equation*}
i_{1,1},..., i_{1,l_1};
\ldots;
i_{m,1},..., i_{m,l_m}
\end{equation*}
is a tiling sequence of paths (of total length $l_1+\ldots+l_m$), then 
$D_{i_{1,1}}^{l_1} \ldots D_{i_{m,1}}^{l_m}$
is a multiplicity-free monomial (of degree $l_1+\ldots+l_m$).
\end{corollary}
\begin{proof}
The first claim follows from Theorem \ref{MainTheorem}, Lemma \ref{excessiveclusterisableissimple}, 
and Theorem \ref{simplyclusterizabledescription}. The second claim follows from Theorem \ref{simplyclusterizabledescription}.
\end{proof}
Unfortunately, this corollary may not describe all multiplicity-free monomials. 
There may exist more multiplicity-free monomials arising from clusterable, but not 
simply clusterable configurations. However, their degrees are the same as the degrees of 
monomials in this corollary by Lemma \ref{excessiveclusterisableissimple}.
Still, the question of a direct characterization of all multiplicity-free monomials
(in other words, of the projection of 
all pairs $(w \in W, n_\bullet \in \ZZ_{\ge 0})$ such that $c_{w,n_\bullet}=1$, like in Theorem \ref{MainTheorem}, to the second coordinate)
remains open.

\section{Numerical estimates}\label{sectionnumerical}
In this section, we will find all possible total lengths of the tiling sequences of paths in the Dynkin diagram
and all possible
degrees of multiplicity-free monomials, depending on the type of $G$.
First, let us reduce the question about semisimple groups to the question about simple groups.
\begin{lemma}\label{answermanytypes}
Let $G$ be a semisimple split algebraic group, $G=(G_1 \times \ldots \times G_k)/H$,
where $G_i$s are simple split groups, and $H$ is a finite central subgroup of the product.
If $d_i$ is the maximal degree of a multiplicity-free product of Schubert divisors corresponding 
to $G_i$ ($1 \le i \le k$), then 
the degrees of all multiplicity-free products of Schubert divisors corresponding to $G$
are all integers between 0 and $d_1+\ldots+d_k$, inclusively.
\end{lemma}
\begin{proof}
If $B_i$ is a Borel subgroup of $G_i$ ($1 \le i \le k$), 
then $G/B \cong G_1/B_1 \times \ldots \times G_k/B_k$, 
and $\CH^*(G/B) \cong \CH^*(G_1/B_1) \otimes_{\ZZ} \ldots \otimes_{\ZZ} \CH^*(G_k/B_k)$.
Hence, a monomial in Schubert divisors in $\CH^*(G/B)$ is multiplicity-free 
if and only if it is a product of 
multiplicity-free monomials in Schubert divisors in all $\CH^*(G_i/B_i)$s.


For the relation between the maximal degree and all degrees, see Remark \ref{shortenpath}.
\end{proof}
This lemma is formulated for groups of arbitrary type, but our previous analysis works 
for groups with simply laced Dynkin diagrams only, so we will get our final answers for 
simple groups of type $A$, $D$, and $E$ only. 
We can combine them into numerical answers for non-simple groups 
using Lemma \ref{answermanytypes}, but only if the whole (disconnected) Dynkin
diagram is simply laced.
So, now we again assume that the Dynkin diagram is simply laced.
\begin{remark}\label{througha}
If there exists a simple path in the Dynkin diagram that passes through all vertices, 
then this path makes this Dynkin diagram of type $A_r$.
\end{remark}

\begin{lemma}\label{multipathalwaysabound}
The total length of a tiling sequence is always $\le r(r+1)/2$. 
If the Dynkin diagram (connected or not) is not of of type $A_r$, then the inequality is strict.
\end{lemma}
\begin{proof}
Let 
\begin{equation*}
i_{1,1},..., i_{1,l_1},
\ldots,
i_{m,1},..., i_{m,l_m}
\end{equation*}
be a tiling sequence. Its total length is 
$l_1+\ldots+l_m$.
By definition,
for each $j$, $1 < j\le k$, 
the vertices $i_{1,1},\ldots,i_{j-1,1}$
do not appear among 
$i_{j,1},\ldots,i_{j,l_j}$.
So, $l_j \le r-(j-1)$.
\kpar
Hence, 
\begin{equation*}
l_1+\ldots+l_m
\le
r+(r-1)+\ldots+r-m+1
\le 
r+(r-1)+\ldots+1=r(r+1)/2.
\end{equation*}

This inequality becomes the equality 
$l_1+\ldots+l_m
=r(r+1)/2$
only if $m=r$ and $l_j=r-(j-1)$ for all $j$ ($1\le j\le m$).
\kpar
In particular, if 
$l_1+\ldots+l_m
=r(r+1)/2$,
then $l_1=r$. By Remark \ref{througha}, this is possible only if the Dynkin diagram is of type $A_r$, 
otherwise 
$l_1+\ldots+l_m
<r(r+1)/2$.
\end{proof}

\begin{lemma}\label{answertypea}
If $G$ is a simple split 
group of type $A_r$, then 
all possible degrees of multiplicity free monomials $D_1^{n_1}\ldots D_r^{n_r}$
are all integers between $0$ and $r(r+1)/2$ (inclusively).

(Note that if $G$ is of type $A_r$, then $\dim (G/B) = r(r+1)/2$. So, this lemma says in other words that 
multiplicity free monomials in type $A_r$ exist in all graded components of $\CH^* (G/B)$.)
\end{lemma}
\begin{proof}
By Remark \ref{shortenpath}, it suffices to show that 
the maximal possible degree of a multiplicity free monomial is exactly $r(r+1)/2$.
By Corollary \ref{tilingsequencemultiplicityfree}, it is enough to show that 
the maximal possible total length of a tiling sequence of simple paths in the Dynkin diagram 
is exactly $r(r+1)/2$.
Finally, by Lemma \ref{multipathalwaysabound}, it is enough to construct a tiling sequence of total length $r(r+1)/2$.
And it is easy to construct 
such a tiling sequence:
\begin{equation*}
1,...,r;
2,...,r;
\ldots;
r.
\end{equation*}
It gives the following multiplicity-free monomial: $D_1^r D_2^{r-1}\ldots D_r$.
\end{proof}

\begin{proposition}\label{answertyped}
If $G$ is a simple split 
group of type $D_r$ $(r\ge 4)$, then 
all possible degrees of multiplicity free monomials $D_1^{n_1}\ldots D_r^{n_r}$
are all integers between $0$ and $r(r+1)/2 - 1$ (inclusively).
\end{proposition}
\begin{proof}
Like in the proof of Lemma \ref{answertypea}, 
by Remark \ref{shortenpath}, by Corollary \ref{tilingsequencemultiplicityfree}, 
and by Lemma \ref{multipathalwaysabound}, 
it suffices to find a tiling sequence of simple paths of total length $r(r+1)/2 - 1$.
And it is easy to construct such a sequence:
\begin{equation*}
r, r-2, r-3, ...,1;
r-1,r-2,...,1
r-2,...,1;
\ldots;
1.
\end{equation*}
The total length is indeed 
$(r-1)+(r-1)+(r-2)+\ldots+1=(r+\ldots+1)-1=r(r+1)/2-1$,
and we get the following multiplicity-free monomial: 
$D_r^{r-1}D_{r-1}^{r-1}D_{r-2}^{r-2}D_{r-3}^{r-3}\ldots D_1$.
\end{proof}

\begin{theorem}\label{answertypee}
If $G$ is a simple split 
group of type $E_r$ $(6\le r\le 8)$, then 
all possible degrees of multiplicity free monomials $D_1^{n_1}\ldots D_r^{n_r}$
are all integers between $0$ and $r(r+1)/2 - 2$ (inclusively).
\kpar
In other words, the maximal degree
\kpar
\kcomment{\begin{enumerate}}
\kcomment{\item[]}
for $E_6$ is 19,

\kcomment{\item[]}
for $E_7$ is 26,

\kcomment{\item[]}
for $E_8$ is 34.
\kcomment{\end{enumerate}}
\end{theorem}
\begin{proof}
Like in the two previous proofs, 
by Remark \ref{shortenpath} and by Corollary \ref{tilingsequencemultiplicityfree},
it suffices to show that 
the maximal possible total length of a tiling sequence is exactly $r(r+1)/2 - 2$.
(In this case, the estimate $\le r(r+1)/2 -1$ from Lemma \ref{multipathalwaysabound} turns out to be not enough.)

Suppose we have a nonempty tiling sequence
\begin{equation*}
i_{1,1},..., i_{1,l_1},
\ldots,
i_{m,1},..., i_{m,l_m}
\end{equation*}
(of total length $l_1+\ldots+l_m$).
By the definition of a tiling sequence, 
\begin{equation*}
i_{2,1},..., i_{2,l_2},
\ldots,
i_{m,1},..., i_{m,l_m}
\end{equation*}
is a (possibly empty) tiling sequence of simple paths, and the vertex $i_{1,1}$ never occurs in these paths. 
In other words, if we denote the original 
Dynkin diagram by $\Xi$, then these paths are in fact in the Dynkin diagram $\Xi\setminus \{i_{1,1}\}$.
The total length of this sequence is $l_2+\ldots+l_m$.
Let us consider 2 cases (recall that we enumerate the vertices as in \cite{bou}):

\emph{Case 1. $i_{1,1}=2$.}
Then $\Xi\setminus \{i_{1,1}\}$ is a diagram of type $A_{r-1}$.
By Lemma \ref{answertypea},
$l_2+\ldots+l_m \le (r-1)r/2$.
A direct observation of Dynkin diagrams of types $E_6$, $E_7$, and $E_8$
shows that 
the maximal length of a path in $\Xi$ starting at the vertex number 2 is always $r-2$, so 
$l_1\le r-2$, 
and 
$l_1+\ldots+l_m\le r-2+(r-1)r/2=r(r+1)/2-2$.

\textit{Case 2. $i_{1,1}\ne 2$.}
Then 
a direct observation of Dynkin diagrams of types $E_6$, $E_7$, and $E_8$
shows that 
$\Xi\setminus \{i_{1,1}\}$ is not of type $A_{r-1}$.
(More precisely, it can be 
either of types $D$ or $E$ if $i_{1,1}$ is 1 or $r$,
or not connected if $i_{1,1}\ne 1$ and $i_{1,1}\ne r$.)
By Lemma \ref{multipathalwaysabound}, 
$l_2+\ldots+l_m \le (r-1)r/2-1$.
\kpar
By Remark \ref{througha} applied to the path $i_{1,1},\ldots, i_{1,l_1}$,
we have 
$l_1\le r-1$.
Therefore, again
$l_1+\ldots+l_m \le r-1+(r-1)r/2-1=r(r+1)/2-2$.

Finally, it is easy to construct a 
tiling sequence
of total length $r(r+1)/2-2$:
\begin{equation*}
2,4,5,...,r;
1,3,4,5,...,r;
3,4,5,...,r;
\ldots;
r.
\end{equation*}
The total length is indeed 
$(r-2)+(r-1)+(r-2)+\ldots+1=(r+\ldots+1)-2=r(r+1)/2-2$, 
and we have the following multiplicity-free monomial:
$D_2^{r-2}D_1^{r-1}D_3^{r-2}D_4^{r-3}\ldots D_r$.
\end{proof}

%


\end{document}